\documentclass[10 pt]{amsart}
\usepackage[top=1.5in, bottom=1.8in, left=1.2in, right=1.2in]{geometry}

\usepackage[small,nohug,heads=vee]{diagrams}
\diagramstyle[labelstyle=\scriptstyle]
\usepackage{amsthm,amsfonts,amssymb,amsmath,amsxtra}
\usepackage[all]{xy}
\usepackage{mathrsfs}
\usepackage{setspace}
\usepackage[final]{pdfpages}
\newtheorem{thm}{Theorem}[section]
\newtheorem{prop}[thm]{Proposition}
\newtheorem{lemma}[thm]{Lemma}
\newtheorem{cor}[thm]{Corollary}

\theoremstyle{definition}
\newtheorem{definition}[thm]{Definition}
\newtheorem{example}[thm]{Example}

\newarrow {Congruent} {3}{3}{3}{3}{3} 
\theoremstyle{remark}
\newtheorem{remark}[thm]{Remark}

\numberwithin{equation}{subsection}
\usepackage{etoolbox}

\makeatletter
\def\subsection{\@startsection{subsection}{1}%
	\z@{.5\linespacing\@plus.7\linespacing}{-.5em}%
	{\normalfont\itshape}}
\patchcmd{\@sect}
{\@addpunct.}
{}
{}
{}
\makeatother

\newcommand{\R}{\mathbb{R}}  % The real numbers.

\newcommand{\C}{\ensuremath{\mathbb{C}}}
\newcommand{\Z}{\ensuremath{\mathbb{Z}}}

\newcommand{\D}{\ensuremath{\mathbb{D}}}

\newcommand{\A}{\ensuremath{\mathbb{A}}}

\newcommand{\F}{\ensuremath{\mathbb{F}}}
\newcommand{\Fpbar}{\ensuremath{\overline{\mathbb{F}}_p}}

\newcommand{\Q}{\ensuremath{\mathbb{Q}}}
\newcommand{\xbar}{\ensuremath{{\overline{x}}}}

\newcommand{\Ok}{\ensuremath{\mathcal{O}}}
\newcommand{\B}{\ensuremath{\mathcal{B}}}
\newcommand{\G}{\ensuremath{\mathcal{G}}}
\newcommand{\pdiv}{\ensuremath{\mathscr{G}}}

\newcommand{\Spec}{\ensuremath{\mbox{Spec}}}

\newcommand{\sa}{\ensuremath{s_{\alpha,0}}}
\newcommand{\set}{\ensuremath{s_{\alpha,\acute{e}t}}}
\newcommand{\s}{\ensuremath{\tilde{s}}}

\newcommand{\et}{\ensuremath{{\acute{e}t}}}

\newcommand{\Hom}{\ensuremath{\mbox{Hom}}}
\newcommand{\loc}{\ensuremath{M^{loc}_{\G}}}

\newcommand{\Mod}{\ensuremath{\mbox{Mod}_{\mathfrak{S}}^\varphi}}

\newcommand{\hW}{\ensuremath{\widehat{W}}}

\newcommand{\rmK}{\ensuremath{\mathrm{K}}}

\newcommand{\Adm}{\ensuremath{\text{Adm}}}
\newlength{\olen}
\newlength{\ulen}
\newlength{\xlen}
\newcommand{\xra}[2][]{%
	\ifbool{@display}%
	{\settowidth{\olen}{$\overset{#2}{\longrightarrow}$}%
		\settowidth{\ulen}{$\underset{#1}{\longrightarrow}$}%
		\settowidth{\xlen}{$\xrightarrow[#1]{#2}$}%
		\ifdimgreater{\olen}{\xlen}%
		{\underset{#1}{\overset{#2}{\longrightarrow}}}%
		{\ifdimgreater{\ulen}{\xlen}%
			{\underset{#1}{\overset{#2}{\longrightarrow}}}
			{\xrightarrow[#1]{#2}}}}%
	{\xrightarrow[#1]{#2}}
}
\makeatother
\newcommand{\xyra}[2][]{%
	\settowidth{\xlen}{$\xrightarrow[#1]{#2}$}%
	\ifbool{@display}%
	{\settowidth{\olen}{$\overset{#2}{\longrightarrow}$}%
		\settowidth{\ulen}{$\underset{#1}{\longrightarrow}$}%
		\ifdimgreater{\olen}{\xlen}%
		{\mathrel{\xymatrix@M=.12ex@C=3.2ex{\ar[r]^-{#2}_-{#1} &}}}%
		{\ifdimgreater{\ulen}{\xlen}%
			{\mathrel{\xymatrix@M=.12ex@C=3.2ex{\ar[r]^-{#2}_-{#1} &}}}
			{\mathrel{\xymatrix@M=.12ex@C=\the\xlen{\ar[r]^-{#2}_-{#1} &}}}}}%
	{\mathrel{\xymatrix@M=.12ex@C=\the\xlen{\ar[r]^-{#2}_-{#1} &}}}%
}
\makeatletter
\newcommand{\xla}[2][]{%
	\ifbool{@display}%
	{\settowidth{\olen}{$\overset{#2}{\longleftarrow}$}%
		\settowidth{\ulen}{$\underset{#1}{\longleftarrow}$}%
		\settowidth{\xlen}{$\xleftarrow[#1]{#2}$}%
		\ifdimgreater{\olen}{\xlen}%
		{\underset{#1}{\overset{#2}{\longleftarrow}}}%
		{\ifdimgreater{\ulen}{\xlen}%
			{\underset{#1}{\overset{#2}{\longleftarrow}}}
			{\xleftarrow[#1]{#2}}}}%
	{\xleftarrow[#1]{#2}}
}

\begin{document}
\title{Mod-$p$ isogeny classes on Shimura varieties with parahoric level structure}
\author{Rong Zhou}

\begin{abstract}
	We study the special fiber  of the integral models for Shimura varieties of Hodge type with parahoric level structure constructed by Kisin and Pappas in \cite{KP}. We show that when the group is residually split, the points in the mod $p$ isogeny classes have the form predicted by the Langlands-Rapoport conjecture in \cite{LR}. 
	
	We also verify most of the He-Rapoport axioms for these integral models without the residually split assumption. This allows us to prove that all Newton strata are non-empty for these models.
\end{abstract}
\address{School of Mathematics, Institute for Advanced Study, Princeton, NJ 08540}
\email{rzhou@ias.edu}
\maketitle

% ----------------------------------------------------------------

%%% Local Variables: 
%%% mode: latex
%%% TeX-master: "main"
%%% End: 

% ----------------------------------------------------------------

%%

%%

% ----------------------------------------------------------------

\tableofcontents

% ----------------------------------------------------------------

%%% Local Variables: 
%%% mode: latex
%%% TeX-master: "main"
%%% End: 

% ----------------------------------------------------------------

%%
%% Now edit the following to give your name and address:
%% 

\section{Introduction}An essential part of Langlands' philosophy is that the Hasse-Weil zeta function of an algebraic variety should be a product of automorphic $L$-functions. In \cite{La1}, \cite{La2}, \cite{La3}, Langlands outlined a program to verify this for the case of Shimura varieties, for which an essential ingredient was to obtain a description of the mod-$p$ points of a suitable integral model for the Shimura variety. 
Such a conjectural description first appeared in \cite{La1}, and was later refined by \cite{LR}, \cite{Ko1} and \cite{Ra}. To explain it, we first introduce some notations.

 Let $(G,X)$ be a Shimura datum and $\rmK_p\subset G(\Q_p)$ and $\rmK^p\subset G(\A_f^p)$ compact open subgroups where $\A_f^p$ are the finite adeles with trivial component at $p$. We assume $\rmK_p$ is a parahoric subgroup of $G(\Q_p)$. For $\rmK^p$ sufficiently small we have the Shimura variety $Sh_{\rmK_p\rmK^p}(G,X)$ which is an algebraic variety over a number field $E$, known as the reflex field. We will mostly be considering Shimura varieties of Hodge-type in which case $Sh_{\rmK_p\rmK^p}(G,X)$ can be thought of as a moduli space of abelian varieties equipped with some cycles in its Betti cohomology. Let $p$ be a prime and $v|p$ a prime of $E$, then conjecturally there should exist an integral $\mathscr{S}_{\rmK_p\rmK^p}(G,X)/\Ok_{E_{(v)}}$ for $Sh_{\rmK_p\rmK^p}(G,X)$ satisfying certain good properties. When the group is unramified and $\mathrm{K}_p$ is hyperspecial there is a characterization of such an integral model, however for general parahorics such a characterization is not known. However as long as the integral model has good local properties (more precisely, one desires that its nearby cycles are amenable to computation) and one can obtain some global information about the $\F_q$-rational points, then this is already enough for many applications such as the  computation of the (semisimple) local factor of the Hasse-Weil zeta function of the Shimura variety. 
 
 We consider also the inverse limit of integral models $$\mathscr{S}_{\rmK_p}(G,X):=\lim_{\leftarrow \mathrm{K}^p}\mathscr{S}_{\rmK_p\rmK^p}(G,X)$$
 
  Then conjecturally there should be a bijection (see \cite{LR}, \cite{Ra}, \cite{Ha1}):
 
$$\mathscr{S}_{\rmK_p}(G,X)(\Fpbar)\cong \coprod_{\phi}S(\phi)$$ where $$S(\phi)=\lim_{\leftarrow \rmK^p}I_\phi(\Q)\backslash X_p(\phi)\times X^p(\phi)/\mathrm{K}^p$$

When $\mathscr{S}_{\rmK_p\rmK^p}(G,X)$ arises as a moduli space of abelian varieties, this represents the decomposition of the special fiber into disjoint isogeny classes parametrised by $\phi$. The individual isogeny class $S(\phi)$ breaks up into a prime to $p$ part $X^p(\phi)$ and $p$-power part $X_p(\phi)$, and the $I_\phi(\Q)$ is the group of self-isogenies of  any member of the isogeny class $S(\phi)$. For general $G$, the objects appearing are approriate group theoretic analogues of the objects described. 

The bijection should satisfy compatibility conditions with respect to certain group actions on either side. For example, on $S(\phi)$ one can define an operator $\Phi$, and this should correspond under the above bijection to the action of Frobenius on $\mathscr{S}_{\rmK_p}(G,X)(\Fpbar)$. Using this, one obtains a completely group theoretic description of the $\F_q$ points of the Shimura variety.

The first major result in this direction was obtained by Kottwitz \cite{Ko1} who gave a description of the $\Fpbar$ points for PEL-type Shimura varieties (more precisely the moduli spaces he considered are actually a union of Shimura varieties, but for the application to computing the zeta function, this was not an issue). In this case,  the integral models of Shimura varieties are moduli spaces of abelian varieties with extra structure, so one ends up counting such abelian varieties. Then after constructing good integral models for Shimura varieties of abelian type in \cite{Ki2}, Kisin proved the conjecture for these integral models. In that case the integral models are no longer moduli spaces and many new ideas were needed. In both these works, the authors worked with hyperspecial level structure at $p$, in particular this meant the Shimura varieties had good reduction, i.e. the integral models $\mathscr{S}_{\rmK_p\rmK^p}(G,X)$ were smooth over $\mathcal{O}_{E_{(v)}}$. In constrast, when considering arbitrary parahoric level structure, the integral models will not in general be smooth and this presents many new difficulties in proving such a result. However, if one is to get a complete description of the zeta function of the Shimura variety, then knowledge of the places of bad reduction are stil needed. Moreover, understanding the cohomology of these spaces at places of bad reduction has many other important applications, such as the local Langlands correspondences, see \cite{HT}.

We assume now that $p>2$. Let $(G,X)$ be a Shimura datum of Hodge type such that $G_{\Q_p}$ is tamely ramified,  $p\nmid |\pi_1(G_{der})|$ and $\rmK_p$ is a connected parahoric\footnote{A connected parahoric is one which is equal to the Bruhat-Tits stablizer scheme} (we will refer to these assumptions as (*)). With these assumptions Kisin and Pappas have constructed good integral models $\mathscr{S}_{\rmK_p}(G,X)$ for the Shimura varieties associated to the above data. These integral models satisfy the correct local properties, in the sense that there exists a local model diagram as in \cite[\S 6]{Ha}. The main result of this paper is then the  following.

\begin{thm}
 Let $(G,X)$ be  Shimura datum of Hodge type as above. We assume $G_{\Q_p}$ is residually split at $p$. Then the isogeny classes in $\mathscr{S}_{\rmK_p}(G,X)(\Fpbar)$ have the form $$\lim_{\leftarrow \rmK^p}I_\phi(\Q)\backslash X_p(\phi)\times X^p(\phi)/\mathrm{K}^p$$

ii) Each isogeny class contains a point $x$ which lifts to a special point in $Sh_{\mathrm{K}_P}(G,X)$.

\end{thm}

Let us first explain what we mean by an isogeny class. We assume for simplicity that $\rmK_p=\G(\Z_p)$, where $\G$ is an Iwahori group scheme for the rest of the introduction. It follows from the construction of the integral models, that to  each $x\in\mathscr{S}_{\rmK_p}(G,X)(\Fpbar)$, one can associate an abelian variety $\mathcal{A}_x$ with $G$ structure. This means $\mathcal{A}_x$ is equipped with certain tensors in its \'etale and crystalline cohomologies, whose stabilizer subgroups are related to the group $G$. This leads to a natural notion of the isogeny class of $x$, which breaks up into a prime to $p$ part and $p$-power part. We then obtain a decomposition of the special fiber into disjoint isogeny classes as in the conjecture, and to prove the conjecture in full one needs therefore a description of points in an individual isogeny class, and then also an enumeration of the set of all isogeny classes. In this paper we focus on the first problem. The key ingredient needed for the enumeration of the set of all isogeny classes is part ii) of the above theorem, this allows one to relate the set of isogeny classes to some data on the generic fiber where one has a good description of the points. Note that part ii) of the Theorem has been announced in \cite{KMS}, here we provide a different proof more along the lines of \cite[\S 2]{Ki3}. To go from the above theorem to the conjecture in full requires some technical computations involving Galois cohomology, which the author intends to return to in a future work. The above then can really be thought of as the arithmetic heart of the conjecture of \cite{LR}.

Let us now give some details about the theorem and its proof.  The general strategy follows that of \cite{Ki3}, however there are many obstructions to adapting the proof over directly for the case of general parahorics. As was mentioned above, each isogeny class decomposes into a $p$-power part and a prime-to-$p$ part; describing the $p$-power part is the most difficult part of the problem.

To an $x$ as above we can associate an $X_p(\phi)$ which is a union of affine Deligne-Lusztig varieties (see \S5.2 for the precise definition).  By the construction of these integral models, one has a map $$\mathscr{S}_{\rmK_p}(G,X)\rightarrow \mathscr{S}_{\rmK_p'}(GSp(V),S^\pm)\otimes\Ok_{E_{(v)}}$$ where $\mathscr{S}_{\rmK_p'}(GSp(V),S^\pm)$ is an integral model for the Siegel Shimura variety $Sh_{\rmK_p'}(GSp(V),S^\pm)$, defined as a moduli space for abelian varieties with polarization and level structure. 
Using Dieudonn\'e theory it is possible to define a natural map $$\tilde{i}_x:X_p(\phi)\rightarrow \mathscr{S}_{\rmK_p'}(GSp,S^\pm)(\Fpbar)$$ one would like to show this lifts to a well-defined map $$i_x:X_p(\phi)\rightarrow \mathscr{S}_{\rmK_p}(G,X)(\Fpbar)$$ satisfying good properties, the image of the map will then be the $p$-power part of the isogeny class for $d$. This is carried out in two steps. One can show that $X_p(\phi)$ has a geometric structure as a closed subscheme of the Witt vector affine flag variety of \cite{Zhu} and \cite{BS}. In particular there is a notion of connected components for the $X_p(\phi).$

1) Show that if $i_x$ is defined at a point of $X_p(\phi)$ then it is defined on the whole connected component containing the point.

2) Show that  every connected component of $X_p(\phi)$ contains a point at which $i_x$ is well-defined by lifting isogenies to characteristic 0.

For part 1), one uses an argument involving deformations of $p$-divisible groups. The analogous argument in \cite{Ki3} uses Grothendieck-Messing theory, in our context this is not possible since the test rings one needs to deform to are no longer smooth. Hence we use a new argument using Zink's theory of displays.

To carry out  2), an essential part is to get a description of (or at least a bound on) the connected components of $X_p(\phi)$. Such a bound is obtained in \cite{HZ}. The bound obtained there is somewhat more complicated than the description for the case of hyperspecial level.  This necessitates an improvement in the argument for lifting isogenies to characteristic 0. The main new innovation here is that one can move about through different Levi-subgroups of $G$ using characteristic $0$ isogenies.

Note that the bound obtained in \cite{HZ}, is only good enough to carry out the argument for groups which are residually split at $p$. However part 1) of the argument goes through in the more general setting. This already allows us to deduce the following interesting corollary. 

By construction of $\mathscr{S}_{\rmK_p}(G,X)$, we have a well-defined map $$\delta:\mathscr{S}_{\rmK_p}(G,X)(\Fpbar)\rightarrow B(G,\mu)$$ called the Newton stratification. Here $\mu$ is the inverse of the Hodge cocharacter and $B(G,\mu)\subset B(G)$ consists of the set of  neutral acceptable $\sigma$-conjugacy classes as in \cite{RV}, it is the group theoretic analogue of the set of isomorphism classes of isocrystals satisfying Mazur's inequality. In the case $G=GSp$, the integral model is a moduli space for polarized abelian varieties and this map sends an abelian variety to the isomorphism class of the associated isocrystal. The following result can then be thought of  as a generalization of the classical Manin's problem, which asks whether a $p$-divisible with Newton slopes 0,1 symmetric about $\frac{1}{2}$ arises from an abelian variety up to isogeny.

\begin{thm}
Let $(G,X)$ and $\mathrm{K}_p$ satisfy the assumptions $(*)$. Then $\delta$ is surjective. 
\end{thm}

This is proved by verifying some of the He-Rapoport axioms for integral models of Shimura varieties in \cite{HR}. In work in progress of \cite{KMS}, the authors have shown surjectivity of this map for groups which are quasi-split at $p$ using a different method. There is an obstruction to their technique working for non quasi-split groups. In contrast, our proof works for non quasi-split groups. For this it is essential to be able to work Iwahori level. The key part is to prove non-emptiness of the minimal Kottwitz-Rapoport stratum at Iwahori level, this shows the surjectivity at Iwahori level. This allows one to deduce the surjectivity statement for all parahoric levels by using suitable comparision maps between models with different levels. However, one major input to the proof is the non-emptiness of the basic locus, which is proved in \cite{KMS}.

Let us give a brief outline of the paper. In section 2 we recall some preliminaries on Bruhat-Tits buildings and Iwahori Weyl groups associated to a $p$-adic group. In section 3 we recall the construction of local models of Shimura varietes in \cite{PZ} and prove certain results about their embeddings into Grassmannians. In section 4 we recall the construction in \cite{KP} of the universal $p$-divisible group over the completion of an $\Fpbar$-point of the Shimura variety. We construct in Proposition \ref{adapted} a specific lifting which will be needed in the lifting isogenies argument. Section 5 is the technical heart of the paper. We recall the bound on the connected components of affine Deligne-Lusztig varieties obtained in \cite{HZ} and show that for the basic case, enough isogenies lift to characteristic 0. In section 6 we put the results together to deduce the existence of the required map from $X_p(\phi)$ into the integral model when the level $\rmK_p$ is Iwahori, this is Proposition \ref{prop6.1}. In section 7 we deduce the existence of good maps between Shimura varieties of different level which allows us to use the case of Iwahori level to deduce the result for other parahorics. This also verifies one of the He-Rapoport axioms for these integral models. The rest of the axioms are verified in section 8 which allows us to deduce the non-emptiness of Newton strata. Finally in section 9 we prove part ii) of the main theorem.

{\it Acknowledgements:}  It is a pleasure to thank my advisor Mark Kisin for suggesting this problem to me and for his constant encouragement. I would also like to thank
George Boxer, Xuhua He, Erick
Knight, Chao Li, Tom Lovering, Anand Patel, Ananth Shankar, Xinwen Zhu and Yihang Zhu for useful discussions.

\section{Preliminaries}Let $p>2$ be a prime. Let $F$ be a $p$-adic field  with ring of integers $\Ok_F$ and residue field $\F_q$. Let $L$ be the completion of the maximal unramified extension $F^{ur}$ of $F$ and $\Ok_L$ its ring of integers. Fix an algebraic closure $\overline{F}$  of $F$ and let $\Gamma:=Gal(\overline{F}/F)$. We also write $I$ for the absolute Galois group of $L$ which can be identified with the inertia subgroup $Gal(\overline{F}/F^{ur})$ of $\Gamma$. We let $\sigma\in Gal(F^{ur}/F)$ denote the Frobenius automorphism which extends by continuity to an automorphism of $L$.

Let $G$ be a connected reductive group over $F$. We assume $G$ splits over a tamely ramified extension of $F$.   Let $\mathcal{B}(G,F)$ be the (extended) Bruhat-Tits building of $G(F)$. For any $x\in\mathcal{B}(G,F)$, there is a smooth affine group scheme $\G_x$ over $\Ok_F$ such that $\G_x(\Ok_F)$ can be identified with the stabililzer of $x$ in $G(F)$. The connected component $\G_x^\circ$ of $\G_x$ is the parahoric group scheme associated to $x$. We can also consider the corresponding objects over $L$. Then for $x\in\mathcal{B}(G,L)$, we have $\G^\circ_x(\mathcal{O}_L)=\G_x(\Ok_L)\cap\ker \kappa_G$ where $$\kappa_G:G(L)\rightarrow \pi_1(G)_{I}$$ is the Kottwitz homomorphism, cf. \cite[Prop. 3 and Remarks 4 and 11]{HaRa}. Thus if $x\in\B(G,F)$, $\G_x^\circ(\Ok_F)=\G_x(\Ok_F)\cap\ker\kappa_G$. We say a parahoric subgroup $\G_x^\circ$ is connected if $\G_x^\circ=\G_x$. When $G$ is unramified or semi-simple and simply connected, every parahoric is connected.

Let $S\subset G$ be a maximal $L$-split torus defined over $F$ and $T$ its centralizer. Since $G$ is quasi-split over $L$ by Steinberg's theorem, $T$ is a maximal torus of $G$. Let $\mathfrak{a}$ denote a $\sigma$-invariant alcove in the apartment $V$ associated to $S$.
The relative Weyl group $W_0$ and the Iwahori Weyl group are defined as $$W_0=N(L)/T(L)\ \ \ W=N(L)/\mathcal{T}_0(\mathcal{O}_L)$$
where $N$ is the normalizer of $T$ and $\mathcal{T}_0$ is the connected  Neron model for $T$. These are related by an exact sequence 
$$0\rightarrow X_*(T)_I\rightarrow W\rightarrow W_0\rightarrow 0$$
For an element $\lambda\in X_*(T)_I$ we write $t_\lambda$ for the corresponding element in $W$, such elements will be called translation elements.

Let $\mathbb{S}$ denote the simple reflections in the walls of $\mathfrak{a}$. We let $W_a$ denote the affine Weyl group, it is the subgroup of $W$ generated by the reflections in $\mathbb{S}$. $W_a$ has the structure of a coxeter group and  hence a notion of length and Bruhat order which extends to $W$ in the natural way. The Iwahori Weyl group and affine Weyl group are related via the following exact sequence.

$$0\rightarrow W_a\rightarrow W\rightarrow \pi_1(G)_I\rightarrow 0$$
The choice of $\mathfrak{a}$ induces a splitting of this exact sequence and $\pi_1(G)_I$ can be identified with the subgroup $\Omega\subset W$ consisting of length $0$ elements.

Now let $\{\mu\}$ be a geometric conjugacy class of homomorphisms of $\mathbb{G}_m$ into $G$. Let $\underline{\mu}$ denote the image in $X_*(T)_I$ of a dominant (with respect to some choice of Borel defined over $L$) representative of $\mu\in X_*(T)$ of $\{\mu\}.$ The $\mu$-admissible set (cf. \cite[\S 3]{Ra})is defined to be $$\Adm(\{\mu\})=\{w\in W|w\leq t_{x(\underline{\mu})} \text{ for some }x\in W_0\}$$
Note the admissible set has a unique minimal element denoted  $\tau_{\{\mu\}}$; it is the unique element of $\Adm(\{\mu\})\cap\Omega$.

Now let $K\subset \mathbb{S}$ be a $\sigma$-stable subset and let $W_K$ denote the subgroup of $W$ generated by $K$. If $W_K$ is finite, then the fixed points of $K$ determines a parahoric subgroup $\G$ which is defined over $\Ok_F$. We then set
$\Adm_{K}(\{\mu\})$ to be the image of $\Adm(\{\mu\})$ in $W_K\backslash W/W_K.$ This subset only depends on the parahoric $\G$ and not on the choice of alcove $\mathfrak{a}$. We sometimes write $\Adm^G_K(\{\mu\})$ if we want to specify the group $G$ we are working with.

We have the Iwahori decomposition. For $w\in W$, we write $\dot{w}$ for a lift of $w$ to $N(L)$. Then the map $w\mapsto \dot{w}$ induces a bijection:
$$W_K\backslash W / W_K \cong \G(\Ok_L)\backslash G(L)/\G(\Ok_L)$$

Finally we recall the definition  and some properties of $\sigma$-straight elements. The Frobenius $\sigma$ induces an action on $W$ and $W_a$ which preserves $\mathbb{S}$. 

\begin{definition}
We say an element $w$ is $\sigma$-straight if $nl(w)=l(w\sigma(w)\dots\sigma^{n-1}(w))$ for all $n\in \mathbb{N}$. 
\end{definition}
For $w,w'\in W$ and $s\in\mathbb{S}$ we write $w\sim_sw'$ if $w'=sw\sigma(s)$ and $l(w)=l(sw\sigma(s))$. We write $w\sim w'$ if these exists a sequence $w=w_1,\dots,w_n\in W$,  $s_1,\dots,s_{n-1}\in\mathbb{S}$ and $\tau\in\Omega$ such that $w_i\sim_{s_i} w_{i+1}$ for all $i$ and we have $\tau^{-1} w_n\sigma(\tau)=w'$. We have the following which is \cite[Theorem 3.9]{HeNie}.
\begin{thm}
Let $w, w'\in W$ be straight elements such that there exists $x\in W$ with $xw\sigma(x)=w'$. Then $w\sim w'$.
\end{thm}

We will also need the following property of the Iwahori double coset corresponding to straight elements.

\begin{thm}[\cite{He1} Proposition 4.5]\label{thm2} Let $w$ be $\sigma$-straight and $\mathcal{I}$ the Iwahori subgroup correspondig to $\mathfrak{a}$. Then for every $g\in \mathcal{I}(\Ok_L)\dot{w}\mathcal{I}(\Ok_L)$ there exists $i\in \mathcal{I}(\Ok_L)$ such that $i^{-1}g\sigma(i)=\dot{w}$.
	
	\end{thm}
\section{Local models of Shimura varieties}
\subsection{}In this section we recall the construction of the local models of Shimura varieties and prove certain results concerning their embeddings into Grassmannians. Let $F$ denote a finite unramified extension of $\mathbb{Q}_p$ and $L/F$ the completion of the maximal unramified extension of $F$. We write $k$ for the residue field of $\Ok_L$.

 We start with a triple $(G,\G,\{\mu\})$ where:

$\bullet$  $G$ is a connected reductive group over $F$ which splits over a tamely ramified extension of $F$.

$\bullet$ $\G$ is a connected parahoric group scheme   associated to a point $x\in\B(G,F)$

$\bullet\ \{\mu\}$ is a conjugacy class of minuscule geometric cocharacters of $G$.

We will only need the case when $G$ is quasi split, so for the rest of this section we will make this assumption. We also assume $\G$ is the parahoric group scheme associated to a subset $K\subset \mathbb{S}$.

Let $E$ be the field of definition of  the conjugacy class $\{\mu\}$. In \cite{PZ} there is a construction of a reductive group scheme $\underline{G}$ over $\Ok_F[u^\pm]:=\Ok_F[u,u^{-1}]$ which specializes to $G$ under the map $\Ok_F[u^\pm]\rightarrow F$ given by $u\rightarrow  p$. There is also the construction of a smooth affine group scheme $\underline{\mathcal{G}}$ over $\Ok_F[u]$ which specialises to $\mathcal{G}$ under the map $\Ok_F[u]\rightarrow \Ok_F$ given by $u\rightarrow  p$. Moreover the specialization of $\underline{\G}_{{k[[t]]}}$ of $\underline{\G}$ under the map  $\Ok_F[u]\rightarrow k[[t]]$ given by $u\mapsto t$ is a parahoric subgroup of $\underline{G}_{k((t))}:=\underline{G}\otimes_{\Ok_F[u^\pm]}k((t))$.

Using these groups, there is constructed the global affine Grassmanian $Gr_{\mathcal{G},X}$ over $X:=Spec(\Ok_F[u])$ which, under the base change $\Ok_F[u]\rightarrow F$ given by $u\mapsto  p$, can be identified with the affine Grassmanian $Gr_{G,F}$ for $G$.  Recall $Gr_{G,F}$ is the ind-scheme which represents the fpqc sheaf associated to the functor on $F$-algebras $R\mapsto G(R((t))/G(R[[t]])$ (the identification is given by $t=u- p$).

Let $P_{\mu^{-1}}$ be the parabolic corresponding to $\mu^{-1}$ (we use the convention that the parabolic $P_\nu$ defined by a cocharacter $\nu$ has Lie algebra consisting of the subspace of the Lie algebra of $G$ where $\nu$ acts by weights $\geq 0$). The homogeneous space $G_{\overline{\Q}_p}/P_{\mu^{-1}}$ has a canonical model $X_\mu$ defined over $E$. We may consider $\mu$ as a $\overline{\Q}_p((t))$-point $\mu(t)$ of $G$ which gives a $\overline{\Q}_p$ point of $Gr_{G,F}$. As $\mu$ is minuscule, the action $G(\overline{\Q}_p[[t]])$ on $\mu(t)$ factors through $G(\overline{\Q}_p[t]])\rightarrow G(\overline{\Q}_p)$ and the image of the stabilizer of $\mu(t)$ in $G(\overline{\Q}_p)$ is equal to $P_{\mu^{-1}}$. Thus the $G(\overline{\Q}_p[[t]])$-orbit of $\mu(t)$ in $Gr_{G,F}$ can be $G_E$-equivariantly identified with $X_\mu$.
  \begin{definition} The local model $M^{loc}_{\G,\mu}$ is defined to be the Zariski closure of $X_\mu$ in $Gr_{\G,X}\times _X Spec(\Ok_E)$, where the specialization map $X\rightarrow \Spec\Ok_E$ is given by $u\mapsto p$
\end{definition}
We will usually write $M^{loc}_{\G}$ for $M^{loc}_{\G,\mu}$ when it is clear what the  cocharacter $\mu$ is.
By its construction $M^{loc}_{\G}$ is a projective scheme over $\Ok_E$ admitting an action of $\G\times_{\Ok_F}\Ok_E$. The following is \cite[Theorem 8.1]{PZ}
\begin{thm} Suppose $p$ does not divide the order of $\pi_1(G_{der})$. Then the scheme $\loc$ is normal, the geometric special fibre is reduced and admits a stratification by locally closed smooth strata; the closure of each stratum is normal and Cohen-Macaulay.
\end{thm}

This theorem is proved by identifying the geometric special fiber with an explicit subscheme of the partial affine flag variety $\mathcal{FL}_ {\underline{\G}_{k[[t]]}}$ for $\underline{\G}_{{k[[t]]}}$. This is an ind-scheme which represents the fpqc sheaf associated to the functor $R\mapsto \underline{G}_{k((t))}(R((t)))/\underline{\G}_{k[[t]]}(R[[t]])$ for a $k$-algebra $R$. We have an identification $$\mathcal{FL}_ {\underline{\G}_{k[[t]]}}(k)\cong \underline{G}_{k((t))}(k((t)))/\underline{\G}_{k[[t]]}(k[[t]])$$
By \cite[\S 3.a.1]{PZ}, there is an identification of Iwahori Weyl groups for $G$ and $\underline{G}_{k((t))}$. Thus for $w\in W_K\backslash W /W_K$ we obtain a point $\dot{w}_{k[[t]]}\in\mathcal{FL}_{\underline{\G}_{k[[t]]}}$, a lift of $w$ to $\underline{G}_{k((t))}(k((t)))$ and we let $C_w$ denote the $\underline{\G}_{k[[t]]}$ orbit of $\dot{w}_{k[[t]]}$ in $\mathcal{FL}_{\underline{\G}_{k[[t]]}}$ and $S_w$ its closure. $C_w$ and $S_w$ are respectively known as the Schubert cell and Schubert variety corresponding to $w$. Then by \cite[Theorem 8.3]{PZ} we have an identification $$M_{\G}^{loc}\otimes_{\Ok_E}k\cong \bigcup_{w\in \Adm_K(\{\mu\})}C_w$$

\begin{example}\label{GL}
Let $G=GL_n$ and we let $\mu$ is the cocharacter $a\mapsto\text{diag}(a^{(r)},1^{(n-r)})$. Let $e_1,\dots,e_n$ be the standard basis for $F^n$, and $\mathcal{GL}$ the Iwahori subgroup which is the stabilizer of the standard lattice chain $$\Lambda_0\supset\Lambda_1\supset \dots \supset\Lambda_{n-1}$$ where $$\Lambda_i:=\text{span}\langle  p e_1,\dots, p e_i,e_{i+1},\dots,e_n\rangle$$

The local model in this case agrees with that considered in \cite{RZ}, as mentioned in \cite[\S 6.b.1]{PZ}. In this case there is the following description. Given an $\Ok_F$-scheme $S$, we let $\mathbf{M}^{loc}_{\mathcal{GL}}(S)$ denote the set of isomorphism classes of commutative diagrams:

\[\xymatrix{ \Lambda_{0,S} & \Lambda_{1,S} \ar[l] & \dots\ar[l] & \Lambda_{n-1,S}\ar[l] \\ \mathcal{F}_0 \ar[u]&\mathcal{F}_1 \ar[l] \ar[u]  & \dots\ar[l] & \mathcal{F}_{n-1} \ar[l] \ar[u]}\]
where $\Lambda_{i,S}:=\Lambda_i\otimes_{\Ok_F}S$ and $\mathcal{F}_i$ is a locally free $\Ok_S$ module of rank $r$ and $\mathcal{F}_i\rightarrow\Lambda_{i,S}$ is an inclusion which locally on $S$ is a direct summand of $\Lambda_{i,S}$.
Let us explain how this description  is related to the $M^{loc}_{\mathcal{GL}}$ considered by \cite{PZ} and which was described in the last section. 

We use the following convention for a filtration defined by a cocharacter. Let $V$ be a finite dimensional vector space over $F$ or a finite free $\Ok_F$-module. Then a cocharacter $\mu:\mathbb{G}_m\rightarrow GL(V)$ induces a grading $V=\bigoplus_{i\in \Z}V_i$ where $\mathbb{G}_m$ acts on $V_i$ by the character $z\mapsto z^i$. It induces the filtration on $V$ given by $F^i:=\bigoplus_{j\geq i}V_j$. The stabilizer of this filtration is given by the parabolic subgroup $P_{\mu}\subset GL_n$ associated to $\mu$.

From the description of $\mathbf{M}^{loc}_{\mathcal{GL}}$ above, its generic fiber can be identified with the homogeneous space $GL_n/P_{\mu}$. Indeed when $ p$ is inverted, all the $\Lambda_i$ coincide and the choice of $\mathcal{F}_0\subset \Lambda_{0,F}$ determines the other $\mathcal{F}_i$. 

This implies that in the  construction of the local model $M^{loc}_{\mathcal{GL}}$, we must therefore take the defining cocharacter to be $\mu^{-1}$. In this case the stabilizer of the point of $GL_n(\overline{\Q}_p((t)))/GL_n(\overline{\Q}_p[[t]])$ corresponding to $\mu^{-1}$ is $P_{\mu}$, hence the generic fiber of $M^{loc}_{\G,\mu^{-1}}$ is identified with $GL_n/P_\mu$ as above. %%The special fiber is the union of the Schubert cells in $GL_n(k((t)))/\underline{\mathcal{GL}}(k[[t]])$ corresponding to the $\mu^{-1}$-admissible set $\Adm(\{\mu^{-1}\})$. Here $\underline{\mathcal{GL}}(k[[t]])$ is the Iwahori subgroup of $GL_n(k((t)))$ which stabilizes the standard lattice chain $\Lambda'_0\supset\Lambda_1'\supset\dots\supset\Lambda'_{n-1}$ in $k((t))^n$, where $$\Lambda_i':=\text{span}\langle te_1,\dots,te_i,e_{i+1},\dots,e_n\rangle$$ 

The identification of $\mathbf{M}^{loc}_{\mathcal{GL}}(k)$ and $M^{loc}_{\mathcal{GL}}(k)$ is given as follows. The group $\underline{\mathcal{GL}}_{{k[[t]]}}$ is the Iwahori subgroup in $GL_n(k((t))$ stabilizing the standard lattice chain $$\Lambda'_0\supset\dots\supset\Lambda_{n-1}'$$ in $k((t))^n$ where $\Lambda'_i:=\text{span}\langle t e_1,\dots,t e_i,e_{i+1},\dots,e_n\rangle$. We may identify the special fibers of $\Lambda_i$ and $\Lambda_i'$ by the choice of standard basis. Given a point of $\mathbf{M}^{loc}_{\mathcal{GL}}(k)$, we obtain a subspace $F^i\subset \Lambda_i'\otimes k$ via the above identification, where $F^i$ is of dimension $r$. The preimages  $\mathcal{L}_i$ of $F^i$ in $\Lambda'_i$ corresponds to a lattice chain of the same type as $\Lambda_0'\supset\dots\supset\Lambda_{n-1}'$. As in \cite[3.1.3]{Go1}, there exists an element $g\in GL_n(k((t)))/\underline{\mathcal{GL}}(k[[t]])$ such that $\mathcal{L}_i=g\Lambda_i'$ for all $i$. The corresponding point of $M^{loc}_{\mathcal{GL}}(k)$ is given by $xt^{-1}$ (here we consider $k((t))^\times\subset GL_n(k((t)))$ via the scalar matrices).

The same consideration apply when considering parahorics of $GL_n$ which contain the standard Iwahori. These parahorics are the stabilizers of subchains of the standard lattice chain considered in the example.

\end{example}

\subsection{}Let $\rho:G\rightarrow GL_{2n}$ be a closed group scheme immersion over $F$ such that $\rho\circ\mu$ is in the conjugacy class of the minuscule cocharacter $a\mapsto\text{diag}(1^{(n)},(a^{-1})^{(n)})$. Suppose also that $\rho$ satisfies the following conditions:

$\bullet$ $\rho$ extends to a closed group scheme immersion $\underline{\G}\rightarrow GL(W_{\bullet})$, where $W_{\bullet}$ is a lattice chain in $\Ok_F[u]^{2n}$ as in \cite[6.b.1]{PZ}.

$\bullet$ The Zariski closure of $\underline{G}_{k((u))}$ in $GL(W_{\bullet}\otimes_{\Ok_F[u]}k[[u]])$ is a smooth group scheme $\mathcal{P}'$ whose identity component can be idenitified with $\underline{\G}_{k[[u]]}$.

Then it is shown in \cite[Proposition 7.1]{PZ} that extending torsors along $\underline{\G}\rightarrow GL(W_\bullet)$ induces a closed immersion:

$$\iota:M^{loc}_{\G,\mu}\rightarrow M^{loc}_{\mathcal{GL},\rho\circ\mu}\otimes_{\mathcal{O}_F}\mathcal{O}_E$$
where $\mathcal{GL}$ is the parahoric subgroup of $GL_{2n}$ corresponding to the lattice chain $W_{\bullet}\otimes_{\Z_p[u]}\mathcal{O}_F$.  We will need a more explicit description of this map on the level of $k$ points which we now explain.

\subsection{}Let $\rho:G\rightarrow GSp(V)$ be a local Hodge embedding in the sense of \cite[\S2.3]{KP}, in particular we assume $G$ contains the scalars in $GL(V)$. As explained in \cite[2.3.7]{KP}, there is an embedding $\underline{\G}\rightarrow GL(W_{\bullet})$ satisfying the above conditions and hence an embedding of local models. Base changing to $\mathcal{O}_L[u^{\pm}]$ we obtain an embedding $\underline{G}\rightarrow GL(\underline{\Lambda})$ where $\underline{\Lambda}$ is a free module over $\mathcal{O}_L[u^\pm]$ and is the common generic fibre of $W_{\bullet}$. The fiber over $L$ of this embeddding is given by $\rho$ and we denote the fibers over $\kappa((u))$, where $\kappa=k,L$, by $\rho_{\kappa((u))}$. As shown in \cite[\S 1.2]{KP}, these maps induce embeddings

$$\mathcal{B}(G,L)\rightarrow \mathcal{B}(GL_{2n},L)$$ $$\mathcal{B}(G,\kappa((u)))\rightarrow\mathcal{B}(GL_{2n},\kappa((u)))$$
 These embeddings satisfy the following property: There is a choice of a maximal $\Ok_L[u^\pm]$ split torus $\underline{S}$ of $\underline{G}$ and a choice of basis $\underline{b}$ for $\underline{\Lambda}$ such that the above embeddings of buildings induce embeddings of the corresponding apartments
\begin{equation}\mathcal{A}(G,S,L)\rightarrow\mathcal{A}(GL_{2n},S',L)
\end{equation}
\begin{equation}
 \mathcal{A}(\underline{G}_{\kappa((u))},\underline{S}_{\kappa((u))},\kappa((u)))\rightarrow\mathcal{A}(GL_{2n},S',\kappa((u)))\end{equation} The choice of $\underline{b}$ determines an isomorphism of $GL(\underline{\Lambda})$ with $GL_{2n}$ and $S'$ is the diagonal torus of $GL_{2n}$. 

The choice of $\underline{S}$ and $\underline{b}$ also give identifications $$\mathcal{A}(G,S,L)\cong \mathcal{A}(\underline{G}_{\kappa((u))},\underline{S}_{\kappa((u))},\kappa((u)))$$ and $$\mathcal{A}(GL_{2n},S',L)\cong\mathcal{A}(GL_{2n},S',\kappa((u)))$$ Moreover there is an identification of Iwahori Weyl groups for the different groups over the fields $L$ and $\kappa((u))$ and the identification of apartments is compatible with the actions of these groups, see \cite[\S 3.a.1]{PZ}. The maps (3.3.1) and (3.3.2) are compatible with these identifications.

Let $x\in\mathcal{A}(G,S,L)$ be a point corresponding to the parahoric $\G$ over $\Ok_L$ ($\underline{S}$ can always be chosen in this way) and let $x_{\kappa((u))}\in\mathcal{A}(\underline{G}_{\kappa((u))},\underline{S}_{\kappa((u))},\kappa((u)))$ be the corresponding points. Then the group scheme $\underline{\G}_{\kappa[[u]]}$ is identified with the parahoric group scheme corresponding to $x_{\kappa((u))}$. The images of $x$ (resp. $x_{\kappa((u))}$)  under the above embeddings to the buildings for $GL_n$ give points $y$ (resp. $y_{\kappa((u))}$) whose corresponding parahoric $\mathcal{GL}$ (resp. $\underline{\mathcal{GL}}_{\kappa[[u]]}$) is the stabilizer of the base change of $W_{\bullet}$ to $L$ (resp. $\kappa((u))$).

\subsection{}The image of the alcove $\mathfrak{a}$ under the embedding of apartments $\mathcal{A}(G,S,L)\rightarrow \mathcal{A}(GL_{2n},S',L)$ is contained in (the closure of) an alcove in the apartment for $GL_{2n}$. We fix such an alcove and let $\mathbb{S}'$ denote the corresponding set of simple reflections. Then $\mathcal{GL}$ corresponds to a subset $K'\subset \mathbb{S}'$ and we may apply the constructions in \S2 to $\mathcal{GL}$ to obtain $\Adm^{GL_{2n}}_{K'}(\{\mu\}_{GL_{2n}})$, this is a subset of $W'_{K'}\backslash W' / W'_{K'}$ where $W'$ denotes the Iwahori subgroup for $GL_n$ and $\{\mu\}_{GL_{2n}}$ denotes the $GL_{2n}$-conjugacy class of cocharacters induced by $\{\mu\}$.

We identify $$M^{loc}_{\G}(k)\subset\underline{G}(k((u)))/\underline{\G}_{k[[t]]}(k[[u]])$$ with the union over the Schubert varieties $S_w$ where $w\in\Adm_{K}(\{\mu\})$. 
Similarly  $$M^{loc}_{\mathcal{GL}}(k)\subset GL_{2n}(k((u)))/\mathcal{GL}(k[[u]])$$ is the union of the Schubert varieties $S_{w'}^{GL_{2n}}$ in $GL_{2n}$ for $w'\in \Adm^{GL_{2n}}_{K'}(\{\mu\}_{GL_{2n}})$.

On the level of $k$ points the embeddings $M^{loc}_{\G}(k)\hookrightarrow M^{loc}_{\mathcal{GL}_{2n}}(k)$ is induced by the map $\underline{G}(k((u)))\rightarrow GL_{2n}(k((u)))$. On the other hand, the choice of basis $\underline{b}$ gives an embedding \begin{equation}\label{eq2} M^{loc}_{\mathcal{GL}}(k)\subset GL_{2n}(L)/\mathcal{GL}(\Ok_L)\end{equation} Indeed the choice of basis gives an identification between the special fibers of the lattice chains $W_\bullet\otimes k[[u]]$ and $W_{\bullet}\otimes\Ok_L$. Then as in Example \ref{GL} a $k$-point of $M_{\mathcal{GL}}^{loc}$ corresponds to a filtration on each of the $k$ vectors spaces $W_{\bullet}\otimes k$. If $g'\in GL_n(k((u)))/\mathcal{GL}(k((u)))$ lies in $M^{loc}_{\mathcal{GL}}(k)$, the filtration is induced by reducing the image of the lattice chain $ug'W_\bullet\otimes k$ modulo $u$. Taking the preimage of this filtration in $W_{\bullet}\otimes\Ok_L$, we obtain a lattice  chain of type $W_\bullet\otimes\Ok_L$ which is given by $p^{-1}gW_\bullet\otimes \Ok_L$ for some $g\in GL_{2n}(L)/GL_{2n}(\Ok_L)$. The embedding is then given by $g'\mapsto g$.  We may thus use (\ref{eq2}) to identify $M_{\G}^{loc}(k)$ with a subset of $GL_n(L)/GL_n(\Ok_L)$. 

Note that the embedding (\ref{eq2}) identifies $M_{\mathcal{GL}}^{loc}(k)$ with \begin{equation}\label{eq4}\bigcup_{w\in \Adm^{GL_{2n}}_{K'}(\{\mu\})}\mathcal{GL}(\Ok_L)\dot{w}\mathcal{GL}(\Ok_L)/\mathcal{GL}(\Ok_L)\end{equation}  where $\dot{w}$ denotes a representative of $w$ in $G(L)$ (see for example, \cite[\S11]{Ha1}). Then this identification is equivariant for the action of $\mathcal{GL}(k)$, where $\mathcal{GL}(k)$  acts on (\ref{eq4}) by left multiplication. Indeed since $\mu$ is minuscule, left multiplication by $\mathcal{GL}(\Ok_L)$ factors through $\mathcal{GL}(k)$.

\subsection{}Recall we have assumed $G\subset GL(V)$ contains the scalars. We let $\lambda:\mathbb{G}_m\rightarrow G$ denote the cocharacter giving scalar multiplication.

\begin{prop}\label{prop4}
Let $g\in G(L)$ with $$g\in \mathcal{G}(\Ok_L)\dot{w}\mathcal{G}(\Ok_L)$$ for some $w\in W_K\backslash W/W_K$. Then the image $\overline{\rho}(g)$ of $\rho(g)$ in $GL_{2n}(L)/\mathcal{GL}(\Ok_L)$ lies in $M^{loc}_\mathcal{G}(k)$ if and if and only if $w\in\Adm_{\G}(\{\mu\})$.
\end{prop}

\begin{proof}Let $g_1\dot{w}g_2$ in the Bruhat decomposition $\mathcal{G}(\Ok_L)\dot{w}\mathcal{G}(\Ok_L)$. Since $\G(\Ok_L)$ maps to $\mathcal{GL}(\Ok_L)$, we may assume $g=g_1\dot{w}$. 
	
	Now $M^{loc}_{\G}$ is equipped with an action by $\mathcal{G}\times_{\Ok_F}\Ok_E$ and $M^{loc}_{\mathcal{GL}}$ with an action of $\mathcal{GL}$. Over the special fiber this action is identified with the one given by left multiplication by $\G(\Ok_L)$ on $M_{\G}^{loc}(k)\subset GL_n(L)/\mathcal{GL}(\Ok_L)$, which as above, factors through $\G(k)$. Thus, modifying $g$ by $g_1$ on the left, we may assume $g=\dot{w}$.

Since the embedding of apartments (3.1) and (3.2) over $L$ and $k((u))$ is compatible with the identification of apartments  $$\mathcal{A}(G,S,L)\cong \mathcal{A}(\underline{G}_{k((u))},\underline{S}_{k((u))},k((u)))$$ respecting the action of Iwahori Weyl groups, we see that $\overline{\rho}(g)$ corresponds to the point $\rho_{k((u))}(\dot{w}_{k[[u]]})\in GL_{2n}(k((u)))/\mathcal{GL}(k[[u]])$. Thus by the description of $M^{loc}_{\G}(k)$ above, we see that $\overline{\rho}(g)\in M^{loc}_{\G}(k)$ if and only if $w$ lies in $\Adm_{\G}(\{\mu\})$.
\end{proof}

\begin{cor}\label{cor1}
Let $g\in\G(\Ok_L)\dot{w}\G(\Ok_L)$ with $w\in\Adm_K(\{\mu\})$, then $$\rho(g)\in \mathcal{GL}(\Ok_L)\dot{w}'\mathcal{GL}(\Ok_L)$$ for some  $\dot{w}'\in\Adm^{GL_{2n}}_{J\mathcal{GL}}(\{\mu\}_{GL_{2n}})$.
\end{cor}
\begin{proof} This follows from Proposition \ref{prop4} and the description  of $M^{loc}_{\mathcal{GL}}(k)$ as a subset of $GL_{2n}(L)/\mathcal{GL}(\Ok_L)$.
\end{proof}

\subsection{}As explained in \cite[2.3.15]{KP}, we may compose $\rho:G\rightarrow GSp(V,\Psi)$ with a diagonal embedding to obtain a new minuscule Hodge embedding $\rho':GSp(V',\Psi')$ with $\dim V'=2n'$ such that there is a self-dual lattice $V'_{\Z_p}\subset V'$ and the above embedding of buildings takes $x$ to the hyperspecial point $y\in\mathcal{B}(GL(V'),L)$ corresponding to $V'_{\Z_p}\otimes_{\Z_p}\Ok_L$. $V'_{\Z_p}$ is contructed by taking the direct sum of the lattices in the lattice chain corresponding to $\mathcal{GL}$. Then $\rho'$ factors through a diagonal embedding $GL(V)\rightarrow GL(V')$. In this case we obtain an embedding of local models:

$$M_{\G}^{loc}\rightarrow Gr(V_{\Z_p}')\otimes_{\Z_p}\Ok_E$$
where $Gr(V_{\Z_p}')$ is the smooth grassmannian parametrising dimension $n'$ sub-bundles $\mathcal{F}\subset V_{\Z_p}'\otimes_{\Z_p}S$ for any $\Z_p$-scheme $S$.

Choosing a basis $\underline{b}$ as above, we obtain an embedding $$M^{loc}_{\G}(k)\hookrightarrow GL_{2n'}(L)/\mathcal{GL}'(\Ok_L)$$ where $\mathcal{GL}'$ is the hyperspecial subgroup stabilising $V'_{\Z_p}$. Let $T'\subset GL'(V')$ denote a maximal torus whose apartment contains the hyperspecial vertex corresponding to $\mathcal{GL}'$.

\begin{cor}\label{cor2}
Let $g\in\G(\Ok_L)\dot{w}\G(\Ok_L)\subset G(L)$ with $w\in \Adm_{\G}(\{\mu\})$, then $$\rho(g)\in\mathcal{GL}'(\Ok_L)\mu_{GL}'(p)\mathcal{GL}'(\Ok_L)$$ where $\mu_{GL}'$ is a representative of $\{\mu\}_{GL(V')}$.
\end{cor}

\begin{proof}Under the diagonal  embedding $GL(V)\rightarrow GL(V')$, we have that $\Adm^{GL(V)}_{\mathcal{GL}}(\{\mu\}_{GL(V)})$ maps to $\Adm^{GL(V')}_{\mathcal{GL}'}(\{\mu\}_{GL(V')})$. This follows by the equality $\Adm(\{\mu\})=\mbox{Perm}(\{\mu\})$ for general linear groups, see \cite{HN1}. Since $\mathcal{GL}'$ is hyperspecial, $\Adm^{GL(V')}_{\mathcal{GL}'}(\{\mu\}_{GL(V')})$ is just the one coset corresponding to  $t_\mu$, hence the result follow from Corollary \ref{cor1}.
\end{proof}

\section{$p$-divisible groups}
In this section we review the theory of $\mathfrak{S}$-modules and their applications to deformation theory of $p$-divisible groups equipped with a collection of crystalline tensors. The main result is the construction of a certain deformation of such a $p$-divisible group in Proposition \ref{adapted} which is needed in the arguments of \S5.

\subsection{}We now let $F=\Q_p$ so that $L=W(\Fpbar)[\frac{1}{p}]$. For $K/L$ a finite totally ramified extensions, let $\Gamma_K$ be the absolute Galois group of $K$. Let $Rep_{cris}$ denote the category of crystalline $\Gamma_K$ representations, and $Rep_{cris}^\circ$ the category of $\Gamma_K$ representations in finite free $\Z_p$ modules which are lattices in some crystalline representation of $\Gamma_K$. For $V$ a crystalline representation of $\Gamma_K$, recall Fontaine's functors $D_{cris},D_{dR}$:

$$D_{cris}(V)=(V\otimes_{\Q_p}B_{cris})^{\Gamma_K}\ \ \ \ \ D_{dR}(V)=(V\otimes_{\Q_p}B_{dR})^{\Gamma_K}$$

Fix a uniformizer $\pi$ for $K$ and let $E(u)$ be the Eisenstein polynomial which is the minimal polynomial of $\pi$. Let $\mathfrak{S}=\Ok_L[[u]]$, we equip this with a lift $\varphi$ of Frobenius given by the usual Frobenius on $\Ok_L$ and $u\mapsto u^p$. We write $D^\times$ for the scheme $\Spec\mathfrak{S}$ with its closed point removed. Let $\Mod$ denote the category of finite free $\mathfrak{S}$ modules $\mathfrak{M}$ equipped with $\varphi$-linear isomorphism:

$$1\otimes\varphi:\mathfrak{M}\otimes_{\mathfrak{S},\varphi}\mathfrak{S}[1/E(u)]\rightarrow\mathfrak{M}[1/E(u)]$$

Let $BT^\varphi$ denote the subcategory of $\Mod$ consisting of $\mathfrak{M}$, such that $1\otimes\varphi$ maps $\varphi^*(\mathfrak{M})$ into $\mathfrak{M}$ and whose cokernel is killed by $E(u)$.

Given $\mathfrak{M}\in\Mod$ we equip $\varphi^*(\mathfrak{M})$ with the filtration:

$$Fil^i\varphi^*(\mathfrak{M})=(1\otimes\varphi)^{-1}(E(u)^i\mathfrak{M})\cap\varphi^*(\mathfrak{M}))$$

Let $\Ok_{\mathcal{E}}$ denote the $p$-adic completion of $\mathfrak{S}_{(p)}$; it is a discrete valuation ring with uniformiser $p$ and residue field $k((u))$ and let $\mathcal{E}$ denote its fraction field. We equip $\Ok_{\mathcal{E}}$ with the unique Frobenius $\varphi$ which extends that on $\mathfrak{S}$, and let $\mbox{Mod}_{\Ok_{\mathcal{E}}}^\varphi$ denote the category of finite free $\Ok_{\mathcal{E}}$-modules $M$ equipped with a Frobenius semilinear isomorphism:

$$1\otimes\varphi:\varphi^*(M)\rightarrow M$$

There is a functor $\Mod\rightarrow \mbox{Mod}_{\Ok_{\mathcal{E}}}^\varphi$ given by $$\mathfrak{M}\mapsto \mathfrak{M}\otimes_{\mathfrak{S}}\Ok_{\mathcal{E}}$$
 with the Frobenius on $\mathfrak{M}\otimes_{\mathfrak{S}}\Ok_{\mathcal{E}}$ induced by that on $\mathfrak{M}$.
 
 Let $\Ok_{\widehat{\mathcal{E}^{ur}}}$ denote the $p$-adic completion of a strict Henselization of $\Ok_{\mathcal{E}}$. We have the following which is contained in \cite[Theorem 1.1.2]{Ki3} and \cite[Theorem 3.3.2]{KP}:
 
 \begin{prop}\label{F-cryst}
There is a fully faithful functor $$\mathfrak{M}:Rep_{cris}^\circ\rightarrow \Mod$$
which is compatible with the formation of symmetric and exterior products and is such that $\Lambda\mapsto \mathfrak{M}(\Lambda)|_{D^\times}$ is exact. If $\Lambda$ is in $Rep^\circ_{cris}$, $V=\Lambda\otimes\Q_p$ and $\mathfrak{M}=\mathfrak{M}(\Lambda)$

 i) There are canonical isomorphisms 
 
 $$D_{cris}(V)\cong \mathfrak{M}/u\mathfrak{M}[\frac{1}{p}] \text{ and }  D_{dR}(V)\cong\varphi^*(\mathfrak{M})\otimes_{\mathfrak{M}}K$$
the first being compatible with $\varphi$ and the second being compatible with filtrations. 

 ii) There is a canonical isomorphism
 
 $$\Lambda\otimes_{\Z_p}\Ok_{\widehat{\mathscr{E}^{ur}}}\cong\mathfrak{M}\otimes_{\mathfrak{S}}\Ok_{\widehat{\mathscr{E}^{ur}}}$$

 \end{prop}

\subsection{}For an $R$-module $M$, we let $M^\otimes$ denote the direct sum of all $R$-modules obtained from $M$ by taking duals, tensor products, symmetric and exterior products.

Let $\Lambda\in Rep_{cris}^\circ$ and suppose $\set\in\Lambda^\otimes$ are a collection of $\Gamma_K$-invariant tensors whose stabilizer is a smooth group scheme $\G$ over $\Z_p$ with reductive generic fiber $G$. Since the $\set$ are $\Gamma_K$-invariant, we obtain a representation:

$$\rho:\Gamma_K\rightarrow \G(\Z_p)$$

We may think of each $\set$ as a morphism in $Rep_{cris}^\circ$ from the trivial representation $\Z_p$ to $\Lambda^\otimes$. Applying the functor $\mathfrak{M}$ to these morphisms gives us $\varphi$-invariant tensors $\tilde{s}_\alpha\in\mathfrak{M}(\Lambda)^\otimes$. 

\begin{prop}\label{sigma}
Suppose that the special fiber of $\G$ is connected and $H^1(\G,D^\times)=1$. Then there exists an isomorphism.
$$\Lambda\otimes_{\Z_p}\mathfrak{S}\cong \mathfrak{M}(\Lambda)$$
taking $\set$ to $\tilde{s}_\alpha$.
\end{prop}

\begin{proof}This is a special case of \cite[3.3.5]{KP}, indeed with our assumptions, $\G=\G^\circ$.
\end{proof}
 \subsection{}For  a $p$-divisible group $\pdiv$ over a scheme where $p$ is locally nilpotent we write $\D(\pdiv)$ for its contravariant Dieudonn\'e crystal. For $\pdiv$ a $p$-divisible group over $\Ok_K$, we let $T_p\pdiv$ be the Tate module of $\pdiv$ and $T_p\pdiv^\vee$ the linear dual of $T_p\pdiv$. We will apply the above to $\Lambda=T_p\pdiv^\vee$.

Let $R$ be a complete local ring with maximal ideal $\mathfrak{m}$ and residue field $k$. We let $W(R)$ denote the Witt vectors of $R$. Recall \cite{Zi1} we have a subring $\widehat{W}(R)=W(k)\oplus\mathbb{W}(\mathfrak{m})\subset W(R)$, where $\mathbb{W}(\mathfrak{m})\subset W(R)$ consists of Witt vectors $(w_i)_{i\geq1}$ with $w_i\in\mathfrak{m}$ and $w_i\rightarrow 0$ in the $\mathfrak{m}$-adic topology. Then $\hW(R)$ is preserved by the Frobenius $\varphi$ on $W(R)$ and we write $V$ for the Verschiebung. We have $I_R:=V\hW(R)$ is the kernel of the projection map $\hW(R)\rightarrow R$. Fix a uniformizer $\pi$ of $K$ and write $[\pi]\in \hW(\Ok_K)$ for its Teichmuller representative. Recall the following definition from \cite{Zi1}.
\begin{definition}A Dieudonn\'e display over $R$ is a tuple $(M,M_1,\Phi,\Phi_1)$ where

i) $M$ is a free $\hW(R)$ module.

ii) $M_1\subset M$ is a  $\hW(R)$ submodule such that $$I_RM\subset M_1\subset M$$
and $M/M_1$ is a projective $R$-module.

iii) $\Phi:M\rightarrow M$ is a $\varphi$ semi-linear map.

iv) $\Phi_1:M_1\rightarrow M$ is a $\varphi$ semi-linear map whose image generates $M$ as a $\widehat{W}(R)$ module and which satisfies

$$\Phi_1(V(w)m)=w\Phi(m), \text{ for } w\in\hW(R), m\in M$$

\end{definition}

Let $\pdiv$ be a $p$-divisible group over $R$. Then $\D(\pdiv)(\hW(R))$ naturally has the structure of a Dieudonn\'e display, and by the main result of \cite{Zi1} the functor $\pdiv\mapsto \D(\pdiv)(\hW(R))$ is an anti-equivalence of categories between $p$-divisible groups over $R$ and Dieudonn\'e displays over $R$.

\subsection{}If $\pdiv$ is a $p$-divisible group over $\Ok_K$, then by \cite[Theorem 3.3.2]{KP} there is a canonical isomorphism $$\D(\pdiv)(\hW(\Ok_K))\cong\mathfrak{M}\otimes_{\mathfrak{S},\varphi}\hW(\Ok_K)$$
where $\mathfrak{M}=\mathfrak{M}(T_p\pdiv^\vee)$ and the tensor product is over the map  given by composing the map $\mathfrak{S}\rightarrow \hW(\Ok_K), u\mapsto [\pi]$ with $\varphi$.  Moreover the induced map $$\D(\pdiv)(\Ok_K)\cong\varphi^*(\mathfrak{M})\otimes_{\mathfrak{S}}\Ok_K\rightarrow D_{dR}(T_p\pdiv^\vee\otimes_{\Z_p}\Q_p)$$ respects filtrations and we have a canonical identification
$$\D(\pdiv_0)(\Ok_L)\cong \varphi^*(\mathfrak{M}/u\mathfrak{M})$$
where $\pdiv_0:=\pdiv\otimes_{\Ok_K}k$.

If $\set\in T_p\pdiv^{\vee,\otimes}$ are a collection of $\Gamma_K$ invariant tensors, we let $$\sa\in D_{cris}(T_p\pdiv^\vee\otimes_{\Z_p}\Q_p)$$ denote the $\varphi$-invariant tensors corresponding to $\set$ under the $p$-adic comparison isomorphism. We assume from now on that the stabilizer of $\set$ is of the form $\G_x$ for $x\in\mathcal{B}(G,\Q_p)$, where $G$ is a tamely ramified reductive group containing no factors of type $E_8.$ The following is \cite[3.3.8]{KP}

\begin{prop}\label{prop3}
$\sa\in\D(\pdiv_0)(\Ok_L)^\otimes$ where we view $\D(\pdiv_0)(\Ok_L)^\otimes$ as an $\Ok_L$-submodule of $D_{cris}(T_p\pdiv^\vee\otimes_{\Z_p}\Q_p)^\otimes$. Moreover the $\sa$ lift to $\varphi$-invariant tensors $\tilde{s}_\alpha\in\D(\pdiv)(\hW(\Ok_K))^\otimes$ which map to $\mbox{Fil}^0\D(\pdiv)(\Ok_K)^\otimes$, and there exists an isomorphism:

$$\D(\pdiv)(\hW(\Ok_K))\cong T_p\pdiv^\vee\otimes_{\Z_p}\hW(\Ok_K)$$ taking $\tilde{s}_\alpha$ to $\set$. In particular, there is an isomorphism
$$\D(\pdiv_0)(\Ok_L)\cong T_p\pdiv^\vee\otimes_{\Z_p}\Ok_L$$ taking $s_{\alpha,0}$ to $s_{\alpha,\acute{e}t}$.
\end{prop}

\subsection{} Now let $\pdiv_0$ be a $p$-divisible group over $k$ and suppose $(\sa)\in\D^\otimes$, where $\D:=\D(\pdiv_0)(\Ok_L)$ are a collection of $\varphi$-invariant tensors  whose image in $\D(\pdiv_0)(k)$ lie in $Fil^0$. We  assume that the stabilizer $\G_{\Ok_L}$ of the $\sa$ is a connected Bruhat-Tits parahoric group scheme, i.e.  $\G_{\Ok_L}=\G_x=\G_x^{\circ}$ for some $x\in\B(G,L)$ as above and also that $G$ contains the scalars.

Let $P\subset GL(\D)$ be a parabolic subgroup lifting the parabolic $P_0$ corresponding to the filtration on $\D(\pdiv_0)(k)$. Write $M^{loc}=GL(\D)/P$ and $\mbox{Spf}A=\hat{M}^{loc}$ the completion at the identity. We write $\overline{M}_1$ for the universal filtration on $\D\otimes_{\Ok_L}A$. Let $y:A\rightarrow K'$ be a map such that $\sa\in Fil^0\D\otimes_{\Ok_L}{K'}$ for the filtration induced by $y$ on $\D\otimes_{\Ok_L}K'$. By \cite[Lemma 1.4.5]{Ki1}, the filtration corresponding to $y$ is induced by a $G$-valued cocharacter $\mu_y$.

Let $G.y$ be the orbit of $y$ in $M^{loc}\otimes_{\Ok_L}K'$ which is defined over the field of definition $E$ of the $G$-conjugacy class of  cocharacters $\{\mu_y\}$, and we write $M^{loc}_{\G_{\Ok_L}}$ for the closure of this orbit in $M^{loc}$. By \cite[Proposition 2.3.16]{KP}, $M^{loc}_{\G_{\Ok_L}}$ can be identified with the local model for $\G_{\Ok_L}$ and the conjugacy class of  cocharacters $\{\mu_y^{-1}\}$ considered in \S3. 

\begin{definition}\label{G-adapted}
Let $\pdiv$ be a $p$-divisible group over $\Ok_K$ whose special fiber is isomorphic to $\pdiv_0$. We say $\pdiv$ is $(\G_{\Ok_L},\mu_y)$-adapted if the tensors $\sa$ extend to tensors $\tilde{s}_\alpha\in\D(\pdiv)(\hW(\Ok_K))^\otimes$ such that the following two conditions hold:

1) There is isomorphism $\D(\pdiv)(\hW(\Ok_K))\cong \D\otimes_{\Ok_L}\hW(\Ok_K)$ taking $\tilde{s}_\alpha$ to $\sa$.

2) Under the canonical identification $\D(\pdiv)(\Ok_K)\otimes_{\Ok_K}K\cong \D\otimes_{\Ok_L}K$, the filtration on $\D\otimes_{\Ok_L}K$ is induced by a $G$-valued cocharacter conjugate to $\mu_y$. 
\end{definition}
\begin{remark}\label{remark}
It can be checked from the construction in \cite{KP}, that the notion of $(\G_{\Ok_L},\mu_y)$ adapted liftings only depends on the $G$ conjugacy class of $\mu_y$ and the specialization of the filtration induced by $\mu_y$. 
\end{remark}

\begin{prop}
Let $\mbox{Spf}A$ denote the versal deformation space of $\pdiv_0$. Then there is a versal quotient $A_{\G}$ of $A\otimes_{\Ok_L}\Ok_E$ such that for any $K$ as above, a map $\varpi:A\otimes_{\Ok_L}\Ok_E\rightarrow K$ factors through $A_{\G}$ if and only if the $p$-divisible group $\pdiv_{\varpi}$ induced is $(\G,\mu_y)$ adapted. 
\end{prop}

\begin{proof}This is \cite[Prop. 3.2.17]{KP}. Indeed it is clear from their construction that the $p$-divisible group $\pdiv_\varpi$ induced by  a map $\varpi:A_{\G}\rightarrow K$ is $(\G,\mu_y)$ adapted. Then Proposition 3.2.17 in loc. cit. provides the converse.
\end{proof}

 \subsection{}Now assume there is a $\Z_p$-module $U$ and an isomorphism $U\otimes_{\Z_p}\Ok_L\cong \D$ such that $\sa\in U^\otimes$. Then the stabilizer of $\sa$ in $U^\otimes$ is a group $\G$ over $\Z_p$ such that $\G\otimes_{\Z_p}\Ok_L\cong \G_{\Ok_L}$. We assume $\G$ is of the form $\G_x$ for some $x\in B(G,\Q_p)$. Since the $\sa$ are $\varphi$-invariant, we have $\varphi$ is of the form $b\sigma$ for some $b\in G(L)$.

Under these assumptions one can make the following construction of a certain $(\G_{\Ok_L},\mu_y)$-adapted  lift, which will be needed in \S5 for the reduction to Levi subgroups argument

\begin{prop}\label{adapted} There exists a $(\G_{\Ok_L},\mu_y)$-adapted deformation of $\pdiv$ such that $\sa\in\D^\otimes$ correspond to tensors $\set\in T_p\pdiv^\vee$ and such that there exists an isomorphism:

$$T_p\pdiv^\vee\otimes_{\Z_p}\Ok_L\cong\D$$
taking $\set$ to $\sa$.
\end{prop}
\begin{proof}

Let $\mathfrak{M}:=\D\otimes_{\sigma^{-1},\Ok_L}\mathfrak{S}$, then $\sigma^*(\mathfrak{M})\cong \D\otimes_W\mathfrak{S}$. Let $y^*(\overline{M}_1)\subset \D\otimes_{\Ok_L}\Ok_{K'}$ denote the filtration induced by $y:A\rightarrow \Ok_{K'}$ and let $F\subset \sigma^*(\mathfrak{M})$ denote be the premiage of $y^*(\overline{M}_1)$. By \cite[Lemma 3.2.6]{KP}, $F$ is a free $\mathfrak{S}$ module and  $s_{\alpha,0}\in F$, moreover the scheme $\underline{\text{Isom}}_{\sa,\mathfrak{S}}(F,\sigma^*(\mathfrak{M}))$ of $\mathfrak{S}$-isomorphisms which respect the $\sa$ is a $\G$ torsor. The Frobenius $\varphi$ on $\D$ induces a map
$$\D_1\xrightarrow{\simeq}\D\xrightarrow \cong\sigma^{-1*}\D$$
Here $\D_1$ is the preimage of the filtration on $\D(\pdiv_0)(k)$, the first arrow is given by $\sigma^{-1}(b/p)$ and the second isomorphism is induced by the identity on $U$. The specialization of $F$ at $u=0$ is identified with $\D_1$, then since $G$ contains the scalars, $\sigma^{-1}(b/p)$ preserves $\sa$ and hence corresponds to a point $\underline{\text{Isom}}_{\sa,\mathfrak{S}}(F,\sigma^*(\mathfrak{M}))(W)$. By smoothness of $\G$, this lifts to an isomorphism $$\Theta:F\xrightarrow\sim \sigma^*(\mathfrak{M})$$
respecting $\sa$.
Let $c=p\frac{E(u)}{E(0)}$. The morphism 
$$\varphi:\sigma^*(\mathfrak{M})\xrightarrow{\times c}F\xrightarrow\Theta\sigma^*(\mathfrak{M})\rightarrow \mathfrak{M}$$
where the last map is induced from the identity on $U$, gives $\mathfrak{M}$ the structure of an element of $BT^{\varphi}$ such that $\varphi^*(\mathfrak{M}/u\mathfrak{M})$ is identified with $\D$, and hence corresponds to a $p$-divisible group $\pdiv$ over $\Ok_K$ deforming $\pdiv_0$.

Since $\varphi$ preserves $\sa$, these give Frobenius invariant tensors in $\s_\alpha\in\D(\pdiv)(\hW(\Ok_K))^\otimes$, and by construction, we have an isomorphism $\D(\pdiv)(\hW(\Ok_K))\cong\D\otimes_{\Ok_L}\hW(\Ok_K)$ taking $\s_\alpha$ to $\sa$. Moreover under this isomorphism, the filtration on $\D\otimes_{\Ok_L}K$ is given by $\mu_y$. The natural isomorphism $$\D(\pdiv)(\Ok_K)\otimes_{\Ok_K}K\cong\D\otimes_{\Ok_L}K$$ takes $\s_\alpha$ to $\sa$ by \cite[Lemma 3.2.13]{KP}, hence the natural filtration on $\D\otimes_{\Ok_L} K$ is induced by a $G$ cocharacter conjugate to $\mu_y$. Thus $\pdiv$ is a $(\G,\mu_y)$ adapted deformation, and since $\tilde{s}_\alpha\in\mathfrak{M}^\otimes$, we have $\set\in T_p\pdiv^\vee$ by the fully faithfulness of $\mathfrak{M}$ in Proposition \ref{F-cryst}.

We now show that there exists an isomorphism $T_p\pdiv^\vee\otimes_{\Z_p}\Ok_L\cong \D$ respecting tensors. Let $\mathcal{P}\subset\underline{Isom}(T_p\pdiv^\vee\otimes_{\Z_p}\Ok_L,\D)$ be the isomorphism scheme taking $\set$ to $\sa$. By construction we have an isomorphism $\mathfrak{M}(T_p\pdiv^\vee)\cong\D\otimes_{\sigma^{-1},\Ok_L}\mathfrak{S}\xrightarrow \sim\D\otimes_w\mathfrak{S}$ taking $\set$ to $\sa$. By Proposition \ref{F-cryst} there is a canonical isomorphism 

$$T_p\pdiv^\vee\otimes_{\Z_p}\Ok_{\widehat{\mathcal{E}^{ur}}}\cong \mathfrak{M}(T_p\pdiv^\vee)\otimes_{\mathfrak{S}}\Ok_{\widehat{\mathcal{E}^{ur}}}$$ 
and this isomorphism takes $\set$ to $\s_{\alpha}$. Thus there is an isomorphism $T_p\pdiv^\vee\otimes_{\Z_p}\Ok_{\widehat{\mathcal{E}^{ur}}}\cong \D\otimes_{\Ok_L}\Ok_{\widehat{\mathcal{E}^{ur}}}$ taking $\set$ to $\sa$, i.e. $\mathcal{P}\otimes_{W}\Ok_{\widehat{\mathcal{E}^{ur}}}$ is a trivial $\G$ torsor. Since $\Ok_L\rightarrow \Ok_{\widehat{\mathcal{E}^{ur}}}$ is faithfully flat, $\mathcal{P}$ is a $\G$ torsor which is necessarily trivial since $\G$ is smooth and $\Ok_L$ is strictly henselian.

\end{proof}

\section{Affine Deligne Luzstig varieties}
This section forms the main part of the local argument for the description of the isogeny classes. It is used for the argument in \S6 of lifting isogenies to characteristic $0$. An essential part is a bound on the connected components of affine Deligne-Lusztig varieties obtained in \cite{HZ}, which is recalled here.

\subsection{}Let $G$  be a reductive group over $\Q_p$ which splits over a tamely ramified extension. Recall $S$ is a maximal $L$ split torus defined over $\Q_p$ and $T$ its centralizer. We have fixed a $\sigma$-invariant alcove in the apartment corresponding to $S$ which induces a length function and ordering on the affine Weyl group $W_a$ and hence on the Iwahori Weyl group $W$. We also fix a special vertex $\mathfrak{s}$ (not necessarily $\sigma$-invariant) which determines a Borel $B$ of $G$ defined over $L$.

Let $b\in G(L)$, we denote by $[b]=\{g^{-1}b\sigma(g)|g\in G(L)\}$  its $\sigma$-conjugacy class. Let $B(G)$ be the set of $\sigma$-conjugacy classes of $G(F)$. We let $\nu$ be the Newton map:

$$\nu:B(G)\rightarrow X_*(T)^+_{I,\Q}$$
where $X_*(T)^+_{I,\Q}$ is the intersection of $X_*(T)_{I}\otimes_{\Z}\Q$ with the dominant chamber determined by $B$. We let $$\kappa:B(G)\rightarrow \pi_1(G)_\Gamma$$ denote the map induced by composition of $\tilde{\kappa}:G(L)\rightarrow \pi_1(G)_I$ with the projection $\pi_1(G)_I\twoheadrightarrow\pi_1(G)_\Gamma$.

By \cite[\S4.13]{Ko1} the map $$(\nu,\kappa):B(G)\rightarrow X_*(T)^+_{I,\Q}\times\pi_1(G)_\Gamma$$ is injective.

\subsection{}Let $K\subset\mathbb{S}$ be a $\sigma$-invariant subset and $W_K$ the group generated by the reflection in $K$.  Let $\G$ denote the associated parahoric group scheme over $\Z_p$. For $b\in G(L)$ and $w\in W_K\backslash W/W_K$ we have the affine Deligne-Lusztig variety

$$X_{K,w}(b):=\{g\in G(L)/\G(\Ok_L)|g^{-1}b\sigma(g)\in \G(\Ok_L)\dot{w}\G(\Ok_L)\}$$ 
It is known that $X_{K,w}(b)$ has the structure of a perfect scheme over $k$, for example by \cite{BS} (see also \cite{Zhu}). When $K=\emptyset$ and $\G$ is an Iwahori subgroup,  we write $X_w(b)$ for the corresponding affine Deligne-Lusztig variety.

 Let $\{\mu\}$ be a geometric conjugacy class of cocharacters for $G$ and let $\underline{\mu}$ be the image in $X_*(T)_I$ of a dominant representative $\mu$ in $X_*(T)$. Recall we have associated to this data the $\mu$-admissible set $\Adm_K(\{\mu\})\subset W_K\backslash W/W_K$.
 
Let 
$$X(\{\mu\},b)_K:=\{g\in G(L)/\G(\Ok_L)|g^{-1}b\sigma(g)\in\bigcup_{w\in\Adm_{K}(\{\mu\})} \G(\Ok_L)\dot{w}\G(\Ok_L)\}$$
$$=\bigcup_{w\in\Adm_K(\{\mu\})}X_{K,w}(b)$$
As before, when $\G$ is the Iwahori subgroup we write $X(\{\mu\},b)$ for this union of affine Deligne-Lusztig varieties. For notational convenience will also consider the unions $$X(\sigma(\{\mu\}),b)_K:=\cup_{w\in\Adm_K(\{\mu\})}X_{K,\sigma(w)}(b)$$
It can be identified with the set $X(\{\sigma'(\mu)\},b)$ where $\sigma'\in Gal(\overline{\Q_p}/\Q_p)$ is a lift of Frobenius. The map $g\mapsto b\sigma(g)$ defines a bijection from $X(\{\mu\},b)$ to $X(\sigma(\{\mu\}),b)$.

We recall the definition of the neutral acceptable set $B(G,\{\mu\})$ in \cite{RV}. For $\lambda,\lambda'\in X_*(T)\otimes_{\Z}\Q$ be dominant, we write $\lambda\leq\lambda'$ if $\lambda'-\lambda$ is a non-negative rational linear combination of positive coroots. Set

$$B(G,\{\mu\})=\{[b]\in B(G):\kappa([b])=\mu^{\natural},\nu([b])\leq\overline{\mu}\}$$
where $\mu^\natural$ is the common image of $\mu\in\{\mu\}$ in $\pi_1(G)_\Gamma$,  and  $\overline{\mu}$ denotes the Galois average of $\underline{\mu}\in X_*(T)$ with respect to the $L$-action of $\sigma$ on $X_*(T)_{I,\Q}^+$.

\subsection{}The following result on the non-emptiness pattern of the $X(\{\mu\},b)_K$ was conjectured by Kottwitz and Rapoport in \cite{KR} and proved by He in \cite{He3}.

\begin{thm}
[\cite{He3}] 1) The set $X(\{\mu\},b)_K\neq\emptyset $ if and only if $[b]\in B(G,\{\mu\})$.

2) Let $K\subset K'$ with $K'$ also $\sigma$-invariant and let $\G'$ be the associated parahoric. The natural projection $G(L)/\G'(\Ok_L)\rightarrow G(L)/\G(\Ok_L)$ induces a surjection

$$X(\{\mu\},b)_{K'}\twoheadrightarrow X(\{\mu\},b)_K$$
\end{thm}

We will need the following two results which is proved in \cite{HZ}.

\begin{thm}\label{thm1}
i) Let $Y\subset X(\{\mu\},b)$ be a connected component, then $Y\cap X_w(b)\neq\emptyset$ for some $\sigma$-straight element $w\in W$.

ii) Assume $G_{ad}$ is $\Q_p$ simple. If $\mu$ is not central, then the Kottwitz homomorphism induces an isomorphism:

$$\pi_0(X(\{\mu\},\dot{\tau}_{\{\mu\}}))\xrightarrow \sim\pi_1(G)_{I}^\sigma$$ and if $\mu$ is central, $X(\{\mu\},b)$ is discrete and we have a bijection $$X(\{\mu\},b) \simeq G(\Q_p)/\G(\Z_p)$$
\end{thm}

Now assume $G$ is residually split. To any straight element $w$ one may associate a vector $\nu_w\in X_*(T)_{I,\Q}$, its non-dominant Newton vector. Explicitly, we let $n$ be sufficiently large such that  $w\sigma(w)\dots\sigma^{n-1}(w)=t_\lambda$ with $\lambda \in X_*(T)_I$ and we define $$\nu_w:=\frac{t_\lambda}{n}$$ We have an associated semistandard Levi subgroup $M_{\nu_w}$ which is generated by $T$ and the roots subgroups $U_a$ such that $\langle a,\nu_w\rangle=0$. The alcove $\mathfrak{a}$ determines an alcove for $\mathcal{B}(M,\Q_p)$ and hence an Iwahori subgroup $\mathcal{M}$. Explicitly, $\mathcal{M}(\Ok_L)=\G(\Ok_L)\cap M(L)$. This induces a Bruhat order and length function on $W_M$. Then it is known that $w$ lies in $W_M$ and is a length $0$ element, i.e. $\dot{w}\mathcal{M}(\Ok_L)\dot{w}^{-1}$, cf \cite[Theorem 1.3]{Nie}. Hence $\dot{w}$ is a basic element in $M(L)$. The following is \cite[Theorem 7.1]{HZ}.

\begin{thm}\label{bound} There is a map
$$\coprod_{w\in W, w \text{ a straight element with } \dot{w}\in[b]}X^{M_{\nu_w}}(\{\lambda_w\}_{M_{\nu_w}},\dot{w})\rightarrow X(\{\mu\},b)$$
which induces a surjection
$$\coprod_{w\in W, w \text{ a straight element with } \dot{w}\in[b]}\pi_0(X^{M_{\nu_w}}(\{\lambda_w\}_{M_{\nu_w}},\dot{w}))\rightarrow \pi_0(X(\{\mu\},b))$$
\end{thm}

Here $\{\lambda_w\}_{M_{\nu_w}}$ is a certain $M_{\nu_w}$ conjugacy class of cocharacters of $M_{\nu_w}$ which maps to $\{\mu\}$.

\subsection{}Now let $\pdiv_0$ be a $p$-divisible group over $k=\Fpbar$ and write $\D(\pdiv_0)$ for $\D(\pdiv_0)(\Ok_L)$. Let $\sa\in\D(\pdiv_0)^\otimes$ be a collection of $\varphi$-invariant tensors such that $\sa$ lie in $Fil^0\D(\pdiv_0)(k)$ and let $\G_{\Ok_L}\subset GL(\D(\pdiv_0))$ denote their stabilizer.
Assume that there is a free $\Z_p$-module $U$ together with an isomorphism $U\otimes_{\Z_p}W\cong \D(\pdiv_0)^\otimes$ such that $\sa\in U^\otimes $. Assume also that the stabilizer  $\G\subset GL(U)$ of these tensors is a connected  parahoric group scheme corresponding to $K\subset\mathbb{S}$. Then we have an isomorphism $\G\otimes_{\Z_p}\Ok_L\simeq \G_{\Ok_L}$ so that $\G_{\Ok_L}$ is also parahoric group scheme over $\Ok_L$. If $U'$ is another such $\Z_p$ module, the scheme of isomorphisms $U'\xrightarrow\sim U$ taking $\sa$ to $\sa$ is $\G$ torsor, which is necessarily trivial since $\G$ is smooth and has connected special fiber. Let $G$ denote the generic fibre of $\G$ which is reductive group over $\Q_p$.

Since the $\sa$ are $\varphi$ invariant, we can write $\varphi=b\sigma$ for some $b\in G(L)$, and the choice is independent of the choice of $U$ up to $\sigma$-conjugation by an element of $\G(\Ok_L)$. We pick a filtration on $\D(\pdiv_0)\otimes_{\Ok_L}K'$ lifting the one on $\D(\pdiv_0)(k)$ as in \S4.1 so that $s_{\alpha,0}\in Fil^0\D\otimes_{\Ok_L}K'$. This filtration is defined by a $G$-valued cocharacter $\mu_y$ and we have the embedding of local models $$M^{loc}_{\G}\hookrightarrow M^{loc}_{\mathcal{GL}}\otimes_{\Ok_L}\Ok_E$$ where the defining cocharacter for $M^{loc}_{\G}$ is $\mu_y^{-1}$.

The filtration on $\D(\pdiv_0)\otimes k=\D(\pdiv_0)(k)$ is by definition the kernel of $\varphi$, thus the preimage of the filtration in $\D(\pdiv_0)$ is given by $$\{v\in \D(\pdiv_0)|b\sigma(v)\in p\D(\pdiv_0)\}$$

This is precisely the sub $\Ok_L$ lattice in $\D(\pdiv_0)$ corresponding to $\sigma^{-1}(b^{-1})p\D(\pdiv_0)$. By Proposition \ref{prop4} we have 
$$\sigma^{-1}(b^{-1})\in \bigcup_{w\in\Adm_K(\{\mu_y^{-1}\})}\G(\Ok_L)\dot{w}\G(\Ok_L)$$
i.e., $1\in X(\{\sigma(\mu_y)\},b)_K$.

Let $K/L$ be a finite extension. By Proposition \ref{adapted}, there exists a $(\G_{\Ok_L},\mu_y)$ adapted lifting $\pdiv$ such that if $\set\in T_p\pdiv^\vee\otimes \Q_p$ denotes the tensors corresponding to $s_{\alpha,0}$ under the $p$-adic comparison isomorphism, then we have $s_{\alpha,\acute{e}t}\in T_p\pdiv^\vee$ and we have an isomorphism 

$$T_p\pdiv^\vee\otimes_{\Z_p}\mathfrak{S}\cong\mathfrak{M}(T_p\pdiv^\vee)$$
taking $\set $ to $\tilde{s}_\alpha$, which induces an isomorphism
$$T_p\pdiv^\vee\otimes_{\Z_p} W\simeq\D(\pdiv_0)$$ taking $\set$ to $\sa$. As in \S4, $\tilde{s}_{\alpha}$ denotes the functor $\mathfrak{M}$ applied to $s_{\alpha,\acute{e}t}$.

\subsection{}Now suppose $M\subset G$ is a closed reductive subgroup defined over $\Q_p$ such that $b\in M(L)$. Suppose that $M(L)\cap\G(\Ok_L)$ is the $\Ok_L$ points of a parahoric  subgroup $\mathcal{M}$ of $M$, then $\mathcal{M}$ is a defined over $\Z_p$.
Since $b\in M(L)$, we may extend the tensors $s_{\alpha,0}\in U^\otimes$ to $\varphi$-invariant tensors $t_{\beta,0}$ whose stabilizer is $\mathcal{M}$. We make the following assumption:

(*) The filtration on $\D(\pdiv_0)\otimes_{\Ok_L}k$ lifts to a filtration on $\D(\pdiv_0)\otimes_{\Ok_L}K$ which is induced by an $M$-valued cocharacter $\mu_y'$ which is conjugate $\mu_y$ in $G$.

We have the local model $M^{loc}_{\mathcal{M},\mu_{y'}^{-1}}$ which is defined over $\Ok_{E'}$ where $E'$ is the local reflex field for $\{\mu_y'\}$.  Then there is an embedding $$M^{loc}_{\mathcal{M},\mu_{y'}^{-1}}\hookrightarrow M^{loc}_{\mathcal{GL}}\otimes_{\Ok_F}\Ok_{E'}$$
which factors through a closed embedding $M^{loc}_{\mathcal{M},\mu_{y'}^{-1}}\hookrightarrow M^{loc}_{\mathcal{G},\mu_y^{-1}}\otimes\Ok_{E'}$. 

Now by Proposition \ref{G-adapted}, there exists a $(\mathcal{M},\mu_{y'})$-adapted lifting  $\pdiv$ of $\pdiv_0$ such that there is an isomorphism \begin{equation}\label{id}T_p\pdiv^\vee\otimes_{\Z_p}\mathfrak{S}\cong\mathfrak{M}(T_p\pdiv^\vee)\end{equation}
such that if $t_{\alpha,\acute{e}t}$ denotes the tensors corresponding to $t_{\alpha,0}$ under the $p$-adic comparison isomorphism, then $t_{\alpha,\acute{e}t}\in T_p\pdiv ^{\vee,\otimes}$ and the above map takes $t_{\alpha,\acute{e}t}$ to $\tilde{t}_\alpha$. In particular since $t_{\alpha,0}$ extends $\sa$, it takes $\set$ to $\sa$. Note that since $\mu_y'$ is $G$-conjugate to $\mu_y$, any $(\mathcal{M},\mu_{y'})$-lifitng is also a $(\G,\mu_y)$-adapted lifting of $\pdiv_0$. Fixing such an isomorphism as above, we may take $U$ to be $T_p\pdiv^\vee$. Since the notion of $(\G,\mu_y)$-adapted lifting only depends on the $G$-conjugay class of $\mu_y$ and its specialization, we may replace $\mu_y$ with $\mu_{y'}$ (see Remark \ref{remark}). We relabel this $\mu_y$, thus $\mu_y$ is an $M$-valued cocharacter inducing the filtration on $\D(\pdiv_0)\otimes K$.

\subsection{}Let $g\in G(\Q_p)$, then there is a finite extension $K'/K$ for which $g^{-1}T_p\pdiv$ is stable by $\Gamma_{K'}$ in $T_p\pdiv\otimes_{\Z_p}\Q_p$ hence corresponds to a $p$-divisible group $\pdiv'$ over $K'$, which is isogenous to $\pdiv$.  Let $\mathfrak{M}':=\mathfrak{M}(T_p\pdiv'^\vee)$ and $\mathfrak{M}:=\mathfrak{M}(T_p\pdiv^\vee)$, then the quasi-isogeny $\theta:\pdiv \rightarrow \pdiv'$ induces an identification $$\tilde{\theta}:\mathfrak{M}(T_p\pdiv'^\vee)[1/p]\xrightarrow\sim \mathfrak{M}(T_p\pdiv^\vee)[1/p]$$ so that $\mathfrak{M}'=\tilde{g}\mathfrak{M}$ for some $\tilde{g}\in GL(\mathfrak{M}[1/p])$.

\begin{prop}\label{key1}

i) $\tilde{g}$ can be taken to be in $G(\mathfrak{S}[1/p])$ under the above identification $$T_p\pdiv^\vee\otimes_{\Z_p}\mathfrak{S}
\xrightarrow\sim \mathfrak{M}$$

ii) We have $\tilde{g}\in\mathcal{M}(\Ok_{\widehat{\mathscr{E}^{ur}}}) g\mathcal{G}(\Ok_{\widehat{\mathscr{E}^{ur}}})$.
\end{prop}
\begin{proof}i) Since $\set$ are fixed by $G(\Q_p)$, we have $\set\in T_p\pdiv'^{\vee\otimes}$, and we have the stabiliser of $\set$ in $T_p\pdiv'^{\vee\otimes}$ is a parahoric subgroup of $G$. Thus by Proposition 3.2. there is an isomorphism:
$$T_p\pdiv'^\vee\otimes_{\Z_p}\mathfrak{S}\cong\mathfrak{M}'$$ taking $\set$ to $\tilde{s}$. Under the identification $ T_p\pdiv^\vee\otimes_{\Z_p}\mathfrak{S}\cong\mathfrak{M}$, we have $\tilde{g}$ is given by the composition:
\begin{equation}
T_p\pdiv^\vee\otimes_{\Z_p}\mathfrak{S}\xrightarrow{ g}T_p\pdiv'^\vee\otimes_{\Z_p}\mathfrak{S}\xrightarrow{s_\alpha}\mathfrak{M}'\xrightarrow{\tilde{\theta}}\mathfrak{M}[1/p]\xrightarrow{t_{\alpha}} T_p\pdiv^\vee\otimes_{\Z_p}\mathfrak{S}[1/p]\end{equation}
where the $\xrightarrow{s_\alpha}, \ \xrightarrow{t_\alpha}$ means that map preserves tensors of that type. Thus the composition preserves, $\s_\alpha$ and we have $\tilde{g}\in G(\mathfrak{S}[1/p])$.

ii) Over $\Ok_{\widehat{\mathscr{E}^{ur}}}$ there are canonical identifications:
$$T_p\pdiv^\vee\otimes_{\Z_p}\Ok_{\widehat{\mathscr{E}^{ur}}}\cong\mathfrak{M}\otimes_{\mathfrak{S}}\Ok_{\widehat{\mathscr{E}^{ur}}}$$
$$T_p\pdiv'^\vee\otimes_{\Z_p}\Ok_{\widehat{\mathscr{E}^{ur}}}\cong\mathfrak{M}'\otimes_{\mathfrak{S}}\Ok_{\widehat{\mathscr{E}^{ur}}}$$
the first one taking $t_{\alpha,\acute{e}t}$ to $ \tilde{t}_{\alpha}$ and the second taking $\set$ to $\tilde{s}_\alpha$. Thus, if we identify $T_p\pdiv^\vee$ with $T_p\pdiv'^\vee$ via $g$, these isomorphisms differ from the ones above by elements of $\mathcal{M}(\Ok_{\widehat{\mathscr{E}^{ur}}})$ and $\mathcal{G}(\Ok_{\widehat{\mathscr{E}^{ur}}})$ respectively. Since $\tilde{g}$ is identified with the map:

$$T_p\pdiv^\vee\otimes_{\Z_p}\Ok_{\widehat{\mathscr{E}^{ur}}}\xrightarrow{ g}T_p\pdiv'^\vee\otimes_{\Z_p}\Ok_{\widehat{\mathscr{E}^{ur}}}\xrightarrow{can}\mathfrak{M}'\otimes_{\mathfrak{S}}\Ok_{\widehat{\mathscr{E}^{ur}}}$$ $$\xrightarrow{\tilde{\theta}_{\widehat{\mathscr{E}^{ur}}}}\mathfrak{M}\otimes_{\mathfrak{S}[1/p]}{\widehat{\mathscr{E}^{ur}}}\xrightarrow{can} T_p\pdiv^\vee\otimes_{\Z_p}{\widehat{\mathscr{E}^{ur}}}$$
we have $\tilde{g}\in\mathcal{M}(\Ok_{\widehat{\mathscr{E}^{ur}}}) g\mathcal{G}(\Ok_{\widehat{\mathscr{E}^{ur}}})$.

\end{proof}

We will apply the above Proposition in the cases $M\subset G$ is a Levi subgroup or if we are in the situation of Proposition \ref{prop2}.

\subsection{}Using the canonical identification $\D(\pdiv_0)$ with $\varphi^*(\mathfrak{M}/u\mathfrak{M})$, we have $\tilde{\theta}^{-1}$ induces an isomorphism $\D(\pdiv_0)[1/p]\xrightarrow\sim\D(\pdiv'_0)[1/p]$. Then $\D(\pdiv_0')$ can be identified with $g_0\D(\pdiv_0)$ for $g_0=\sigma^{-1}(\tilde{g})|_{u=0}\in G(L)$. 

\begin{prop}\label{lifting}The association $g\mapsto g_0$ induces a well-defined map.
$$G(\Q_p)/\G(\Z_p)\rightarrow X(\sigma(\{\mu_y\}),b)_K$$

and we have $\kappa(g)=\kappa(g_0)\in \pi_1(G)_I$.
\end{prop}

\begin{remark}
$g_0$ and $b\in G(L)$ both depend on the choice of the isomorphism \ref{id}. Modifying the isomorphism by $h\in\G(\Ok_L)$ conjugates $g_0$ by $h$ and $\sigma$-conjugates $b$ by $h$, so that $$(hg_0^{-1}h^{-1})(hb\sigma(h^{-1}))\sigma(hg_0h^{-1})=hg_0b\sigma(g_0)\sigma(h^{-1})$$
In particular, the part of the Iwahori decomposition it lies in is independent of the choice of \ref{id}.
\end{remark}
\begin{proof}

We  identify $\D(\pdiv'_0)\otimes_{\Ok_L}L$ with $\D(\pdiv_0)\otimes_{\Ok_L}L$, so that we consider $$\D(\pdiv_0')=g_0\D(\pdiv_0)\subset \D(\pdiv_0)\otimes_{\Ok_L}L$$
Under this identification, we have $\sa\in\D(\pdiv_0')^\otimes$ and the stabilizer of these tensors in $\D(\pdiv_0')$ can be identified with $g_0\G_{\Ok_L}g_0^{-1}$. 

By \cite[Corollary 3.3.8]{KP}, there exists a $G$-valued cocharacter $\mu_y'$ defined over $K''$ such that $\pdiv'$ is a $(g_0\G_{\Ok_L}g_0^{-1},\mu_y')$ adapted-lifting of $\pdiv_0'\otimes_{\Ok_K}k$. We have three filtrations on $\D\otimes_{\Ok_L}K$: the one induced by $\mu_y$, the canonical filtration corresponding to the Galois representation $T_p\pdiv^\vee\otimes_{\Z_p}\Q_p\cong T_p'\pdiv^\vee\otimes_{\Z_p}\Q_p$ and the one induced by $\mu_y'$.  The second filtration is induced by $G$-valued cocharacters $\mu$ and $\mu'$ which are conjugate to $\mu_y$ and $\mu_y'$ respectively. The same proof as in \cite[Lemma 1.1.9]{Ki3}, shows that $\mu$ and $\mu'$ are $G$-conjugate, hence $\mu_y$ and $\mu_y'$ are $G$-conjugate. 

Thus $g_0^{-1}\mu_y'g_0$ induces a filtration on $\D\otimes_{\Ok_L}K''$ corresponding to a point in $M^{loc}_{\G,\mu_y^{-1}}(K'')$. The specialization of this point gives a filtration on $\D(\pdiv_0)\otimes_{\Ok_K}k$ which lies in $M^{loc}_{\G,\mu_y^{-1}}(k)$. This filtration is given by the reduction of $g_0^{-1}\sigma^{-1}(b^{-1}g_0)p\mod p$, hence by Propositon \ref{prop4} we have $$g_0^{-1}b\sigma(g_0)\in \G(\Ok_L)\sigma(\dot{w})\G(\Ok_L)$$ where $w\in\Adm_K(\{\mu_y\})$, i.e. $g_0\in X(\sigma(\{\mu_y\}),b)$.

To show $\kappa_G(g)=\kappa_G(g_0)$, note that $\tilde{g}$ gives a $k[[u]]^{perf}$ point of $Gr_{\G}$, where $Gr_{\G}$ is the Witt vector affine flag variety (see \cite{Zhu}, \cite{BS}). For $\mathbf{k}$ a perfect field of characterstic $p$, we have $$\kappa_G:Gr_{\G}(\mathbf{k})\rightarrow \pi_1(G)_I$$ and this induces an isomorphism $\pi_0(Gr_{\G})\cong\pi_1(G)_I$. In particular, $\kappa_g$ is a locally constant function. Let $h\in Gr_{\G}(\overline{k((u))})$ be the generic point of $\tilde{g}$, then by Proposition \ref{key1} ii), we have $\kappa_G(h)=\kappa_G(g)$, hence $\kappa_G(\sigma^{-1}(g_0))=\kappa_G(g)$. Since $g$ is $\sigma$-invariant, we hav $\kappa_G(g_0)=\kappa_G(g)$.

\end{proof}

\begin{prop}\label{prop2}
Let  $f:\G\rightarrow \mathcal{H}$ be a surjection of parahoric subgroups such that the following two conditions hold:

i) The composition of $f$ with $\Gamma_K\rightarrow \G$ factors through the center $Z_{\mathcal{H}}$ of $\mathcal{H}$. 

ii) The connected component of the identity of $f^{-1}(Z_{\mathcal{H}})$ is a parahoric subgroup of its generic fibre.

 Then we may choose the isomorphism \ref{id} so that we have $$f(g)=f(g_0)\in H(L)/\mathcal{H}(\Ok_L)$$
\end{prop}

\begin{proof}We write $\mathcal{G}'$ for the connected component of the identity of $f^{-1}(Z_{\mathcal{H}})$. By assumption $\G'$ is a parahoric subgroup of its generic fibre $G'$. Upon replacing $K$ by a finite extension, we may assume $\Gamma_K\rightarrow \G(\Z_p)$ factors through $\G'(\Z_p)$, and we may extend $\set\in T_p\pdiv^{\vee,\otimes}$ to a set $t_{\beta,\acute{e}t}$ of $\Gamma_K$-invariant tensors whose stabilizer is $\G'$. By Proposition \ref{prop3}, we obtain tensors $t_{\beta,0}\in\D(\pdiv_0)^\otimes$ whose stabilizer $\G'_{\Ok_L}$ can be identified with $\G'\otimes_{\Z_p}\Ok_L$. By \cite[Corollary 3.3.10]{KP} there is a $G'$-valued cocharacter $\mu_y'$ satisfying the conditions in (*), hence we may apply the construction in \ref{key1}. We fix a $\mathfrak{S}$-linear bijection $$T_p\pdiv^\vee\otimes_{\Z_p}\mathfrak{S}\cong \mathfrak{M}(T_p\pdiv^\vee)$$
taking $t_{\beta,\acute{e}t}$ to $\tilde{t}_\beta$. 

Let $g\in G(\Q_p)$, applying the previous construction we obtain $\tilde{g}\in G(\mathfrak{S}[1/p])$ and by Proposition \ref{key1} we have $\tilde{g}=hgi$ where $h\in \G'(\Ok_{\widehat{\mathscr{E}^{ur}}})$ and $i\in\G(\Ok_{\widehat{\mathscr{E}^{ur}}})$. Since $\G'\subset f^{-1}(Z_{\mathcal{H}})$, we have $f(g^{-1}hg)=f(h)$, so $$p:=g^{-1}hgh^{-1}\in  \mathcal{P}({\widehat{\mathscr{E}^{ur}}})$$ where $\mathcal{P}:=\ker(f:\G\rightarrow \mathcal{H})$.

Thus $f(\tilde{g})=f(gphi)=f(g)f(hi)$, and we obtain $$f(hi)\in \mathcal{H}(\Ok_{\widehat{\mathscr{E}^{ur}}})\cap\mathcal{H}(\mathfrak{S}[1/p])=\mathcal{H}(\mathfrak{S})$$

Since $\mathfrak{S}$ is strictly henselian, $f:\mathcal{G}(\mathfrak{S})\rightarrow\mathcal{H}(\mathfrak{S})$ is surjective, so let $m\in\G(\mathfrak{S})$ such that $f(m)=hi$. Thus $f(phim^{-1})=0$, and we have $$phim^{-1}\in\mathcal{P}({\widehat{\mathscr{E}^{ur}}})\cap\G(\mathfrak{S}[1/p])$$
But this last group is just $\mathcal{P}(\mathfrak{S}[1/p])$. Thus $\tilde{g}=g(phim^{-1})m\in g\mathcal{P}(\mathfrak{S}[1/p])\G(\mathfrak{S})$, and hence $$g_0=\sigma^{-1}(\tilde{g})|_{u=0}\in g\mathcal{P}(L)\G(\Ok_L)$$
so that $f(g_0)=f(\sigma(g))=f(g)\in H(L)/\mathcal{H}(\Ok_L)$.
\end{proof}

\begin{lemma}\label{lemma1}
i) The map $\kappa_G|_{G(\Q_p)}:G(\Q_p)\rightarrow \pi_1(G)_I^{\sigma}$ is surjective.

ii) Let $g_{ad}\in G_{ad}(\Q_p)$. Suppose the image of $g_{ad}$  under $\kappa_{G_{ad}}$ lifts to an element of $\pi_1(G)_I^\sigma$, then the image of $g_{ad}$ in $G(\Q_p)/\G(\Z_p)$ is in the image of $$G(\Q_p)/\G(\Z_p)\rightarrow G_{ad}(\Q_p)/\G_{ad}(\Z_p)$$
\end{lemma}

\begin{proof} i) By \cite[Lemma 5]{HaRa}, the Kottwitz homomorphism induces an exact sequence:

$$0\rightarrow \mathcal{T}^\circ(\Ok_L)\rightarrow T(L)\xrightarrow{\kappa_G}\pi_1(G)_I\rightarrow 0$$ 
where $\mathcal{T}^\circ$ is the connected Neron model of $T$. Since $H^1(\mathcal{T}^\circ,\Z_p)=0$ we have $$\kappa_G|_{T(\Q_p)}:T(\Q_p)\rightarrow \pi_1(G)_I^\sigma$$ is surjective, hence $\kappa_G|_{G(\Q_p)}$ is surjective.

ii) By part i), there exists $g\in G(\Q_p)$ which such that $\kappa_G(g)\in\pi_1(G)^\sigma_I$ lifts $\kappa_{G_{ad}}(g_{ad})$. Replacing $g_{ad}$ with $g_{ad}g^{-1}$, we may assume $\kappa_{G_{ad}}(g_{ad})$ is trivial.

By the Iwahori decomposition there exists $w_{ad}\in W_{ad}^\sigma, g_1,g_2\in\G_{ad}(\Z_p)$ such that $g_{ad}=g_1\dot{w}_{ad}g_2$, with $\dot{w}_{ad}\in G_{ad}(\Q_p)$ a lift of $w_{ad}$.  Changing our choice of torus $T$ with $T':=g_1Tg_1^{-1}$ we may assume $g_{ad}=\dot{w}_{ad}g_2$. Since $\kappa_{G_{ad}}(g)$ is trivial, $w_{ad}$ lies in the affine Weyl group $W_{a,ad}$ of $G_{ad}$. But the natural projection induces an isomorphism $W_a\cong W_{a,ad}$ hence, $w_{ad}$ lifts to an element $w$ of $W_a^\sigma$. Thus we may take $\dot{w}\in G_{der}(\Q_p)$ lifting $w$, and the image of $\dot{w}$ in $G(\Q_p)/\G(\Z_p)$ gives the required lifting.
\end{proof}
With the above notations we have the following

\begin{prop}\label{prop1} Assume $b=\sigma(\dot{\tau}_{\{\mu_y\}})$. Then there exists a $(\G,\mu_y)$-adapted lifting such that the map 

$$G(\Q_p)/\G(\Z_p)\rightarrow \pi_0(X(\sigma(\{\mu_y\}),\sigma(\dot{\tau}_{\mu_y}))),\ \ \ g\mapsto g_0$$
defined above is surjective.
\end{prop}

\begin{proof}We let $\mu_{y}^{ad}$ denote the cocharacter $G_{ad}$ induced by $\mu_y$. Let $G_{ad}=G_1\times G_2$ where $\mu_y^{ad}$ induces the trivial cocharacter of $G_1$ and induces a non-trivial cocharacter in every $\Q_p$-factor of $G_2$. We write $\G_1\times\G_2$ for the parahoric in $G_{ad}$ corresponding to $\G$. By Theorem 4.1, we have 

$$\pi_0(X(\{\mu_y^{ad}\},\dot{\tau}_{\mu_y^{ad}}))\cong G_1(\Q_p)/\G_1(\Z_p)\times \pi_1(G_2)_I^\sigma$$
We pick the isomorphism \ref{id} so that the conclusion of Proposition \ref{prop2} holds for the projection $\G\rightarrow \G_1$.

Let $h\in X(\sigma(\{\mu_y\}),\sigma(\dot{\tau}_{\mu_y})),$ $g\in G(\Q_p)$ and $h_{ad}$ the image of $h$ in $\pi_0(X(\sigma(\{\mu_y^{ad}\}),\sigma(\dot{\tau}_{\mu_y^{ad}})))$, then by Lemma \ref{lemma1} i), there exists $g_{ad}\in G_{ad}(\Q_p)$ mapping to $h_{ad}$. Since $\kappa_{G_{ad}}(g_{ad})$ lifts to the element $\kappa_G(h)\in\pi_1(G)_I^\sigma$, we have by Lemma \ref{lemma1} ii) that $g_{ad}$ lifts to an element $g\in G(\Q_p)$.  By Proposition \ref{prop2} and \ref{prop1}, the image of $g_0$ is equal to the image of $h$ in $\pi_0(X(\sigma(\{\mu_y^{ad}\}),\sigma(\dot{\tau}_{\mu_y^{ad}})))$. By \cite[Corollary 4.5]{HZ}, there exists $z\in Z(\Q_p)$ such that $g_0z=h$. By the functoriality of the construction, $(gz)_0=g_0z_0=g_0z=h$.
\end{proof}
\section{Shimura varieties}

\subsection{}We recall the construction of the integral models of Shimura varieties of Hodge type in \cite{KP}.

Let $G$ be a reductive group over $\Q$ and $X$ a conjugacy class of homomorphisms

$$h:\mathbb{S}:=\text{Res}_{\mathbb{C}/\mathbb{R}}\mathbb{G}_m\rightarrow G_\mathbb{R}$$
 such that $(G,X)$ is a Shimura datum in the sense of \cite{De}. 
 
 Let $c$ be complex conjugation, then $\text{Res}_{\C/\mathbb{R}}(\C)\cong(\C\otimes_{\R}\C)^\times\cong \C^\times \times c^*(\C^\times)$ and we write $\mu_h$ for the cocharacter given by $$\C^\times\rightarrow \C^\times\times c^*(\C^\times)\xrightarrow h G(\C)$$
 We set $w_h:=\mu_h^{-1}\mu_h^{c-1}$.

Let $\A_f$ denote the ring of finite adeles and $\A_f^p$ the subring of $\A_f$ with trivial $p$-component. Let $\rmK_p\subset G(\Q_p)$ and $\rmK^p\subset G(\A_f^p)$ be compact open subgroups and write $\rmK:=\rmK_p\rmK^p$. Then for $\rmK^p$ sufficiently small

$$Sh_{\rmK}(G,X)_{\C}=G(\Q)\backslash X\times G(\A_f)/\rmK$$
has the structure of an algebraic variety over $\C$, which has a model over the reflex field $\mathbf{E}:=E(G,X)$, which is a number field and is the field of definition of the conjugacy class of $\mu_h$.

We will also consider the pro-varieties

$$Sh(G,X):=\lim\limits_{\leftarrow \rmK}Sh_{\rmK}(G,X)$$

$$Sh_{\rmK_p}(G,X):=\lim\limits_{\leftarrow\rmK^p}Sh_{\rmK_p\rmK^p}(G,X)$$

\subsection{}Let $V$ be a vector space over $\Q$ of dimension $2g$ equipped with an alternating bilinear form $\psi$, we write $V_R=V\otimes_{\Q}R$ for an $\Q$ algebra $R$. Let $GSp=GSp(V,\psi)$ denote the corresponding group of symplectic similitudes, the Siegel half space is defined to be the set of homomorphisms $h:\mathbb{S}\rightarrow GSp_\R$ such that:

1) The Hodge structure on $V_\C$ induced by $h$ is of type $(-1,0),(0,-1)$, i.e. $$V_{\C}=V^{-1,0}\oplus V^{0,-1}$$

2) $(x,y)\mapsto \psi(x,h(i)y)$ is positive or negative definite on $V_{\R}$.

For the rest of this section we assume there is an embedding of Shimura data $$\iota:(G,X)\rightarrow (GSp,S^{\pm})$$ We sometimes write $G$ for $G_{\Q_p}$ when there is no risk of confusion. For the rest of this section we assume the following condition holds 
\begin{equation}\text{$G$ splits over a tamely ramified extension of $\Q_p$ and $p\nmid|\pi_1(G_{der})|$}\end{equation}

Let $\G$ be a connected parahoric subgroup of $G$, i.e. $\G=\G_x=\G_x^{\circ}$ for some $x\in B(G,\Q_p).$ By \cite[2.3.15]{KP}, upon replacing $\iota$ by another symplectic embedding, there is a closed immersion $\G\rightarrow \mathcal{GSP}$, where $\mathcal{GSP}$ is a parahoric group scheme of $GSp$ corresponding to the stabilizer of a lattice $V_{\Z_p}\subset V$.  Upon scaling $V_{\Z_p}$, we may assume $V_{\Z_p}^\vee\subset V_{\Z_p}$ and we let $p^d=|V_{\Z_p}/V_{\Z_p}^\vee|$.
This induces a closed immersion of local models

$$M^{loc}_{\G,\mu_h}\rightarrow M^{loc}_{\mathcal{GSP},\mu_h}\otimes\Ok_{E}$$
where $E$ is the local reflex of $\mu_h$ in $G_{\Q_p}$.

\subsection{} Let $V_{\Z_{(p)}}=V_{\Z_p}\cap V$, we write $G_{\Z_{(p)}}$ for the Zariski closure of $G$ in $GL(V_{\Z_{(p)}})$, then $G_{\Z_{(p)}}\otimes_{\Z_{(p)}}\Z_p\cong \G$. The choice of $V_{\Z_{(p)}}$ gives rise to an interpretation of $Sh_{\rmK'}(GSp,S^\pm)$ as a moduli space of abelian varieties and hence an integral model over $\Z_{(p)}$ which we now describe. We let $\rmK'=\rmK_p'\rmK'^p $ where $\rmK'_p=\mathcal{GSP}(\Z_p)$ and $\rmK'^p \subset GSp(\A_f^p)$ is a compact open.

Let $\mathcal{A}$ be an abelian scheme of dimension $2g=\dim V$ over a scheme $T$. We write $$\widehat{V}(\mathcal{A})=\lim_{\leftarrow p\nmid n}\mathcal{A}[n]$$
Consider the category obtained from the category of abelian varieties by tensoring the $\Hom$ groups by $\Z_{(p)}$,  an object in this category will be called an abelian variety up to prime to $p$ isogeny. An isomorphism in this category will be called a prime to $p$ isogeny.

Let $\mathcal{A}$ be an abelian variety up to prime to $p$ isogeny and let $\mathcal{A}^*$ be the dual abelian variety, by a weak-polarization we mean an equivalence class of quasi-isogenies $\lambda:\mathcal{A}\rightarrow\mathcal{A}^*$ such that $p^d$ exactly divides $\deg\lambda$ and some multiple of $\lambda$ is a polarization. Two such quasi-isogenies are equivalent if they differ by a multiple of $\Z_{(p)}^\times$.

Let $(\mathcal{A},\lambda)$ be a pair as above, we write $\underline{Isom}_{\lambda,\psi}(\widehat{V}(\mathcal{A}),V_{\A_f^p})$ for the (pro)-\'etale sheaf of isomorphisms $\widehat{V}(\mathcal{A})\cong V_{\A_f^p}$ which preserves the pairings induced $\lambda$ and $\psi$ up to a $\widehat{\Z}_p^\times$ scalar. 

We write $\mathscr{A}_{g,d,\rmK'}(T)$ for the set of triples $(\mathcal{A},\lambda,\epsilon_{\rmK'}^p)$ consisting of an abelian variety up to prime to $p$-isogeny $\mathcal{A}$ over $T$ together with a weak polarization $\lambda:\mathcal{A}\rightarrow \mathcal{A}^*$ and a global section $$\epsilon_{\rmK'}^p\in\Gamma(T,\underline{Isom}_{\lambda,\psi}(\widehat{V}(\mathcal{A}),V_{\A_f^p})/\rmK'^p)$$
For $\rmK'^p$ sufficiently small, $\mathscr{A}_{g,d,\rmK'}$ is representable by a quasi-projective scheme over $\Z_{(p)}$ which we denote by $\mathscr{S}_{\rmK'}(GSp,S^\pm)$.

\subsection{} For the rest of this paper we fix an algebraic closure $\overline{\Q}$, and for each place $v$ of $\Q$ an algebraic closure $\overline{\Q}_v$ together with an embedding $\overline{\Q}\rightarrow \overline{\Q}_v$.

By \cite[Lemma 2.1.2]{Ki1}, we can choose $\rmK'$ such that $\iota$ induces a closed immesion:

$$Sh_{\rmK}(G,X)\hookrightarrow Sh_{\rmK'}(GSp,S^\pm)_{\mathbf{E}}$$ defined over $\mathbf{E}$. The choice of embedding $\mathbf{E}\rightarrow \overline{\Q}_p$ determines a place $v$ of $\mathbf{E}$. Write $\Ok_{\mathbf{E},{(v)}}$ for the localisation of $\Ok_{\mathbf{E}}$ at $v$, $E$ the completion of $\mathbf{E}$ at $v$ and $\Ok_E$ the ring of intgers of $E$. We assume the residue field has $q=p^r$ elements and as before $k$ will denote an algebraic closure of $\mathbb{F}_q$. We define $\mathscr{S}_\rmK(G,X)^-$ to be the Zariski closure of $Sh_{\rmK}(G,X)$ inside $\mathscr{S}_{\rmK'}(GSp,S^\pm)\otimes_{\Z_{(p)}}\Ok_{\mathbf{E},(v)}$, and $\mathscr{S}_{\rmK}(G,X)$ to be its normalization. By construction, for $\rmK^p_1\subset \rmK^p_2$ compact open subgroups of $G(\A_f^p)$, there are well defined maps $\mathscr{S}_{\rmK_p\rmK^p_1}(G,X)\rightarrow \mathscr{S}_{\rmK_p\rmK^p_2}(G,X)$ and we write $\mathscr{S}_{\rmK_p}(G,X):=\varprojlim_{\rmK^p}\mathscr{S}_{\rmK_p\rmK^p}(G,X)$.

Under these assumptions we have the following.
\begin{thm}[\cite{KP} Theorem 4.2.2, Theorem 4.2.7]) The $\Ok_{\mathbf{E},(v)}$ scheme $\mathscr{S}_{\rmK_p}(G,X)$ is a flat $G(\A_f^p)$-equivariant extension of $Sh_{\rmK_p}(G,X)$. 

ii) Let $\hat{U}_x$ be the completion of $\mathscr{S}_{\rmK}(G,X)^-$ at some $k$-point $x$, there exists a point $x'\in M^{loc}_{\G,\mu_h}(k)$ such that the irreducible components of $\hat{U}_x$ are isomorphic to the completion $\widehat{M}^{loc}_{\G,\mu_h}$ at $x'$. Moreover $\mathscr{S}_{\rmK}(G,X)$ fits in a local model diagram:

\[\xymatrix{ &\widetilde{\mathscr{S}}_{\rmK}(G,X)_{\Ok_E}\ar[dr]^\pi\ar[dl]^q&\\
\mathscr{S}_{\rmK}(G,X)_{\Ok_E} & &M^{loc}_{\G,X}}\]
where $q$ is a $\G$-torsor and $\pi$ is smooth of relative dimension $\dim G$.
\end{thm}

\subsection{}We will need a more explicit description of $\hat{U}_x$ and this local model diagram for the next section. To do this we will need to introduce Hodge cycles.

By \cite[1.3.2]{Ki1}, the subgroup $G_{\Z_{(p)}}$ is the stabilizer of a collection of tensors $s_\alpha\in V_{\Z_{(p)}}^\otimes$. Let $h:\mathcal{A}\rightarrow \mathscr{S}_{\rmK}(G,X)$ denote the pullback of the universal abelian variety on $\mathscr{S}_{\rmK'}(GSp,S^\pm)$ and let $V_B:=R^1h_{an,*}\Z_{(p)}$, where $h_{an}$ is the map of complex analytic spaces associated to $h$. We also let $\mathcal{V}=R^1h_*\Omega^\bullet$ be the relative de Rham cohomology of $\mathcal{A}$. Using the de Rham isomorphism, the $s_\alpha$ give rise to a collection of Hodge cycles $s_{\alpha,dR}\in \mathcal{V}_\C^\otimes$, where $\mathcal{V}_\C$ is the complex analytic vector bundle associated to $\mathcal{V}$. By \cite[\S 2.2.]{Ki1}, these tensors are defined over $E$, and in fact over $\mathcal{O}_{E,(v)}$ by \cite[Proposition 4.2.6]{KP}.

Similarly for a finite prime $l\neq p$, we let $\mathcal{V}_l=R^1h_{\acute{e}t*}\Q_l$ and $\mathcal{V}_p=R^1h_{\eta,\acute{e}t*}\Z_p$ where $h_\eta$ is the generic fibre of $h$. Using the \'etale-Betti comparison isomorphism, we obtain tensors $s_{\alpha,l}\in \mathcal{V}^\otimes_l$ and $s_{\alpha,p}\in\mathcal{V}_p^\otimes$.

For $*=B, dR,l$ and $x\in \mathscr{S}_{\rmK_p}(G,X)(T)$, we write $\mathcal{A}_x$ for the pullback of $\mathcal{A}$ to $x$ and $s_{\alpha,*,x}$ for the pullback of $s_{\alpha,*}$ to $x$.

As in \cite[3.4.2.]{Ki1}, if $x\in\mathscr{S}_{K_p}(G,X)(T)$ corresponds to a triple $(\mathcal{A}_x,\lambda,\epsilon_{\rmK'}^p)$, then $\epsilon_{K'}^p$ can be promoted to a section:

$$\epsilon_{\rmK}^p\in\Gamma(T,\underline{Isom}_{\lambda,\psi}(\widehat{V}(\mathcal{A}),V_{\A_f^p})/K^p)$$ which takes $s_{\alpha,l}$ to $s_{\alpha}$ ($l\neq p$).

\subsection{} Recall $k$ is an algebraic closure of $\F_q$ and $L=W(k)[1/p]$. Let $\overline{x}\in\mathscr{S}_{\rmK}(G,X)(k)$ and $x\in\mathscr{S}_{\rmK}(G,X)(\Ok_K)$ a point lifting $\overline{x}$, where $K/L$ is a finite extension.

Let $\pdiv_x$ denote the $p$-divisible group associated to $\mathcal{A}_x$ and $\pdiv_{x,0}$ its special fiber. Then $T_p\pdiv_x^\vee$ is identified with $H^1_{\acute{e}t}(\mathcal{A}_x,\Z_p)$ and we obtain $\Gamma_K$-invariant tensors $s_{\alpha,\acute{e}t,x}\in T_p\pdiv^{\vee\otimes}$ whose stabilizer can be identified with $\G$. We may thus apply the constructions of section 3 and we obtain $\varphi$-invariant tensors $s_{\alpha,0,x}\in\D(\pdiv_{x,0})$ whose stabilizer group $\G_{\Ok_L}$ can be identified with $\G\otimes_{\Z_p}\Ok_L$. The filtration on $\D\otimes_{\Ok_L}K$ corresponding to $\pdiv_x$ is induced by a $G$-valued cocharacter conjugate to $\mu_h^{-1}$. By \cite[Proposition 3.3.8]{KP}, there is an isomorphism:

$$\D(\pdiv_x)(\Ok_K)\cong\D(\pdiv_{x,0})\otimes_{\Ok_L}\Ok_K$$ taking $s_{\alpha,dR,x}$ to $s_{\alpha,0,x}$ lifting the identity mod$\pi$. Thus there is a $G$-valued cocharacter $\mu_y$ which is $G$-conjugate to $\mu_h^{-1}$ and induces a filtration on $\D(\pdiv_{x,0})$ lifting the filtration on $\D(\pdiv_{x,0})$. Thus we have a notion of $(\G_{\Ok_L},\mu_y)$-adapted liftings as in section 3 and by definition $\pdiv_x$ is a $(\G_{\Ok_L},\mu_y)$-adapted lifting.

As before we let $P\subset GL(\D)$ be a parabolic lifting $P_0$. We obtain formal local models $\widehat{M}^{loc}_{\mu_y^{-1}}=\mbox{Spf}A$ and $\widehat{M}^{loc}_{\G,\mu_y^{-1}}=\mbox{Spf}A_{\G}$, and the filtration corresponding to $\mu_y$ is given by a point $y:A_{\G}\rightarrow \Ok_K$.

\begin{prop}\label{prop12}Let $\widehat{U}_{\overline{x}}$ be the completion of $\mathscr{S}^-_{\rmK}(G,X)$ at $\overline{x}$. 

i) $\widehat{U}_{\overline{x}}$ can be identified with a closed subspace of $\mbox{Spf}A$ containing $\mbox{Spf}A_{\G}$. 

ii) Let $x'\in\mathscr{S}_{\rmK}(G,X)(\Ok_{K'})$ whose special fibre $\overline{x}'$  maps to the image of  $\overline{x}$ in $\mathscr{S}^-_{\mathrm{K}}(G,X)$. Then $s_{\alpha,0,x'}=s_{\alpha,0,x}\in\D(\pdiv_{x,0})$ if and only if $x$ and $x'$ lie on the same irreducible component of $\mathscr{S}_{\rmK}(G,X)$.

ii) A deformation $\pdiv$ of $\pdiv_{x,0}$ corresponds to a point on the same irreducible component of $\widehat{U}_{\overline{x}}$ if and only if $\pdiv$ is $\G_{\Ok_L}$-adapted
\end{prop}
\begin{proof}This is effectively \cite[Proposition 4.2.2]{KP} we recall the argument for the reader's convenience.

Recall we assumed that $G$ splits over a tamely ramified extension of $\Q_p$. Moreover $G_{\Ok_L}\otimes_{\Ok_L}L\subset GL(\D(\pdiv_{x,0}))$ contains the scalars, since it contains the image of $w_h$. Thus we may apply the construction of section 3 to the tensors $s_{\alpha,0,x}$; we may equip $\mbox{Spf}A$ with the structure of a versal deformation space for $\pdiv_{x,0}$ and the subspace $\mbox{Spf}A_{\G}$ is such that $\varpi:A\otimes_{\Z_p}\Ok_{E}\rightarrow K$ factors through $A_{\G}$ if and only if the induced $p$-divisible group $\pdiv_{\varpi}$ is $(\G_{\Ok_L},\mu_y)$-adapted, where $\G_{\Ok_L}$ is the stabilizer of $s_{\alpha,0,x}$  and $\mu_y$ is $G$-conjugate to $\mu_{h^{-1}}$.

The $p$-divisible over $\widehat{U}_{\overline{x}}$ is induced by pullback from a map $\widehat{U}_{\overline{x}}\rightarrow \mbox{Spf}A$ which is a closed immersion by the Serre-Tate theorem. Let $Z\subset U_{\overline{x}}$ be the irreducible component containg $x$. Let $x'\in Z(K')$, the same arguement as in \cite[2.3.5]{Ki1} shows that $s_{\alpha,0,x'}=s_{\alpha,0,x}$, hence we obtain one direction in ii) and $x'\in\text{Spf}A_{\G}$ since the filtration on $\D\otimes_{\Ok_L}K'$ corresponding to $\G_{x'}$ is given by a $G$-valued cocharacter conjugate to $\mu_{h}^{-1}$. Since this holds for all $x'\in Z(K')$, we have $Z\subset \text{Spf}A_\G$ and hence they are equal since they have the same dimension. We thus obtain iii) and the other implication in ii).

\end{proof}

The previous proposition shows that the tensors $s_{\alpha,0,x}$ are independent of the choice of $x\in\mathscr{S}_{\rmK_p}(G,X)$ lifting  $\overline{x}$, thus we denote them by $s_{\alpha,0,\overline{x}}$. The following is then immediate.

\begin{cor}\label{cor3}
Let $\overline{x},\overline{x}'\in\mathscr{S}_{\rmK}(G,X)(k)$ be points whose image in $\mathscr{S}_{\rmK_p}(G,X)^-(k)$ coincide. Then $\overline{x}=\overline{x}'$ if and only if $s_{\alpha,0,\overline{x}}=s_{\alpha,0,\overline{x}'}$.
\end{cor}

\subsection{} We would like to show the isogeny classes in $\mathscr{S}_{\rmK}(G,X)(k)$ admit maps from $X(\sigma(\{\mu_y\}),b)$. We will show this when $G$ is residually split at $p$ and in general for the basic case. In the rest of this section we will prove the case when $\G$ is an Iwahori subgroup of $G$; the general case will be deduced from this in \S 7. We thus assume $\G$ is an Iwahori subgroup for the rest of the section.

Let $\overline{x}\in\mathscr{S}_{\rmK}(G,X)(k)$ and $x\in\mathscr{S}_{\rmK}(G,X)(K)$ a point lifting $\overline{x}$. Let $\G_{\Ok_L}$ denote the stabilizer $s_{\alpha,0,\overline{x}}$. By the above $\pdiv_x$ is a $(\G_{\Ok_L},\mu_y)$-adapted lifting of $x$ and we have an $\Ok_L$-linear bijection $$T_p\pdiv_x^\vee\otimes_{\Z_p}\Ok_L\cong \D(\pdiv_{\overline{x}})$$ taking $s_{\alpha,\acute{e}t,x}$ to $s_{\alpha,0,\overline{x}}$. We fix an isomorphism $V^*_{\Z_p}\cong T_p\pdiv_x^\vee$ taking $\sa$ to $s_{\alpha,\acute{e}t,x}$, this identifies the stabilizer $\G_{\Ok_L}$ of $s_{\alpha,\acute{e}t,x}$ with $\G\otimes_{\Z_p}\Ok_L$. 

Since the $s_{\alpha,0,\overline{x}}$ are $\varphi$-invariant, we may write $\varphi=b\sigma$ for some $b\in G(L)$ which is independent of the above choices up to $\sigma$-conjugation by elements of $\G(\Ok_L).$

Fix $S$ a maximal $L$-split torus in $G$ with centralizer $T$ as in section 5 so that $\G$ corresponds to an alcove in the apartment $\mathcal{A}(G,S,\Q_p)$. As in \S5.4, we have $$b\in \bigcup_{w\in\Adm(\{\mu_y\})}\G(\Ok_L)\sigma(\dot{w})\G(\Ok_L)$$
Write $\mu\in X_*(T)$ for the dominant (with respect to a choice of Borel defined over $L$) representative of $\{\mu_y\}=\{\mu_h^{-1}\}$, and $\underline{\mu}$ its image in $X_*(T)_I$. With the notation of section 5, we have $1\in X(\sigma(\{\mu\}),b)$.

Recall $X(\sigma(\{\mu\}),b)$ is equipped with an action $\Phi$ given by $$\Phi(g)=(b\sigma)^r(g)=b\sigma(b)\dots\sigma^{r-1}(b)\sigma^r(g)$$
where $r$ is the residue degree of $\Ok_{E}/\Z_p$. We have $$\Phi(g)^{-1}b\sigma(\Phi(g))=\sigma^r(g^{-1}b\sigma(g))\in\bigcup_{w\in\Adm(\{\mu\})}\G(\Ok_L)\sigma^{r+1}(\dot{w})\G(\Ok_L)$$ By \cite[Lemma 5.1]{Ra}, $\Adm(\{\mu\})$ is stable under $\sigma^r$, hence $\Phi(g)\in X(\sigma(\{\mu\}),b)$ and $\Phi$ is well defined.

Pick a basis for $V_{\Z_p}$ compatible with $S$ as in section 3, this is equivalent to the choice of a maximal split torus $T'\subset GL(V_{\Z_p})$. By Corollary \ref{cor2}, for $g\in X(\sigma(\{\mu\}),b))$ we have $g^{-1}b\sigma(g)\in \mathcal{GL}(\Ok_L)v_{GL}(p)\mathcal{GL}(\mathcal{O}_L)$ where $v_{GL}$ is the cocharacter $(1^{(g)},0^{(g)})$. Thus the Hodge polygon of the $F$-crystal $g\D(\pdiv_{\overline{x}})$ has slopes 0,1 hence corresponds  to a $p$-divisible group $\G_{g\overline{x}}$ which is isogenous to $\pdiv_x$ and hence to an abelian variety $\mathcal{A}_{g\overline{x}}$ isogenous to $\mathcal{A}_{\overline{x}}$. $\mathcal{A}_{g\overline{x}}$ is equipped with a prime to $p$ level structure corresponding to the one on $\mathcal{A}_{\overline{x}}$. Since $g(s_{\alpha,0,\overline{x}})=s_{\alpha,0,\overline{x}}$, we have $s_{\alpha,0,\overline{x}}\in\D(\pdiv_{g\overline{x}})$.

Since $g\in GSp(L)$ the weak polarisation on $\lambda_{\overline{x}}$ induces a weak polarisation on $\mathcal{A}_{g\overline{x}}$. Thus $\mathcal{A}_{g\overline{x}}$ together with the extra structure gives a point of $\mathscr{S}_{\rmK'}(GSp(V),S^\pm)(k)$ and we obtain a map $$i_{\overline{x}}':X(\sigma(\{\mu_y\}),b)\rightarrow\mathscr{S}_{\rmK'}(GSp(V),S^\pm)(k)$$

The main result of this section is the following:

\begin{prop}\label{prop6.1}
Let $G$ be residually split or suppose $b$ is basic. Then there exists a unique map $$i_{\overline{x}}: X(\sigma(\{\mu_y\}),b)\rightarrow \mathscr{S}_{\rmK}(G,X)(k)$$ lifting $i_{\overline{x}}'$ such that $s_{\alpha,0,i_{\overline{x}}(g)}=s_{\alpha,0,\overline{x}}$. Moreover we have $$\Phi\circ i_{\overline{x}}=i_{\overline{x}}\circ\Phi$$ where $\Phi$ acts on $\mathscr{S}_{\rmK}(G,X)(k)$ via the geometric Frobenius.
\end{prop}

The rest of this section will be devoted to the proof of Proposition \ref{prop6.1}.

\subsection{}The uniqueness follows from Corollary \ref{cor3}. The same proof as in \cite[\S 1.4.4]{Ki1} shows the compatibility with $\Phi$. Thus it remains to show the existence of $i_{\overline{x}}$. The strategy follows \cite[\S 1.4]{Ki1}; the first step is to show that if $g\in X(\sigma(\{\mu_y\}),b)$ can be lifted, then every point on the connected component of $X(\sigma(\{\mu_y\}),b)$ containing $g$ also lifts. The second step is to show that for every connected component of $X(\sigma(\{\mu_y\}),b)$ contains a point which lifts, this is done by showing the quasi-isogeny $\mathcal{A}_{g\overline{x}}\rightarrow \mathcal{A}_{\overline{x}}$ lifts to characteristic 0.

We recall some definitions from \cite[Appendix A]{HZ}, see also \cite{CKV}.
\begin{definition}
Let $R$ be  $k$ algebra. A frame for $R$ is a $p$ torsion free, $p$-adically complete and separated $\mathcal{O}_{L}$ algebra $\mathscr{R}$ equipped with an isomorphism $R\cong \mathscr{R}/p\mathscr{R}$ and a lift (again denoted $\sigma$) of the the Frobenius $\sigma$ on $R$.
 \end{definition}
 
 Let $R$ be as above and fix $\mathscr{R}$ a frame for $R$. We write $\mathscr{R}_{L}$ for $\mathscr{R}[\frac{1}{p}]$. If $\kappa$ is any perfect field of characteristic $p$ and $s: R\rightarrow \kappa$ is a map, then there is a unique $\sigma$-equivariant map $\mathscr{R}\rightarrow W(\kappa)$, also denoted $s$. If $R\rightarrow R'$ is an \'etale map, then there exists a canonical frame $\mathscr{R}'$ of $R'$ and a unique $\sigma$-equivariant lifting $\mathscr{R}\rightarrow\mathscr{R}'$.
 
 Let $g\in G(\mathscr{R}_L)$. For $C\subset W$, we write 
 $$S_C(g)=\bigcup_{w\in W}\{s\in\Spec R|s(g^{-1}b\sigma(g))\in\mathcal{G}(W(\overline{\kappa}(s)))w\mathcal{G}(W(\overline{\kappa}(s)))\}$$
where $\overline{\kappa}(s)$ is an algebraic closure of residue field $k(s)$ of $s$. Note that this only depends on the image of $g\in G(\mathscr{R}_{L})/\mathcal{G}(\mathscr{R})$, hence we can define $S_C(g)$ for any element of $g\in G(\mathscr{R}_{L})/\mathcal{G}(\mathscr{R})$. For $b\in G(L)$, we define the set $$X_C(b)(\mathscr{R})=\{g\in G(\mathscr{R}_L)/\mathcal{G}(\mathscr{R})| S_C(g)=\Spec R\}$$
When $C=\Adm(\{\mu\})$ we write $X(\{\mu\},b)(\mathscr{R})$ for $X_C(b)(\mathscr{R})$. Similarly when $C=\sigma(\Adm(\{\mu\}))$ we write $X(\sigma(\{\mu\}),b)(\mathscr{R})$.
\begin{definition} For $g_0,g_1\in X(\{\mu\},b)$ and $R$ a smooth  $k$-algebra with connected spectrum and frame $\mathscr{R}$, we say $g_0$ is connected to $g_1$ via $R$ if there exists $g\in X(\{\mu\},b)(\mathscr{R})$ and two $k$-points $s_0,s_1$ of $\Spec R$ such that $s_0(g)=g_0$ and $s_1(g)=g_1$. 

\end{definition}
We write $\sim$ for the equivalence relation on $X(\{\mu\},b)$ generated by the relation $g_0\sim g_1$ if $g_0$ is connected to $g_1$ via some $R$ as above, and we write $\pi_0'(X(\{\mu\},b))$ for the set of equivalence classes.

By \cite[Theorem A.4]{HZ}, we have $$\pi'_0(X(\{\mu\},b))=\pi_0(X(\{\mu\},b))$$

\subsection{} Returning to the situation of Proposition \ref{prop6.1}, let $g\in G(\mathscr{R}_L)$ be a lift of some element of $X(\sigma(\{\mu_y\}),b)(\mathscr{R})$. By Corollary \ref{cor2}, for all $s\in \Spec R$, we have $g(s)\in \mathcal{GL}(\Ok_L)\mu_{GL}(p)\mathcal{GL}(\Ok_L)$. Here we have chosen a maximal split torus $T'$ of $GL_{2g}$ as in \S3.6, then $\mu_{GL}$ is a dominant representative of $\{\mu_y\}$ in $X_*(T')$ and $\mathcal{GL}$ is the hyperspecial subgroup of $GL_{2g}$ stabilizing $V_{\Z_p}$. By \cite[Lemma 2.1.4]{CKV} there is an \'etale covering $R\rightarrow R'$ with canonical frame $\mathscr{R}\rightarrow \mathscr{R}'$ such that $$g\in \mathcal{GL}(\mathscr{R}')\mu_{GL}(p)\mathcal{GL}(\mathscr{R}')$$

 For $n\geq 1$ we write $\mathscr{R}_n$ for the ring $\mathscr{R}$ considered as a $\mathscr{R}$-algebra via $\sigma^n:\mathscr{R}\rightarrow \mathscr{R}$ and we define $R_n$ similarly. By \cite[Lemma 1.4.6]{Ki2}, there exists $n\geq 1$ and a $p$-divisible group $\pdiv_{g\overline{x}}$ over $R_n$ together with a quasi isogeny $\pdiv_{g\overline{x}}\rightarrow \pdiv_{\overline{x}}\otimes R_n$ which identifies $\D(\pdiv_{g\overline{x}})(\mathscr{R}'_n)$ with $g\D(\pdiv_{\overline{x}})\subset \D(\pdiv_{\overline{x}})\otimes_{\Ok_L}\mathscr{R}'_{nL}$. Upon relabelling $\mathscr{R}_n'$ as $\mathscr{R}$, we  obtain an abelian variety $\mathcal{A}_{g\overline{x}}$ over $\Spec R$.

Since $g\in GSp(\mathscr{R}_L)$, $\lambda_{\overline{x}}$ induces a weak polarization $\lambda_{g\overline{x}}$ on $\mathcal{A}_{g\overline{x}},$ and $\mathcal{A}_{g\overline{x}}$ is also equipped with a prime to $p$ level structure. Hence $g$ gives a map \begin{equation}\label{eq1}\Spec R\rightarrow \mathscr{S}_{\rmK'}(GSp,S^\pm)\end{equation}

Since $g\in G(\mathscr{R}_L)$, we have $s_{\alpha,0,\overline{x}}=g(s_{\alpha,0,\overline{x}})\in \D(\pdiv_{g\overline{x}})(\mathscr{R})$.

\begin{prop}\label{prop8}
Suppose there is a point $x_R\in\Spec R(k)$ such that $x_R^*(g)=1$. Then there is a unique lifting $i_R:\Spec R\rightarrow \mathscr{S}_{\rmK}(G,X)(R)$ of (\ref{eq1}) such that $$i_R^*(s_{\alpha,0})=s_{\alpha,0,\overline{x}}$$
\end{prop}
\begin{proof}
The uniqueness can be checked on $k$ points, hence this follows from Corollary \ref{cor3}.

To show existence, we first claim (\ref{eq1}) factors through $\mathscr{S}^-_{\rmK}(G,X)$. Let $\hat{R}$ denote the completion of $R$ at $x_R$, since $R$ is integral, it suffices to prove the claim for $\hat{R}$.

 Note that the filtration induced by $g^{-1}b\sigma(g)$ gives an $R$ point of the local model $M^{loc}_{\mathcal{GL}}$. For all $k$ points $s:R\rightarrow k$, we have $$g(s)^{-1}b\sigma(g(s))\in \bigcup_{w\in\Adm(\{\mu_y\})}\G(\Ok_L)\sigma(\dot{w})\G(\Ok_L)$$ By Corollary \ref{cor2}, the map $\Spec R\rightarrow M^{loc}_{\mathcal{GL}}$ factors through $M^{loc}_{\G}$. Taking completions at the image of $x_R$, we obtain a map $$\psi:A_\G\rightarrow \hat{R}$$ 
  By smoothness of $R$, we have $\hat{R}$ is power series ring over $k$. Since the $k[[t]]$ points of $\hat{R}$ are dense, we may assume $\hat{R}=k[[t]]$.

 We have a $p$-divisible group  over $k[[t]]$, we would like to use the map $\psi$ to deform this $p$-divisible group to $\tilde{\pdiv}$ over a ring in characteristic $0$, such that the pullback to every $\Ok_K$ point satisfies the condition in Definition \ref{G-adapted}, i.e. it is $(\G_{\Ok_L},\mu_y)$-adapted. The ring we will deform to is $A_{\G}[[t]].$
 
 We have a map $A_\G[[t]]\twoheadrightarrow k[[t]]$  induced by $\psi:A_\G\rightarrow k[[t]]$ and $t\mapsto t$, this induces a surjection $\widehat{W}(A_{\G}[[t]])\twoheadrightarrow\widehat{W}(k[[t]])$. 
 
  Let us write $\hat{g}$ for the image of $g$ in $ G(\hat{\mathscr{R}})$ and $\pdiv_{\hat{g}\overline{x}}$ for the induced $p$-divisible group, then $\D(\pdiv_{\hat{g}\overline{x}})(\hat{\mathscr{R}})$ can be identified with $\hat{g}\D(\pdiv_{\overline{x}})$. We may use $\hat{g}^{-1}$ to identify $\D(\pdiv_{\hat{g}\overline{x}})(\hat{\mathscr{R}})$ with $\D(\pdiv_{\overline{x}})\otimes_{\Ok_L}\hat{\mathscr{R}}$ as a $\hat{\mathscr{R}}$-module. Under this identification the Frobenius is given by $\hat{g}^{-1}b\sigma(\hat{g})$. It follows that the Dieudonn\'e display $\D(\pdiv_{\hat{g}\overline{x}})(\widehat{W}(k[[t]]))$ can be identified with $\D\otimes_{\Ok_L}\widehat{W}(k[[t]])$ and the Frobenius $\Phi$ preserves $s_{\alpha,0,\overline{x}}$. 
  
 Let $\text{Spf}A$ be the completion of $M_{\mathcal{GL},\mu_h^{-1}}^{loc}$ at the image of $x_R$, then  $\D\otimes_{\Ok_L}A$ is equipped with universal  filtration $\overline{M}_1\subset \D\otimes_{\Ok_L}A$. We let $M_1$ denote the preimage of $\overline{M}_1$ in $M:=\D\otimes_{\Ok_L}\widehat{W}(A)$. Let $\widetilde{M}_1$ denote the image of the map $\varphi^*M_1\rightarrow \varphi^*M$.
  
  By construction, the pushforward of $\overline{M}_1$ along $A\rightarrow A_\G\rightarrow k[[t]]$ is the filtration on $\D\otimes_{\Ok_L}k[[t]]$ induced by $\hat{g}^{-1}b\sigma(\hat{g})$. Therefore by \cite[Lemma 3.1.5]{KP} the structure of display on $\D\otimes_{\Ok_L}\widehat{W}(k[[t]])$ corresponding to $\pdiv_{\hat{g}\overline{x}}$ is given by an isomorphism $$\Psi_{k[[t]]}:\widetilde{M}_{1,k[[t]]}\rightarrow \D\otimes_{\Ok_L}\widehat{W}(k[[t]])$$ where for any ring $R$ with $A\rightarrow R$, we write $\widetilde{M}_{1,R}$ for the base change $\widetilde{M}_1\otimes_{\widehat{W}(A)}\widehat{W}(R)$. Since $A\rightarrow k[[t]]$ factors through $A_{\G}$ it follows from \cite[Corollary 3.2.11]{KP} that $s_{\alpha,0,\overline{x}}\in \widetilde{M}_{1,k[[t]]}$, and since  $\hat{g}b\sigma(\hat{g})$ preserves $s_{\alpha,0,\overline{x}}$ we have $\Psi_{k[[t]]}(s_{\alpha,0,\overline{x}})=s_{\alpha,0,\overline{x}}$.
  
  By \cite[Corollary 3.2.1]{KP}, the scheme $$\mathcal{T}=\underline{Isom}_{s_{\alpha,0,\overline{x}}}(\widetilde{M}_{1,A_\G},M\otimes_{\widehat{W}(A)}\widehat{W}(A_\G))$$ is a $\G$-torsor. Base changing to $A_{\G}[[t]]$ we obtain a $\G$-torsor 
  $$\mathcal{T}_{A_\G[[t]]}=\underline{Isom}_{s_{\alpha,0,\overline{x}}}(\widetilde{M}_{1,A_\G[[t]]},M\otimes_{\widehat{W}(A)}\widehat{W}(A_\G[[t]]))$$
  By smoothness of $\G$, $\Psi_{k[[t]]}$ lifts to an isomorphism $$\Psi:\widetilde{M}_{1,A_\G[[t]]}\xrightarrow\sim M\otimes_{\widehat{W}(A)}\widehat{W}(A_\G[[t]]))$$
  Again, b y \cite[Lemma 3.1.5]{KP}, this corresponds to a display over $A_\G[[t]]$ deforming the $\D(\pdiv_{\hat{g}\overline{x}})(\widehat{W}(k[[t]]))$, and hence a $p$-divisible group $\tilde{\pdiv}$ over $A_\G[[t]]$ deforming $\pdiv_{\hat{g}\overline{x}}$.
  
  Let $\varpi:A_\G[[t]]\rightarrow \Ok_K$ be any map and $\tilde{\pdiv}_{\varpi}$ the $p$-divisible group over $\Ok_K$ obtained by pullback. By construction we have an isomorphism $$\iota:\D(\tilde{\pdiv}_{\varpi})(\widehat{W}(\Ok_K))\cong \D\otimes_{\Ok_L}\widehat{W}(\Ok_K)$$ thus $s_{\alpha,0,\overline{x}}$ give $\Phi$-invariant tensors in $\D(\tilde{\pdiv})(\widehat{W}(\Ok_K))^\otimes$. Moreover, under the canonical identification $\D(\tilde{\pdiv}_\varpi)(\Ok_K)\otimes_{\Ok_K}K\cong\D\otimes_{\Ok_L}K$, the filtration is induced by $G$-valued cocharacter conjugate to $\mu_y$. Indeed the composition $$\D\otimes_{\Ok_L}\Ok_K\xrightarrow\iota \D(\tilde{\pdiv}_\varpi)(\Ok_K)\otimes_{\Ok_K}K\xrightarrow\sim \D\otimes_{\Ok_L}K$$ where the second map is the canonical isomorphism takes $s_{\alpha,0,\overline{x}}$ to itself. Since the filtration on the left is induced by the map $A\rightarrow A_\G\rightarrow A_\G[[t]]\rightarrow\Ok_K$, it corresponds to a point of the local model $M^{loc}_\G$ hence is induced by a $G$-valued cocharacter conjugate to $\mu_y$. Thus $\tilde{\pdiv}_{\varpi}$ is $(\G_{\Ok_L},\mu_y)$-adapted as desired.
  
 Let $\hat{U}'_{\overline{x}}$ denote the completion of $\mathscr{S}_{\rmK'}(GSp,S^\pm)$ at the image of ${\overline{x}}$. Then the $p$-divisible group $\tilde{\pdiv}$ corresponds to a map $\epsilon^-:\text{Spf}A_\G[[t]]\rightarrow\hat{U}'_{\overline{x}}$. Let $\hat{Z}\subset \hat{U}'_x$ denote the completion at $\overline{x}$ of the irreducible component containing $x$ (recall $x\in \mathscr{S}_{\rmK}(G,X)(K)$ was a point lifting $\overline{x}$). By the previous paragraph, for any $\varpi:A_\G[[t]]\rightarrow \Ok_K$, the induced point of $\hat{U}'_x$ lies in $\hat{Z}$ by Proposition \ref{prop12} iii). Since this is true for any $\Ok_K$ point, $\epsilon^-$ factors though $\hat{Z}$. Thus $i_R$ factors through $\mathscr{S}^-_{\rmK}(G,X)$.
 
Since $A_{\G}[[t]]$ is normal, the map $\epsilon^-:\Spec A_{\G}[[t]]\rightarrow \mathscr{S}^-_{\rmK}(G,X)$ lifts to $${\epsilon}: \Spec A_{\G}[[t]]\rightarrow\mathscr{S}_{\rmK}(G,X)$$ and this lift is unique since $\mathscr{S}_{\rmK}(G,X)\rightarrow \mathscr{S}^-_{\rmK}(G,X)$  is an isomorphism on the generic fibre. We write $\overline{\epsilon}$ for the induced map $$\overline{\epsilon}:\Spec k[[t]]\rightarrow \mathscr{S}_{\rmK}(G,X)$$ We thus have a diagram:

\[\xymatrix{\Spec k[[t]]\ar[r]^{\overline{\epsilon}}\ar[d] & \mathscr{S}_{\rmK}(G,X)\ar[d]\\
	\Spec R\ar[r]^{i_R}& \mathscr{S}^-_{\rmK}(G,X)}\]

Since $R$ is integral, the map $R\rightarrow k[[t]]$ is injective, in particular, the above diagram induces:

\[\xymatrix{\Spec \text{Frac}(R)\ar[r]\ar[d] & \mathscr{S}_{\rmK}(G,X)\ar[d]\\
\Spec R\ar[r]^{i_R}& \mathscr{S}^-_{\rmK}(G,X)}\]

 By Lemma \ref{lemma2}, there exists a unique lift $i_R:\Spec R\rightarrow \mathscr{S}_{\rmK}(G,X)$.
 
 To show the compatibility of this map with the tensors $s_{\alpha,0,\overline{x}}$, we let $\mathbb{M}$ denote the Dieudonn\'e $F$-crystal over $\mathscr{S}_{\mathrm{K}_p}(G,X)_k$ associated to the universal $p$-divisible group, and $\mathbb{M}[\frac{1}{p}]$ the corresponding $F$-isocrystal. We have by \cite[Corollary A.7]{KMS}, there exists sections:
 $$\mathbf{s}_{\alpha,0}:1\rightarrow \mathbb{M}[\frac{1}{p}]^\otimes$$
 such that for all $\overline{x}'\in\mathscr{S}_{\mathrm{K}_p}(G,X)(k)$, $\mathbf{s}_{\alpha,0}$ pulls back to $s_{\alpha,0,\overline{x}'}\in\D(\pdiv_{\overline{x}'})[\frac{1}{p}]^\otimes$. 
 
 Thus pulling back to $\Spec R$, we obtain $\mathbf{s}_{\alpha,0,R}\in\D(\pdiv_{g\overline{x}})(\mathscr{R})[\frac{1}{p}]^\otimes$ such that for all $z:R\rightarrow k$, the pullback coincides with $s_{\alpha,0,\iota_R(z)}$. Now by construction $s_{\alpha,0,\overline{x}}\in \D(\pdiv_{\hat{g}\overline{x}})(\mathscr{R})^\otimes$ are parallel for the connection and coincide with $\mathbf{s}_{\alpha,0,R}$ at the point $x_R$. Hence since $R$ is integral, we have $\mathbf{s}_{\alpha,0,R}=s_{\alpha,0,\overline{x}}$.

\end{proof}

\begin{lemma}\label{lemma2}
Let $Y$ be a reduced scheme and $Y^n$ it's normalization. Let $X$ be a normal integral scheme with generic point $\Spec K(X)$. Suppose we have a diagram:
\[\xymatrix{ \Spec K(X)\ar[r] \ar[d] & Y^n\ar[d]\\
X\ar[r]  & Y}\]
Then $f$ lifts to a unique map $f':X\rightarrow Y^n$
\end{lemma}
\begin{proof}By uniqueness we may assume $X=\Spec R$ and $Y=\Spec S$ is affine, then $Y^n=\Spec S^{int}$ where $S^{int}$ is the integral closure of $S$ in $K(Y):=\text{Frac}(S)$. Thus it suffices to show the induced map $S^{int}\rightarrow K(X)$ factors through $R$. But this follows since $R$ is integrally closed in $K(X)$.
\end{proof}

\begin{proof}[Proof of Proposition 6.4]By uniqueness, there is a maximal subset $X(\sigma(\{\mu_y\},b)^\circ\subset X(\sigma(\{\mu_y\}),b)$ which lifts to a map:

$$i_x:X(\sigma(\{\mu_y\}),b)^\circ\rightarrow \mathscr{S}_{\rmK}(G,X)(k)$$

By Proposition \ref{prop8}, $X(\sigma(\{\mu_y\}),b)^\circ$ is a union of connected components.

Assume first that $b$ is basic. By Theorem \ref{thm1} i) there exists $g_0\in X_{\sigma(\tau_{\{\mu_y\}})}(b)\cap X(\sigma(\{\mu_y\}),b)^\circ$, i.e. $i_{\overline{x}}(g_0)$ lifts to a point $\overline{x}'\in\mathscr{S}_{\mathrm{K}_p}(G,X)(k)$. We may apply the above construction with $\overline{x}$ replaced by $\overline{x}'$. Then $b$ is replaced by $b':=g_0^{-1}b\sigma(g_0)$, and we have a map $$i'_{\overline{x}'}:X(\sigma(\{\mu_y\}),b')\rightarrow \mathscr{S}_{\rmK'}(GSp,S^\pm)(k)$$ which is identified with the map $i_{\overline{x}}$ under the identification $$X(\sigma(\{\mu_y\}),b')\cong X(\sigma(\{\mu_y\}),b)$$ given by $h\mapsto g_0^{-1}h$. Therefore upon replacing $\overline{x}$ by $\overline{x}'$ we may assume $b\in \G(\Ok_L)\sigma(\dot{\tau}_{\{\mu\}})\G(\Ok_L)$. By Theorem \ref{thm2}, upon changing the isomorphism $$V_{\Z_p}^*\otimes_{\Z_p}\Ok_L\cong \D(\pdiv_{\overline{x}})$$ 
by an element of $\G(\Ok_L)$, we may assume $b=\dot{\tau}_{\{\mu\}}$. 

Let $x\in\mathscr{S}_{\rmK}(G,X)(\Ok_K)$ denote a lift of $\overline{x}$ with $K/L$ finite. Fix the isomorphism $$T_p\pdiv_x^\vee\otimes_{\Z_p}\Ok_L\cong\D(\pdiv_{\overline{x}})$$ taking $s_{\alpha,\acute{e}t,x}$ to $s_{\alpha,0,\overline{x}}$ compatibly with the isomorphism $V_{\Z_p}^*\otimes_{\Z_p}\Ok_L\cong \D(\pdiv_{\overline{x}})$ above. We may now apply the construction of \S5.

Let $g\in G(\Q_p)/\G(\Z_p)$ and $g_0$ the corresponding element in $X(\sigma(\{\mu\}),b)$ constructed in \S5.6; upon replacing $K$ by a finite extension, $g^{-1}T_p\pdiv_x$ corresponds to a $p$-divisible group $\pdiv'$ over $\mathcal{O}_K$ together with a quasi-isogeny $\pdiv'\rightarrow \pdiv_x$ which identifies $\D(\pdiv'_0)$ with $ g_0\D(\pdiv_{\overline{x}})$. This corresponds to a quasi-isogeny $\mathcal{A}'\rightarrow \mathcal{A}_x$ which respect $s_{\alpha,\acute{e}t,x}$ and hence to a point $gx\in Sh_{\rmK}(G,X)(K)$ which lifts $i'_{\overline{x}}(g_0)$.  Therefore $g_0\in X(\sigma(\{\mu_y\}),b)^\circ$. 

By Proposition \ref{prop1}, $g\mapsto g_0$ induces a surjection $$G(\Q_p)/\G(\Z_p)\twoheadrightarrow \pi_0(X(\sigma(\{\mu_y\}),b)),$$ hence $$X(\sigma(\{\mu_y\}),b)^\circ=X(\sigma(\{\mu_y\}),b)$$ This proves the proposition when $b$ is basic in $G$.

Now assume $G_{\Q_p}$ is residually split. In this case $\sigma$ acts trivially on $W$, hence $\Adm(\{\mu_y\})=\sigma(\Adm(\{\mu_y\}))$. Recall by Theoream \ref{bound}, there is a map

\begin{equation}\label{map}\coprod_{w\in W, w \text{ a straight element with } \dot{w}\in[b]}X^{M_{\nu_w}}(\{\lambda_w\}_{M_{\nu_w}},\dot{w})\rightarrow X(\{\mu_y\},b)\end{equation}
which induces a surjection
$$\coprod_{w\in W, w \text{ a straight element with } \dot{w}\in[b]}\pi_0(X^{M_{\nu_w}}(\{\lambda_w\}_{M_{\nu_w}},\dot{w}))\rightarrow \pi_0(X(\{\mu_y\},b))$$

It thus suffices to show for each $w\in\Adm(\{\mu\})$ straight, that each connected component of $X^{M_{\nu_w}}(\{\lambda_w\}_{M_{\nu_w}},\dot{w})$ contains an element whose image in $X(\{\mu_y\},b)$ lifts to $\mathscr{S}_{\rmK}(G,X)(k)$ satisfying the above properties. 

\textit{Step 1}: Recall by our assumptions $1\in X(\{\mu_y\},b)$. There exists a straight element $w$ and $g_0\in X_w(b)$ with $1\sim g_0$. By Proposition \ref{prop3}, $i_{\xbar}'(g_0)$ lifts to a point $\overline{x}'\in \mathscr{S}_{\rmK}(G,X)(k)$. As above, upon replacing $\overline{x}$ by $\overline{x}'$, we may assume $b\in \G(\Ok_L)\dot{w}\G(\Ok_L)$. By Theorem \ref{thm2}, we may change the isomorphism $V_{\Z_p}^*\otimes_{\Z_p}\Ok_L\cong \D(\pdiv_{\overline{x}})$ by an element of $\G(\Ok_L)$ and assume  $b=\dot{w}$.

Let $M:=M_{\nu_w}$ denote the semistandard Levi subgroup corresponding to $\nu_w$. Then $M(L)\cap \G(\Ok_L)$ is an Iwahori subgroup of $M$ which is defined over $\Z_p$, we write $\mathcal{M}$ for the associated group scheme. Let $W_M$ denote the Iwahori Weyl group of $M$, then $w\in W_M$ and the main result of \cite{Nie} implies $\dot{w}\mathcal{M}(\Ok_L)\dot{w}^{-1}=\mathcal{M}(\Ok_L)$. We equip $W_M$ with the Bruhat order $\leq_M$ induced by the Iwahori subgroup $\mathcal{M}$, then $w$ is a basic element of $W_M$. If $w=w_0t_{\underline{\lambda}}$ for some $w_0\in W_{M,0}$ the relative Weyl group and $\underline{\lambda}\in X_*(T)_I$, we have $t_{\underline{\lambda}}\in \Adm(\{\mu\})$ by \cite[7.1 b)]{HZ}. By Lemma \ref{cochar}, there exists a cocharacter $\lambda_w\in X_*(T)$ which lifts $\underline{\lambda}$ and whose image in $G$ conjugate to $\mu_y$. Thus since $w\leq_Mt_{\underline{\lambda}}$, we have $w\in \Adm^M(\{\lambda_w\})$. Therefore by Proposition \ref{prop4}, there is an $M$ valued cocharacter conjugate to $\lambda_w$ such that the induced filtration on $\D\otimes_{\Ok_L}K$ specializes to the one on $\D(\pdiv_{\overline{x}})\otimes_{\Ok_L}k$. We may thus apply the construction of \S5.4.

Since $\dot{w}\in M(L)$, we may extend the tensors $s_{\alpha}$ to tensors $t_{\alpha}\in V_{\Z_p}^\otimes$ whose stablizer is the Iwahori $\mathcal{M}$. We obtain an embedding of local models $M^{loc}_{\mathcal{M},\lambda_w^{-1}}\subset M^{loc}_{\mathcal{GL},\mu_y^{-1}}\otimes\mathcal{O}_{E'}$, where $E'$ is the (local) reflex for the $M$-conjugacy class of $\{\lambda_w\}$. Since $\dot{w}\in \Adm^M(\{\lambda_w\})$ the filtration on $\D(\pdiv_{\overline{x}})\otimes k$ gives a point in the local model $M^{loc}_{\mathcal{M}}(k)$. Replacing $\lambda_w$ by an element in its $M$-conjugacy class, we may assume $\lambda_w$ is defined over a finite extension $K/L$ and the induced filtration lifts the filtration on $\D(\pdiv_{\overline{x}})$. 

Let ${\pdiv}$ denote an $(\mathcal{M}_{\Ok_L},\lambda_w)$-adapted lifting of $\pdiv_{\overline{x}}$ satisfying the conditions in  Proposition \ref{adapted}. Note that any $(\mathcal{M}_{\Ok_L},\lambda_w)$-adapted lifting is also $(\G_{\Ok_L},\mu_y)$-adapted, hence corresponds to a point $x\in\mathscr{S}_{\rmK}(G,X)(\Ok_K)$. We may thus apply the construction of \S5 to $\pdiv_{\overline{x}}$, and we obtain a map $$M(\Q_p)/\mathcal{M}(\Z_p)\rightarrow X^M(\{\lambda_w\},\dot{w})$$ Since $\dot{w}$ is basic in $M(L)$, we have by Proposition \ref{prop1} that this induces a surjection
$$M(\Q_p)/\mathcal{M}(\Z_p)\rightarrow \pi_0(X^M(\{\lambda_w\},b))$$
It follows that $X(\{\mu_y\},b)^\circ$ contains the image of  $X^{M_{\nu_w}}(\{\lambda_w\},\dot{w})$.

\textit{Step 2}: Now suppose there exists   $s\in\mathbb{S}$ such that $w'=sws$ and $l(w')=l(w)$. Assume $b=\dot{w}$ for a straight element $w$, then $w'$ is also a straight element and $w'\in\Adm(\{\mu_y\})$ by \cite[Lemma 4.5]{Ha}. Thus $\dot{s}\in X(\{\mu_y\},b)$. We show that $\dot{s}\in X(\{\mu_y\},b)^\circ$. Recall ${\pdiv}$ is an $(\mathcal{M}_{\Ok_L},\lambda_w)$-adapted lifting. We fix the isomorphism $$T_p{\pdiv}^\vee\otimes_{\Z_p}\Ok_L\cong \D({\pdiv}_{\overline{x}})$$ taking $t_{\alpha,\acute{e}t}$ to $t_{\alpha,0,\overline{x}}$. Upon replacing $K$ by a finite extensions, we have $\dot{s}T_p{\pdiv}$ corresponds to a $p$-divisible group ${\pdiv}'$ over $\Ok_K$ equipped with a quasi isogeny ${\pdiv}'\rightarrow {\pdiv}$. This identifies $\mathfrak{M}(T_p{\pdiv}'^\vee)$ with $\tilde{s}\mathfrak{M}(T_p{\pdiv}^\vee)$ for some $\tilde{s}\in G(\mathfrak{S}[1/p])$.
 We also obtain a quasi-isogeny ${\pdiv}'\rightarrow \pdiv$ over $k$ which identifies $\D(\pdiv')$ with $s_0\D(\pdiv)$ where $s_0=\sigma^{-1}(\tilde{s})|_{u=0}$. By Proposition \ref{key} we have $s=m\tilde{s} g$ where $m\in\mathcal{M}(\Ok_{\widehat{\mathscr{E}^{ur}}})$ and $g\in\mathcal{G}(\Ok_{\widehat{\mathscr{E}^{ur}}})$. Using the natural map $L\rightarrow \mathfrak{S}[1/p]$, we may consider $\dot{w}\in G(L)$ as an element of $G(\mathfrak{S}[1/p])$. Then we have $$\tilde{s}^{-1}\dot{w}\sigma(\tilde{s})=g^{-1}sm^{-1}\dot{w}\sigma(m)s\sigma(g)$$
Without loss of generality assume $l(sw)=l(w)+1$. Since $w$ is basic in $W_M$, we have $$m':=m^{-1}\dot{w}\sigma(m)\dot{w}^{-1}\in \mathcal{M}(\Ok_{\widehat{\mathscr{E}^{ur}}})$$ Thus since $\mathcal{M}(\Ok_{\widehat{\mathscr{E}^{ur}}})\subset \mathcal{G}(\Ok_{\widehat{\mathscr{E}^{ur}}})$ and $l(sw)=l(w)+1$, we have $$\tilde{s}\dot{w}\sigma(\tilde{s})=g^{-1}sm'\dot{w}s\sigma(g)\in\G(\Ok_{\widehat{\mathscr{E}^{ur}}})s\dot{w}s\G(\Ok_{\widehat{\mathscr{E}^{ur}}})$$
We consider $\tilde{s}\dot{w}\sigma(\tilde{s})$ as a $k[[u]]^{perf}$ point of $\mathcal{FL}$. The above calculation shows that the generic fiber of this point lies in the the Schubert variety $S_{w'}\subset \mathcal{FL}$. Since the Schubert variety $S_{w'}$ is closed, the special fiber also lies in $S_{w'}$. Hence we have $$s_0\dot{w}\sigma(s_0)\in \G(\Ok_L)\dot{w}''\G(\Ok_L)$$ 
for some $w''\leq w'$. By Lemma \ref{lemma3} below, we have $w'=w''$ and $s_0\in X_{w'}(b)$. As above, this implies $s_0\in X(\{\mu_y\},b)^\circ$. Upon replacing $\overline{x}$ by $i_{\overline{x}}(s_0)$, and applying step 1 to $i_{\overline{x}}(s_0)$, we obtain $X^{M_{\nu_{w'}}}(\{\lambda_{w'}\},\dot{w}')\subset X(\{\mu\},b)^\circ$.

{\it Step 3:} Suppose $b=\dot{w}$ is a straight element and $\tau\in\Omega$. Then $\tau w\tau^{-1}\in\Adm(\{\mu\})$ again by \cite[Lemma 4.5]{Ha}. Therefore $\dot{\tau}\in X(\{\mu\},b)$. As before let $\tilde{\pdiv}$ be an $(\mathcal{M}_{\Ok_L},\lambda_w)$-adapted lifting of $\pdiv_{\overline{x}}$ to $\Ok_K$ and apply the construction of Proposition \ref{lifting} to $\dot{\tau}$, we obtain an element $\tilde{\tau}=m\dot{s}g\in G(\mathfrak{S}[\frac{1}{p}]$ with $m\in\mathcal{M}(\Ok_{\widehat{\mathscr{E}^{ur}}})$ and $g\in\mathcal{G}(\Ok_{\widehat{\mathscr{E}^{ur}}})$. Since $\dot{\tau}\mathcal{G}(\Ok_{\widehat{\mathscr{E}^{ur}}})\dot{\tau}^{-1}=\mathcal{G}(\Ok_{\widehat{\mathscr{E}^{ur}}})$, we have 

$$\tilde{\tau}^{-1}\dot{w}\tilde{\tau}=g^{-1}\dot{\tau}^{-1}m^{-1}\dot{w}\sigma(m)\dot{\tau}\sigma(g)\in\mathcal{G}(\Ok_{\widehat{\mathscr{E}^{ur}}})\dot{\tau}^{-1}\dot{w}\dot{\tau}\mathcal{G}(\Ok_{\widehat{\mathscr{E}^{ur}}})$$

As in Step 2, this implies $\dot{\tau}_0\in X_{\tau^{-1}w\tau}(\dot{w})\cap X(\{\mu\},\dot{w})^\circ$.

{\it Step 4:} For any two straight element $w,w'$, we have $w\sim w'$, so that we have a sequence $s_1,\dots,s_n\in\mathbb{S}$ and straight elements $w=w_0,\dots,w_n\in W$ such that $w_i\sim_{s_{i+1}}w_{i+1}$ and $\tau^{-1}w_n\tau=w'$ for some $\tau\in\Omega$. Applying Step 2 to each $w_i$ in turn and Step 3 to $w_n$, we see that $X(\{\mu\},b)^\circ$ contains an element of $g_0\in X_{w'}(b)$. Applying Step 1 to the lift of $\iota(g_0)$ in $\mathscr{S}_{\mathrm{K}_p}(G,X)(k)$, we have $X(\{\mu\},b)^\circ$ contains the image of $X^{M_{\nu_{w'}}}(\{\lambda_{w'}\},\dot{w}')$. Hence by Theorem \ref{bound} $X(\{\mu\},b)^{\circ}=X(\{\mu\},b)$.
\end{proof}

\begin{lemma}\label{lemma3}
Let $w\in W$ be a straight element and $g\in G(L)$ such that $g^{-1}\dot{w}\sigma(g)\in \G(\Ok_L)\dot{w}'\G(\Ok_L)$ with $w'\leq w$. Then $w'=w$.
\end{lemma}
\begin{proof}By \cite[\S 3]{He1}, there exists $w''\in W$ $\sigma$-straight such that $l(w'')\leq l(w')$ and $[g^{-1}\dot{w}\sigma(g)]=[\dot{w}'']$. By \cite[Theorem 3.7]{He1} $w$ and $w''$ lie in the same $\sigma$-conjugacy class in $W$, in particular $l(w'')=l(w)$. Thus $l(w)=l(w')$ and since $w'\leq w$, we have $w=w'$.\end{proof}
\subsection{} Recall the local model diagram:

$$\mathscr{S}_{\rmK_p}(G,X)\xleftarrow{q}\tilde{\mathscr{S}}_{\rmK_p}(G,X)\xrightarrow\pi M_{\G}^{loc}$$
This induces the Kottwitz Rapoport stratification on the geometric special fiber $\mathscr{S}_{\mathrm{K}_p}(G,X))_k$. We write $\mu$ for a dominant representative of $\mu_h^{-1}$ in  $X_*(T)$, then we have a map $$\lambda:\mathscr{S}_{\rmK_p}(G,X)(k)\rightarrow \Adm(\{\mu\})$$
Let $\overline{x}\in\mathscr{S}_{\rmK_p}(G,X)(k)$ and $i_{\overline{x}}:X(\sigma(\{\mu\}),b)\rightarrow\mathscr{S}_{\rmK_p}(G,X)(k)$ the map defined in Proposition \ref{map} when $G_{\Q_p}$ is residually split or $b$ is basic (here we use that $\mu_y$ and $\mu$ lie in the same $G$ conjugacy class so the associated admissible sets coincide).
\begin{prop}\label{KR}
Let $g\in X_w(b)$ for some $w\in\sigma(\Adm(\{\mu\}))$. Then $\lambda(i_{\overline{x}}(g))=\sigma^{-1}(w)$
\end{prop}
\begin{proof}Recall how the map $\lambda$ is defined. Let $x\in \mathscr{S}_{\rmK_p}(G,X)(\Ok_K)$ be a point lifting $\xbar$. The torsor $\tilde{\mathscr{S}}_{\rmK_p}(G,X)$ is constructed by taking triviliazations of the relative de Rham cohomology which respects the cycles $s_{\alpha,dR}$. We have an isomoprhism $$V_{\Z_p}^*\otimes_{\Z_p}\Ok_K\cong \D(\pdiv_x)(\Ok_K)$$ taking $s_{\alpha}$ to $s_{\alpha,dR,x}$. We have the local model $M_\G^{loc}\subset M^{loc}_{\mathcal{GL}}\otimes\Ok_E$, where $M^{loc}_{\mathcal{GL}}$ classifies sub-modules of $V_{\Z_p}^*$. The pullback of the Hodge filtration on $\D(\pdiv_x)(\Ok_K)$ which lies in the local model $M^{loc}(\Ok_K)$. We obtain a point $\tilde{x}\in M^{loc}_\G(k)$ which lies in $\underline{\G}(k[[t]])\dot{w}/\underline{\G}(k[[t]])$\footnote{Here $\dot{w}$ denotes a lift of $w$ to $\underline{G}(k((t)))$ via the identification of Iwahori Weyl groups for $G$ and $\underline{G}_{k((t))}$} for some $w\in\Adm(\{\mu\})$. Then $\lambda(\xbar)=\dot{w}$.

There is an isomorphism $\D(\pdiv_x)(\Ok_K)\cong \D(\pdiv_{\xbar})\otimes \Ok_K$ lifting the identity mod$p$ and taking $s_{\alpha,dR,x}$ to $s_{\alpha,0,\xbar}$. Thus if we fix an isomorphism $$V_{\Z_p}^*\otimes\Ok_L\cong\D(\pdiv_{\xbar})$$ taking $s_{\alpha}$ to $s_{\alpha,0,\xbar}$, then the pullback of the filtration on $\D(\pdiv_{\xbar})(k)$ to $V_{\Z_p}^*$ differs from the one above by translation by an element of $\G(\Ok_K)$. We thus obtain a point on $\tilde{x}'\in M^{loc}_\G(k)$ which lies in the Schubert variety $S_w$. Thus $\lambda(\overline{x})$ can also be computed by a trivialization of $\D(\pdiv_x)$.

Now fix an isomorphism $V_{\Z_p}^*\otimes \Ok_L\cong\D(\pdiv_{\xbar})$. Let  $g\in X_w(b)$, then $\D(\pdiv_{i_{\xbar}(g)})$ is identified with $g\D(\pdiv_{\xbar})$. We may trivialize $\D(\pdiv_{i_{\xbar}(g)})\cong V_{\Z_p}^*\otimes \Ok_L$ by composing the trivilization $V_{\Z_p}^*\otimes \Ok_L\D(\pdiv_{\xbar})$ with the element $g$. The filtration mod$p$ on $V_{\Z_p}^*$ is then induced by the element $g^{-1}b\sigma(g).$ By the identification of apartments and Iwahori Weyl group in \S3.3, this filtration corresponds to a point $M^{loc}_\G(k)$ which lies in $\underline{\G}(k[[t]])\sigma^{-1}(\dot{w})/\underline{\G}(k[[t]])$. Hence $\lambda(i_{\xbar}(g))=w$.
\end{proof}
\subsection{}Now assume $G_{\Q_p}$ is residually split or $b$ is basic. Since $\mathscr{S}_{\rmK_p}(G,X)$ is equipped with an action of $G(\mathbb{A}^p_f)$, $i_{\xbar}$ extends to a map:
$$i_{\xbar}:X(\sigma(\{\mu\}),b)\times G(\mathbb{A}^p_f)\rightarrow \mathscr{S}_{\rmK_p}(G,X)(k)$$

As in \cite[Corollary 1.4.13]{Ki2}, this map is equivariant for the action of $\Phi\times Z_G(\Q_p)\times G(\A_f^p)$.

\begin{definition}
Let $\overline{x},\overline{x}'\in\mathscr{S}_{\rmK_p}(G,X)(k)$. We say $\overline{x}$ and $\overline{x}'$ are in the same isogeny class if there exists a quasi-isogeny $\mathcal{A}_{\overline{x}}\rightarrow \mathcal{A}_{\overline{x}'}$ respecting weak polarizations such that the induced maps $\D(\pdiv_{\overline{x}'})\rightarrow \D(\pdiv_{\overline{x}})$ and $\widehat{V}^p(\mathcal{A}_{\overline{x}})\rightarrow \widehat{V}^p(\mathcal{A}_{\overline{x}'})$ take $s_{\alpha,0,\overline{x}}$ to $s_{\alpha,0,\overline{x}'}$  and $\{s_{\alpha,l,\overline{x}}\}_{l\neq p}$ to $\{s_{\alpha,l,\overline{x}'}\}_{l\neq p}$.
\end{definition}

\begin{prop}\label{isog}Assume $G_{\Q_p}$ is residually split or $b$ is basic and $\G$ is an Iwahori subgroup. Then
$\overline{x}$ and $\overline{x}'$ lie in the same isogeny class if and only if $\overline{x}'$ lies in the  image of $$i_{\overline{x}}:X(\{\mu\},b)\times G(\A_f^p)\rightarrow\mathscr{S}_{\rmK_p}(G,X)(k)$$
\end{prop}
\begin{proof} Suppose $\overline{x}$ and $\overline{x}'$ lie in the same isogeny class. The composition

$$V_{\A_f^p}\xrightarrow[\epsilon_{\overline{x}}]{\sim}\widehat{V}(\mathcal{A}_{\overline{x}})\xrightarrow{\sim} \widehat{V}(\mathcal{A}_{\overline{x}'})\xrightarrow[\epsilon_{\overline{x}'}]{\sim} V_{\A_f^p}$$
takes $s_{\alpha}$ to $s_{\alpha}$, hence upon replacing $\overline{x}'$ by a translate under $G(\A_f^p)$, we may assume the quasi isogeny $\mathcal{A}_{\overline{x}}\rightarrow \mathcal{A}_{\overline{x}'}$ is compatible with $\epsilon_{\overline{x}}$ and $\epsilon_{\overline{x}'}$. 

Recall there are isomorphisms $$\D(\pdiv_{\overline{x}})\xrightarrow\sim V^*\otimes_{\Z}\mathcal{O}_L\xrightarrow\sim \D(\pdiv_{\overline{x}}')$$ taking $s_{\alpha,0,\overline{x}}$ to $s_{\alpha,0,\overline{x}'}$. Thus $\D(\pdiv_{\overline{x}})$ corresponds to $g\D(\pdiv_{\overline{x}})$ for some $g\in G(L)$. By the same proof as Proposition \ref{lifting}, we have $g\in X(\{\mu\},b)$ hence $\overline{x}'$ lies in the image of $i_{\xbar}$.

The converse is clear.
\end{proof}

\section{Maps between Shimura varieties}
In this section we show that the Shimura varieties associated to different parahorics levels admit maps between them with good properties. This will allow us to deduce the description of the isogeny classes for general parahorics from the result for Iwahori subgroups proved  in the previous section. This also verifies one of the axioms of \cite{HR} for integral models of Shimura varieties with parahoric level.

\subsection{}We keep the notations from the previous section, so that $\rho:G\rightarrow GSp(V,\psi)$ is a hodge embedding. Let $\rmK_p$ be a connected parahoric subgroup of $G(\Q_p)$ and let $\G$ denote the corresponding group scheme over $\Z_p$. Let $\mathrm{K}_p'\subset G(\Q_p)$ be another connected parahoric subgroup such that $\rmK_p\subset \rmK_p'$. If $\rmK_p$ and $\rmK_p'$ have corresponding facets $\mathfrak{f}$ and $\mathfrak{f}'$, then this is equivalent to $\mathfrak{f}$ lying in the closure of $\mathfrak{f}'$. 

By the construction in the previous section we have integral models $\mathscr{S}_{\rmK'}(G,X)$ and $\mathscr{S}_{\rmK}(G,X)$ over $\Ok_{E_{(v)}}$, where $\rmK'=\rmK_p'\rmK^p$ and $\rmK=\rmK_p\rmK^p$ for some sufficiently small $\rmK^p$. 

\begin{thm}\label{thm3}
i) For sufficiently small $\rmK^p$, there exists a map $$\pi_{\rmK_p,\rmK_p'}:\mathscr{S}_{\rmK}(G,X)\rightarrow\mathscr{S}_{\rmK'}(G,X)$$

ii) The induced  map $$\mathscr{S}_{\rmK}(G,X)(k)\rightarrow\mathscr{S}_{\rmK'}(G,X)(k)$$ is compatible with isogeny classes.
\end{thm}

\subsection{}Let $\mathfrak{g}$ be a facet in $\B(GSp(V_{\Q_p}),\Q_p)$ and let $\mathcal{GSP}$ denote the associated parahoric. Then $\mathfrak{g}$ corresponds to a lattice chain $\Lambda_1\supset\Lambda_2\supset\dots\supset \Lambda_r$ in $V_{\Q_p}$. Let $V_{\Q_p}'=\oplus_{i=1}^rV_{\Q_p}$, then $V'$ is equipped with an alternating form $\psi'$ given by the direct sum of $\psi$. We have the lattice $\Lambda'=\oplus_{i=1}^r\Lambda_i\subset V'_{\Q_p}$ and we write $\mathcal{GSP}'$ for the associated parahoric. 

We have a map $$GSp(V,\psi)\rightarrow GSp(V',\psi')$$ which factors through the subgroup $H:=\prod_{i=1}^{'r}GSp(V,\psi)$ where $\prod'$ denotes the subgroup of the product $\prod_{i=1}^{r}GSp(V,\psi)$ consisting of elements $(g_1,\dots,g_r)$ such that $c(g_1)=\dots=c(g_2)$, where $$c:GSp(V,\psi)\rightarrow \mathbb{G}_m$$ is the multiplier homomorphism. 

The conjugacy class of cocharacters $S^\pm$ for $GSp(V,\psi)$ gives rise to a $H(\mathbb{R})$ conjugacy class of homomorphisms $T$ from $\mathbb{S}$ into $H_{\mathbb{R}}$ and $(H,T)$ is a Shimura datum. We write $\mathrm{H}_p$ and $\mathrm{J}_p$ for the stabilizer of the lattice $\Lambda'$ in $H(\Q_p)$ and $GSp(V'_{\Q_p})$ respectively. We obtain a map of Shimura varieties $$\iota:Sh_{\mathrm{H}_p\mathrm{H}^p}(H,T)\rightarrow Sh_{\mathrm{J}_p\mathrm{J}^p}(GSp(V',S'^\pm))$$ which is a closed immersion. Here $\mathrm{H}^p\subset H(\A_f^p)$ and $\mathrm{J}^p\subset GSp(V'\otimes \A_f^p)$ are sufficiently small compact open subgroups.

The Shimura variety $Sh_{\mathrm{H}_p\mathrm{H}^p}(H,T)$ admits a moduli interpretation over $\Z_{(p)}$ which we will now explain. For $S$ a $\Z_{(p)}$-scheme, we consider the set of tuples $(\mathcal{A}_i,\lambda_i,\epsilon_{i}^p)_{i=1,\dots,r}$, where: 

i) $\mathcal{A}_i$ is an abelian variety over $S$ up to prime to isogeny.

ii) $\lambda_i$ is polarization such that $\deg\lambda_i$ is exactly divisible by $|\Lambda_{i}/\Lambda_i^*|$.

iii) $\epsilon^p_i: \widehat{V}(\mathcal{A}_i)\xrightarrow\sim V\otimes_{\Q}\A_f^p$ is an isomorphism which takes the Riemann form on $\widehat{V}(\mathcal{A}_i)$ to a multiple of $\psi$ on $V\otimes_{\Q}\A_f^p$. This multiple is required to be independent of $i$.

We obtain an integral model $\mathscr{S}_{\mathrm{H}_p\mathrm{H}^p}(H,T)$ of $Sh_{\mathrm{H}_p\mathrm{H}^p}(H,T)$.

\begin{prop}\label{prop7}
For sufficiently small $\mathrm{J}_p$, the embedding $i$ extends to a closed embedding $$\iota:\mathscr{S}_{\mathrm{H}_p\mathrm{H}^p}(H,T)\rightarrow \mathscr{S}_{\mathrm{J}_p\mathrm{J}^p}(GSp(V'),S'^\pm)$$
\end{prop}

\begin{proof}Using the moduli interpretations, we may define $\iota$  by sending $(\mathcal{A}_i,\lambda_i,\epsilon_i^p)_{i=1,\dots,r}\in \mathscr{S}_{\mathrm{H}_p\mathrm{H}^p}(H,T)(S)$ to the product $\mathcal{A}_1\times\dots\times\mathcal{A}_r$, together with the product polarization and level structure. 

As in \cite[2.1.2]{Ki1}, see also \cite[1, 1.5]{De}, it suffices to show $$\iota:\mathscr{S}_{\mathrm{H}_p}(H,T)\rightarrow \mathscr{S}_{\mathrm{J}_p}(GSp(V'_{\Q_p}),S'^\pm)$$ is a closed immersion. We will show $\iota$ is proper and an injection on points, which implies the result.
The injectivity follows from the moduli interpretations of the integral models.

To check properness, we apply the valuative criterion. Let $R$ be a discrete valuation ring with fraction field $K$. We must show for any diagram

\[\xymatrix{\Spec K \ar[r] \ar[d] &\mathscr{S}_{\mathrm{H}_p}(H,T)\ar[d]\\
\Spec R \ar[r] & \mathscr{S}_{\mathrm{J}_p}(GSp(V'),S'^\pm)}\]
there exists a unique lift $\Spec R\rightarrow \mathscr{S}_{\mathrm{H}_p}(H,T)$. Rephrasing in terms of the moduli interpretation, we must show for a triple  $(\mathcal{A},\lambda,\epsilon^p)$ over $R$, such that over $K$ this data decomposes into a product coming from $(\mathcal{A}_i,\lambda_i,\epsilon_i)_{i=1,\dots,r}$, then the triple over $R$ decomposes. This follows by well known properties of Neron models.
\end{proof}
\subsection{}
Recall we have an embedding of buildings, $i:B(G,\Q_p)\rightarrow B(GSp(V_{\Q_p}),\Q_p)$. Let $\mathfrak{f}$ be a facet in $B(G,\Q_p)$ with associated connected parahoric group scheme $\G$. Let $i(\mathfrak{f})$ be contained in a facet $\mathfrak{g}$ of $B(GSp(V_{\Q_p}),\Q_p)$ corresponding to $\Lambda_1\supset \dots\supset \Lambda_r$. Let $(H,T)$ and $V'$ be as above, we obtain a new embedding of Shimura datum $$(G,X)\rightarrow (GSp(V',\psi'),S'^\pm)$$ which factorises as $$(G,X)\xrightarrow{\rho'} (H,T)\rightarrow (GSp(V',\psi'),S'^\pm)$$ Let $\mathrm{H}_p$ and $\mathrm{J}_p$ be as above and let $\mathrm{H}^p$ and $\mathrm{J}^p$ be as in Proposotion \ref{prop7}. We obtain maps of Shimura varieties

$$Sh_{\rmK_p\rmK^p}(G,X)\rightarrow Sh_{\mathrm{H}_p\mathrm{H}^p}(H,T)\rightarrow Sh_{\mathrm{J}_p\mathrm{J}'^p}(GSp(V'),S^{'\pm})_E$$
and each of these maps are closed immersions. Recall $\mathscr{S}^-_{\rmK_p\rmK^p}(G,X)$ was defined to be the closure of $
Sh_{\rmK_p\rmK^p}(G,X)$ in $\mathscr{S}_{\mathrm{J}_p\mathrm{J}^p}(GSp(V'),S'^\pm)_{\Ok_{E_{(v)}}}$.
\begin{cor}\label{cor4}
$\mathscr{S}^-_{\rmK_p}(G,X)$ is the closure of $Sh_{\rmK_p\rmK^p}(G,X)$ in $\mathscr{S}_{\mathrm{H}_p\mathrm{H}^p}(H,T)$.
\end{cor}
\begin{proof} Immediate from Proposition \ref{prop7}.
\end{proof}

Now suppose  $\mathfrak{f}'$ is a facet of $B(GSp(V_{\Q_p}),\Q_p)$ such that $\mathfrak{f}'$ lies in the closure of $\mathfrak{f}$. Then $\mathfrak{f}'$ corresponds to a lattice chain $\Lambda_{i_1}\subset\dots\subset\Lambda_{i_s}$, where $\{i_1,\dots,i_s\}\subset \{1,\dots,r\}$. Let $(H',T')$ be the Shimura datum obtained from the above construction applied to $\mathfrak{g}'$, i.e. $H'=\prod^{'s}_{j=1}GSp(\oplus_{j=1}^sV_{\Q_p})$, and $\mathrm{H}'_p$ the parahoric of $H'(\Q_p)$ stabilizing $\oplus_{j=1}^s\Lambda_{i_j}$. We obtain a morphism of Shimura data $(H,T)\rightarrow (H',T')$, hence choosing suitable levels $\mathrm{H}_p$ and $\mathrm{H}_p'$ away from $p$, we obtain a morphism of Shimura varieties $$\varpi_{H,H'}:Sh_{\mathrm{H}_p\mathrm{H}^p}(H,T)\rightarrow Sh_{\mathrm{H}'_p\mathrm{H}'^p}(H',T')$$

Using the moduli interpretation, this extends in the obvious way to a morphism of integral models
$$\varpi_{H,H'}:\mathscr{S}_{\mathrm{H}_p\mathrm{H}^p}(H,T)\rightarrow \mathscr{S}_{\mathrm{H}'_p\mathrm{H}'^p}(H',T')$$

\begin{proof}[Proof of Theorem \ref{thm3} i)]Recall $\mathfrak{f'}$ is a facet of $B(G,\Q_p)$ such that \label{key} $\mathfrak{f}$ lies in the closure of $\mathfrak{f}'$. Let $z\in \mathfrak{f}$ and let $z'\in\mathfrak{f}'$ be a point sufficiently close to $z$ such that if $\mathfrak{g}$ and $\mathfrak{g}'$ denotes the facets of $B(GSp(V_{\Q_p}), S',\Q_p)$ containing $i(z)$ and $i(z')$, we have $\mathfrak{g}$ lies in the closure of $\mathfrak{g}'$. Applying the above constructions we obtain a diagram:

\[\xymatrix{\mathscr{S}^-_{\mathrm{K}_p\mathrm{K}^p}(G,X) \ar[r] &\mathscr{S}_{\mathrm{H}_p\mathrm{H}^p}(H,T) \ar[d]^{\varpi_{H,H'}}\\
	\mathscr{S}^-_{\mathrm{K}'_p\mathrm{K}'^p}(G,X) \ar[r] &\mathscr{S}_{\mathrm{H}'_p\mathrm{H}'^p}(H',T')	}\]

On the generic fiber, this can be completed to a diagram

\[\xymatrix{Sh_{\mathrm{K}_p\mathrm{K}^p}(H,T)\ar[d] \ar[r] &Sh_{\mathrm{H}_p\mathrm{H}^p}(G,X)\ar[d]^{\varpi_{H,H'}}\\
	Sh_{\mathrm{K}'_p\mathrm{K}'^p}(G,X) \ar[r] &Sh_{\mathrm{H}'_p\mathrm{H}'^p}(H',T')	}\]
hence by Corollary \ref{cor4}, we obtain a map $\mathscr{S}^-_{\rmK_p\rmK^p}(G,X)\rightarrow \mathscr{S}^-_{\rmK'_p\rmK'^p}(G,X)$, and taking normalizations we obtain:

$$\pi_{\rmK_p,\rmK_p'}:\mathscr{S}_{\rmK_p\rmK^p}(G,X)\rightarrow \mathscr{S}_{\rmK'_p\rmK'^p}(G,X)$$

\end{proof}
The above maps then induce by passage to the limit, a map between the pro-varieties $$\pi_{\mathrm{K}_p,\mathrm{K}'_p}:\mathscr{S}_{\mathrm{K}_p}(G,X)\rightarrow \mathscr{S}_{\mathrm{K}'_p}(G,X)$$
\subsection{}Now we relate the isogeny classes on $\mathscr{S}_{\rmK_p}(G,X)_k$ and $\mathscr{S}_{\rmK'_p\mathrm{K}^p}(G,X)_k$. Let $\xbar\in\mathscr{S}_{\rmK_p\rmK^p}(G,X)(k)$ and ${\overline{y}}=\pi_{\rmK_p,\rmK'_p}(\overline{x})$. Let $\mathscr{I}_{\overline{x}}$ and $\mathscr{I}_{\overline{y}}$ denote the isogeny classes of $\overline{x}$ and ${\overline{y}}$. Then $\overline{x}$ corresponds to a collection $(\mathcal{A}_i,\lambda_i,\epsilon_i^p)_{i=1,\dots,r}$ and ${\overline{y}}$ corresponds to $(\mathcal{A}_{i_j},\lambda_{i_j},\epsilon_{i_j}^p)_{j=1,\dots,s}$ and we have inclusion and projection maps $$\iota_0:\oplus_{j=1}^s\D(\pdiv_{i_j})\rightarrow \oplus_{i=1}^r\D(\pdiv_{i})\ \ \ \ \ \ \ \ \pi_0:\oplus_{i=1}^r\D(\pdiv_{i})\rightarrow\oplus_{j=1}^s\D(\pdiv_{i_j})$$
and for $l\neq p$
$$\iota_l:\oplus_{j=1}^s(T_l\mathcal{A}_{i_j})^*\rightarrow \oplus_{i=1}^r(T_l\mathcal{A}_i)^*\ \ \ \ \ \ \ \ \pi_l:\oplus_{i=1}^r(T_l\mathcal{A}_i)^*\rightarrow\oplus_{j=1}^s(T_l\mathcal{A}_{i_j})^*$$

We write $G'_{\Z_{(p)}}$ and $G_{\Z_{(p)}}$ the groups over $\Z_{(p)}$ given by the Zariski closures of $G$ in $GL(\oplus_{j=1}^s\Lambda_{i_j,\Z_{(p)}})$ and $GL(\oplus_{i=1}^r\Lambda_{i,\Z_{(p)}})$. Here we write $\Lambda_{i,\Z_{(p)}}$ for the $\Z_{(p)}$ module $V\cap\Lambda_i$. Then $G'_{\Z_{(p)}}$ is the stabilizer of a collection of tensors $s_{\alpha}\in (\oplus_{j=1}^s\Lambda_{i_j,\Z_{(p)}})^\otimes$. We have the two maps $$\iota: \oplus_{j=1}^s\Lambda_{i_j,\Z_{(p)}}\rightarrow \oplus_{i=1}^r\Lambda_{i,\Z_{(p)}}\mbox{ and }\pi:\oplus_{i=1}^r\Lambda_{i,\Z_{(p)}}\rightarrow \oplus_{j=1}^s\Lambda_{i_j,\Z_{(p)}}$$ given by the inclusion and projection. These induce maps $$\iota^\otimes: (\oplus_{j=1}^s\Lambda_{i_j,\Z_{(p)}})^\otimes\rightarrow (\oplus_{i=1}^r\Lambda_{i,\Z_{(p)}})^\otimes\mbox{ and }\pi^\otimes:(\oplus_{i=1}^r\Lambda_{i,\Z_{(p)}})^\otimes\rightarrow (\oplus_{j=1}^s\Lambda_{i_j,\Z_{(p)}})^\otimes$$ such that $\pi^\otimes \circ\iota^\otimes$ is the identity. Note that since $(\oplus_{j=1}^s\Lambda_{i_j,\Z_{(p)}})^\otimes$ involves taking duals, one needs to use $p$ in the definition of the map $\iota^\otimes$. These maps exhibit $(\oplus_{j=1}^s\Lambda_{i_j,\Z_{(p)}})^\otimes$ as a direct summand of $(\oplus_{i=1}^r\Lambda_{i,\Z_{(p)}})^\otimes$

\begin{lemma}
The $\iota^\otimes(s_{\alpha})$ are fixed by $G_{\Z_{(p)}}$
\end{lemma}
\begin{proof} It suffices to check this after inverting $p$. Then $i$ and $p$ are both equivariant for the action of $G_{\Q_p}$, hence so is $\iota$. Thus $\iota^\otimes(s_\alpha)$ is preserved by $G_{\Q_p}$.
	\end{proof}
We may extend $\iota^\otimes(s_{\alpha})$ to a collection of tensors $t_\beta\in (\oplus_{i=1}^r\Lambda_{i,\Z_{(p)}})^\otimes$ whose stabilizer is $G_{\Z_{(p)}}$. We fix an isomorphism $$(\oplus_{i=1}^r\Lambda_{i,\Z_{(p)}})^*\otimes_{\Z_p}\Ok_L\cong \oplus_{i=1}^r\D(\pdiv_i)$$
taking $t_\beta$ to $t_{\beta,0}$.

\begin{lemma}\label{lemma4}
	Any isomorphism as above preserves the product decomposition on either side and induces an isomorphism $$(\oplus_{j=1}^s\Lambda_{i_j,\Z_{(p)}})^*\otimes_{\Z_p}\Ok_L\cong \oplus_{j=1}^s\D(\pdiv_{i_j})$$taking  $s_\alpha $ to $s_{\alpha,0}$.
	
\end{lemma}
\begin{proof}
	We may assume that among the collection of tensors $t_\beta$ there are tensors $$t_{\beta_k}\in (\oplus_{i=1}^r\Lambda_{i,\Z_{(p)}})^*\otimes(\oplus_{i=1}^r\Lambda_{i,\Z_{(p)}})$$ corresponding to the projection $(\oplus_{i=1}^r\Lambda_{i,\Z_{(p)}})^*\rightarrow \Lambda_{k,\Z_{(p)}}^*$ for $k=1,\dots,r$. Indeed this follows since $$G_{\Z_{(p)}}\subset \prod_{i=1}^rGL(\Lambda_{i,\Z_{(p)}})$$and the latter group fixes these tensors. By the functoriality of the constructions, we have $$t_{\beta_k,0}\in (\oplus_{i=1}^r\D(\pdiv_i))\otimes (\oplus_{i=1}^r\D(\pdiv_i))^*$$ corresponds the to projection $\oplus_{i=1}^r\D(\pdiv_i)\rightarrow \D(\pdiv_k)$ and that the isomorphism takes $t_{\beta_k}$ to $t_{\beta_k,0}$ precisely says that the isomorphism is compatible with the product decompositions.
	
	That the induced isomorphism takes $s_\alpha$ to $s_{\alpha,0}$ follows from the fact that $\pi^\otimes\circ\iota^\otimes=\text{id}$.
\end{proof} 
\begin{prop}\label{maps}
The map $\pi_{\rmK_p,\rmK_p'}$ takes $\mathscr{I}_{\overline{x}}$ to $\mathscr{I}_{\overline{y}}$.
\end{prop}
\begin{proof}Suppose ${\overline{x}}'\in\mathscr{I}_{\overline{x}}$, then we have the triple $(\mathcal{A}_{{\overline{x}}'},\lambda_{{\overline{x}}'},\epsilon_{{\overline{x}}'}^p)$ corresponding to ${\overline{x}}'$ and there exists a quasi-isogeny $\theta:\mathcal{A}_{{\overline{x}}}\rightarrow \mathcal{A}_{{\overline{x}}'}$ taking $t_{\beta,l,{\overline{x}}}$ to $t_{\beta,l,{{\overline{x}}'}}$ for $l\neq p$ and  $t_{\beta,0,{\overline{x}}}$  to $t_{\beta,0,{{\overline{x}}'}}$. $\mathcal{A}_{\overline{x}}$ and $\mathcal{A}_{{\overline{x}}'}$ arise as  products $\prod_{i=1}^r\mathcal{A}_{{\overline{x}},,i}$ and $\prod_{i=1}^r\mathcal{A}_{{\overline{x}}',,i}$, and as in Lemma \ref{lemma4} we may assume that there exists tensors $t_{\beta_k}$ which correspond to the projections to the $k$ component. The tensors $t_{\beta_k,l,*}$ and $t_{\beta_k,0,*}$ then correspond to the projections on the Tate module and Dieudonn\'e modules and $\theta$ therefore respects these projections. It follows that $\theta$ decomposes decomposes as a product of quasi-isogenies $\theta_i:\mathcal{A}_{{\overline{x}},i}\rightarrow \mathcal{A}_{{\overline{x}}',i}$.

By construction $\pi_0^\otimes\circ \iota_0^\otimes (s_{\alpha,0,{\overline{x}}})=s_{\alpha,0,{\overline{x}}}$ and similarly for $s_{\alpha,0,{\overline{x}}'}.$ Thus $$\prod_{j=1}^s\theta_{i_j}:\prod_{j=1}^s\mathcal{A}_{{\overline{x}},i_j}\rightarrow \mathcal{A}_{{\overline{x}},i_j}$$ takes $s_{\alpha,0,{\overline{x}}}$ to $s_{\alpha,0,{\overline{x}}'}$. By a similar argument, it also takes $s_{\alpha,l,{\overline{x}}}$ to $s_{\alpha,l,{\overline{x}}'}$ for $l\neq p$, hence $\pi_{\mathrm{K}_p,\mathrm{K}'_p}({\overline{x}}')$ lies in $\mathscr{I}_{\overline{y}}$.
\end{proof}

\subsection{}We now use the description of the isogeny classes on the Iwahori level Shimura variety to deduce the description for arbitrary parahoric level. Thus suppose $\mathrm{K}_p$ is an Iwahori subgroup and $\mathrm{K}'_p$ is a parahoric whose corresponding facet lies in the closure of the alcove corresponding to $\mathrm{K}_p$. The projection map $\oplus_{i=1}^r\Lambda_{i}\rightarrow \oplus_{j=1}^s\Lambda_{i_j}$ induces a map $\mathrm{K}_p\rightarrow \mathrm{K}_p'$ which is the identity on the generic fiber. Thus if ${\overline{x}}\in \mathscr{S}_{\mathrm{K}_p}(G,X)$ and ${\overline{y}}=\pi_{\mathrm{K}_p,\mathrm{K}_p'}(\overline{x})$, we obtain an element $b\in G(L)$ giving the Frobenius on both $\D(\pdiv_{\overline{x}})$ and $\D(\pdiv_{\overline{y}})$.

Fix the  choice of maximal $L$-split torus $S$ which is compatible with the choice of parahorics. We assume $\rmK_p'$ corresponds to the subset $K'\subset \mathbb{S}$ of simple reflections. We have the affine Deligne-Lusztig varieties $X(\{\mu\},b)$ and $X(\{\mu\},b)_{K'}$ associated to the parahorics $\mathrm{K}_p$ and $\mathrm{K}_p'$ respectively. The natural projection $G(L)/\G(\Ok_L)\rightarrow G(L)/\G'(\Ok_L)$ induces a surjection $$X(\{\mu\},b)\rightarrow X(\{\mu\},b)_{K'}$$  by the main result of \cite{He3}. As in 6.13 we obtain a map $$i_{{\overline{x}}}':X(\{\mu\},b)\rightarrow \mathscr{S}_{\mathrm{K}_p}(GSp(V'),S^\pm)(k)$$
which is easily seen to factor through $\mathscr{S}_{\mathrm{H}_p}(H,T)(k)$. Similarly we obtain a map $$i'_{\overline{y}}:X(\{\mu\},b)_{K'}\rightarrow\mathscr{S}_{\mathrm{H}'_p}(H',T')(k)$$ which fits in a commutative diagram:

\[\xymatrix{X(\{\mu\},b)\ar[r]\ar[d] &\mathscr{S}_{\mathrm{H}_p}(H,T)(k)\ar[d]\\
X(\{\mu\},b)_{K'}\ar[r]&\mathscr{S}_{\mathrm{H}'_p}(H',T')(k)}\]

\begin{prop}The map $i_{{\overline{y}}}'$ lifts to a unique map $i_{\overline{y}}:X(\{\mu\},b)_{K'} \rightarrow \mathscr{S}_{\mathrm{K}_p}(G,{\overline{x}})(k)$ taking $s_{\alpha,0,i_{\overline{y}}(g)}=s_{\alpha,0,{\overline{y}}}$.
\end{prop}
\begin{proof}By Proposition \ref{prop1}, the map $i'_{\overline{x}}$ lifts to a map $i_{\overline{x}}:X(\{\mu\},b)\rightarrow\mathscr{S}_{\mathrm{K}_p}(G,X)(k)$ satisfying $t_{\beta,0,{\overline{x}}}=t_{\beta,0,i_{\overline{x}}(g)}$. By commutativity of the above diagram, we have $i_{{\overline{y}}}'$ factors through $\mathscr{S}_{\mathrm{K}'_p}(G,X)^-(k)$. Since $t_{\beta}$ extends $\iota^\otimes s_{\alpha}$, we have that $i_{\overline{y}}'$ lifts to $i_{\overline{y}}:X(\{\mu\},b)_{K'}\rightarrow \mathscr{S}_{\mathrm{K}_p'}(G,X)(k)$ with $$s_{\alpha,0,i_{\overline{y}}(g)}=s_{\alpha,0,\overline{y}}$$
\end{proof}
As before, $i_{\overline{y}}$ extends to a map $$i_{\overline{x}}:X(\{\mu\},b)_{K'}\times G(\A_f^p)\rightarrow \mathscr{S}_{\rmK_p'}(G,X)(k)$$The same proof as in Proposition \ref{isog} shows that $\mathscr{I}_{\overline{y}}$ can be identified with the image of $i_{\overline{y}}$, and this map is equivariant for the action of $\Phi\times Z_G(\Q_p)\times G(\A_f^p)$.

\section{He-Rapoport axioms}
\subsection{}In this section, we verify the axioms of He-Rapoport in \cite{HR}. We keep the assumptions of \S6, so that $(G,X)$ is Hodge type and $G_{\Q_p}$ splits over a tame extension. We have fixed a base alcove $\mathfrak{a}$ in  the apartment corresponding to some $\Q_p$ split $S$. We write $\mathcal{I}$ for the Iwahori group scheme and $\mathrm{I}_p=\mathcal{I}(\Z_p)$. For $K\subset \mathbb{S}$ a $\sigma$-invariant subset, we write $\mathcal{G}$ for the parahoric group scheme over $\Z_p$ and $\rmK_p=\G(\Z_p)$. As before all parahorics will be assumed connected.

\begin{thm}
i) Most of the five Axioms in \cite{HR} hold for $\mathscr{S}_{\rmK_p}(G,X)$

ii) When $G$ is residually split, all axioms in \cite{HR} hold.
\end{thm}

The only axiom we are unable to verify in the most general case is the surjectivity in 4c) of loc. cit. To verify these would require a better understanding of the isogeny classes in $\mathscr{S}_{\mathrm{K}_p\mathrm{K}^p}(G,X)$. However for the main application to non-emptiness of Newton strata, this result is not necessary (see the remark in \cite{HR} after 3.7).

\textbf{Axiom 1}: (Compatibility with change in parahoric) The compatability with the change in parahorics follows from Theorem 7.1. To show the map $\pi_{\mathrm{K}_p,\mathrm{K}'_p}$ is proper, we may apply a similar argument to the one in Theorem \ref{thm3} to reduce to the case of $GSp(V)$ considered in \cite{HZ}. Indeed let $\mathfrak{f},\mathfrak{f}'$ denote the facets corresponding to $\mathrm{K}_p$ and $\rmK_p'$, $\mathfrak{g}'$ respectively and let $\mathfrak{g}$ and $\mathfrak{g}'$ the facets in $B(GSp(V),\Q_p)$ containing the images of $\mathfrak{f}$ and $\mathfrak{f}'$. Then $\mathfrak{g}$ and $\mathfrak{g}'$ corresponds to the lattice chains in $V_{\Q_p}$ $$\mathcal{L}:=\{\Lambda_{1}\subset\Lambda_{2}\subset\dots\subset\Lambda_r\}$$
$$\mathcal{L}':=\{\Lambda_{i_1}\subset\Lambda_{i_2}\subset\dots\subset\Lambda_{i_s}\}$$

We write $\mathrm{M}_p$ and $\mathrm{M}_p'$ for the stabilizer of these lattice chains in $GSp(V_{\Q_p})$ and fix a suffienciently small compact open $M^p\subset GSp(V\otimes \A_f^p)$. As in \cite{HR}, we may consider the moduli problem which associates to a $\Z_{(p)}$ scheme $S$ the triple:

i) An $\mathcal{L}$-set of abelian varieties

ii) A polarization of the $\mathcal{L}$-set of abelian varieties as in \cite[\S 7]{HR}.

iii) A prime to $p$ level structure on the common rational Tate module away from $p$ of the $\mathcal{L}$-set $$\epsilon^p\in \underline{Isom}(\widehat{V}(\mathcal{A}_i),V\otimes\A_f^p)/M^p$$ compatible with the Riemann form on $\widehat{V}(\mathcal{A}_i)$ and $\psi$ on $V\otimes\A_f^p$.

We refer to \cite[\S7]{HR} for the precise definitions of an $\mathcal{L}$-set and a polarization. This moduli functor is representable by a scheme $\mathscr{S}_{\mathrm{M}_p\mathrm{M}^p}(GSp(V),S^\pm)$ and we have a natural map $$\mathscr{S}_{\mathrm{M}_p\mathrm{M}^p}(GSp(V),S^\pm)\rightarrow \mathscr{S}_{\mathrm{H}_p\mathrm{H}^p}(H,T)$$ where $H$ is the group considered in \S 7. The same proof as in Theorem \ref{thm3} shows that for sufficently small $\mathrm{H}_p$, the above map is a closed immersion. It follows as in Corollary 7.3 that if we take a closed embedding $Sh_{\mathrm{K}_p\mathrm{K}^p}(G,X)\rightarrow Sh_{\mathrm{M}_p\mathrm{M}^p}(H,X)$, we have a closed immersion $$\mathscr{S}_{\rmK_p\rmK^p}(G,X)^-\rightarrow\mathscr{S}_{\mathrm{M}_p\mathrm{M}^p}(GSp(V),S^\pm)$$ and hence a finite map

$$\mathscr{S}_{\rmK_p\rmK^p}(G,X)\rightarrow\mathscr{S}_{\mathrm{M}_p\mathrm{M}^p}(GSp(V),S^\pm)$$

We may apply the same considerations to $\mathfrak{g}'$ and $\mathrm{M}'$; we obtain a commutative diagram:

\begin{equation}\label{proj}
\xymatrix{\mathscr{S}_{\rmK_p\rmK^p}(G,X)\ar[r]\ar[d]_{\pi_{\mathrm{K}_p,\mathrm{K}'_p}} & \mathscr{S}_{\mathrm{M}_p\mathrm{M}^p}(GSp(V),S^\pm)\ar[d] \\
	\mathscr{S}_{\rmK'_p\rmK^p}(G,X) \ar[r] & \mathscr{S}_{\mathrm{M}'_p\mathrm{M}^p}(GSp(V),S^\pm)}
\end{equation}

The horizontal maps are finite hence proper. The vertical map on the right is proper by \cite[\S 7]{HR}, hence $\pi_{\mathrm{K}_p,\mathrm{K}'_p}$ is proper. Since the map $\pi_{\mathrm{K}_p,\mathrm{K}'_p}$ is surjective on the generic fiber, $\pi_{\mathrm{K}_p,\mathrm{K}'_p}$ is also surjective by properness

\textbf{Axiom 2}: (Local Model diagram) As in \cite[Theorem 4.2.7]{KP} we have a diagram:
\[\xymatrix{&\tilde{\mathscr{S}}_{\mathrm{K}_p\mathrm{K}^p}(G,X)\ar[dl]^\pi\ar[dr]_q& \\
\mathscr{S}_{\mathrm{K}_p\mathrm{K}^p}(G,X) & & M^{loc}_\G}\]
where $\pi$ is a $\G$-torsor and $q$ is a smooth map equivariant for the action of $\G$. By \cite[Theorem 8.3]{PZ}, $M^{loc}_\G$ has a stratification by $\Adm(\{\mu\})_K$, and this diagram induces a stratification $$\lambda_{K_p}:\mathscr{S}_{\mathrm{K}_p\rmK^p}(k)\rightarrow \Adm(\{\mu\})_K$$

\textbf{Axiom 3}: (Newton Stratification) By 6.13, to every $\overline{x}\in\mathscr{S}_{\mathrm{K}_p\mathrm{K}^p}(G,X)$, we obtain a $b\in G(L)$ well defined up to $\sigma$-conjugation by $\G(\Ok_L)$.  We have $\sigma^{-1}(b)\in \bigcup_{w\in\Adm(\{\mu\})}\G(\Ok_L)\dot{w}\G(\Ok)$, hence $b\in B(G,\{\mu\})$. We thus obtain a  $$\delta:\mathscr{S}_{\mathrm{K}_p}(k)\rightarrow B(G,\{\mu\})$$
To show that $\delta$ induces a stratification on $\mathscr{S}_{\mathrm{K}_p}$, we must show this map arises from an isocrystal with $G$-structure. This follows from \cite[Corollary A7]{KMS}.

\textbf{Axiom 4}: (Joint stratification) a) Let $G(L)/\G(\Ok_L)_{.\sigma}$ denote the set of $\G(\Ok_L)$ conjugate classes of in $G(L)$. We have natural projection maps $$d_{\rmK_p}:G(L)/\G(\Ok_L)_{.\sigma}\rightarrow B(G)\mbox{ and } l_{\rmK_p}:G(L)/\G(\Ok_L)_{.\sigma}\rightarrow \G(\Ok_L)\backslash G(L)/\G(\Ok_L)$$

 Let $\overline{x}\in\mathscr{S}_{\rmK_p}(k)$, then we have an isomorphism $$V_{\Z_p}^*\otimes\Ok_L\cong \D(\pdiv_{\overline{x}})$$ taking $s_{\alpha}$ to $s_{\alpha,0,x}$. Then we obtain $b\in G(L)$ well-defined up to $\sigma$-conjugation by $\G(\Ok_L)$.
The map $\Upsilon_{\rmK_p}:\mathscr{S}_{\rmK_p}(k)\rightarrow G(L)/\G(\Ok_L)_{.\sigma}$ is defined by $\Upsilon_{\rmK_p}(\overline{x})=[\sigma^{-1}(b)]$. It is clear by definition that $d_{\rmK_p}\circ\Upsilon_{\rmK_p}=\delta_{\rmK_p}$. By definition of the local model diagram and Proposition 3.3, we have $l_{\rmK_p}\circ\Upsilon_{\rmK_p}=\lambda_{\rmK_p}$.

b) By \cite[Lemma 3.11]{HR}, it suffices to prove this in the case of Iwahori level. We need to show for all $\delta\in l_{\mathrm{I}_p}^{-1}(\Adm(\{\mu\}))$, there exists $\overline{x}\in\mathscr{S}_{\mathrm{I}_p}(G,X)(k)$ such that $\Upsilon_{\mathrm{I}_p}(\overline{x})=\delta$. We fix a lift $b\in G(L)$ of $\sigma(\delta)$, then by definition $1\in X(\sigma(\{\mu\}),b)$. By Theorem \ref{thm1}, there exists a $\sigma$-straight element $w\in \sigma(\Adm(\{\mu\}))$ and $h\in X_w(b)$ such that $h$ lies on the same connected component of $X(\sigma(\{\mu\}),b)$ as $1$. Since $w$ is fundamental (see \cite{Nie}) we may assume $h^{-1}b\sigma(h)=\dot{w}$. We pick $\overline{x}'\in\lambda_{\mathrm{I}_p}^{-1}(\sigma^{-1}(w))$, which exists since $\lambda_{\mathrm{I}_p}$ is surjective (see Theorem 8.2 below), and we may pick the isomorphism $V_{\Z_p}^*\otimes \Ok_L\cong \D(\pdiv_{\overline{x}})$ so that the Frobenius is given by $\dot{w}\sigma$. Using the isomorphism $X(\sigma(\{\mu\}),b)\cong X(\sigma(\{\mu\}),\dot{w})$ given by $g\mapsto h^{-1}g$, we have $h^{-1}\in X(\sigma(\{\mu\}),\dot{w})$ is connected to $1\in X_w(\dot{w})$. By Proposition \ref{prop8}, $i_{\overline{x}'}(h^{-1})\in \mathscr{S}_{\mathrm{I}_p}(G,X)(k)$ is well defined and we have $\Upsilon_{I_p}(\overline{x})=\delta.$

c) For $K_p\subset K_p'$, let $\delta\in\text{Im}(\Upsilon_{\rmK_p})$ and $\delta'$ its image in $G(L)/\G'(\Ok_L)_{.\sigma}$. The finiteness of the fibers of $\pi_{\rmK_p,\rmK_p'}|_{\Upsilon_{\rmK_p}^{-1}(\delta)}$ can be deduced from the case of $GSp(V)$. Indeed it follows from the diagram \ref{proj}, that a fiber of $\pi_{\rmK_p,\rmK_p'}|_{\Upsilon_{\rmK_p}^{-1}(\delta)}$ admits a finite map to a corresponding fiber for the integral model for $GSp$.

For the surjectivity, we assume $G_{\Q_p}$ residually split. Let $\overline{x}'\in \Upsilon_{K_p'}^{-1}(\delta')$, then by Axiom 1, there exist $\overline{x}\in \mathscr{S}_{\mathrm{K}_p}(G,X)(k)$ such that $\pi_{\mathrm{K}_p,\mathrm{K}'_p}(\overline{x})=\overline{x}'$. Let $\gamma=\Upsilon_{\rmK_p}(\overline{x})$, then by compatibility of the map $\Upsilon$ with $\pi_{\mathrm{K}_p,\mathrm{K}'_p}$, the image of $\gamma\in G(L)/\G'(\Ok_L)_{.\sigma}$ is equal to $\delta'$. Thus $\gamma$ and $\delta$ are sigma-conjugate by an element of $\G'(\Ok_L)$, and the same is true for $b:=\sigma(\gamma)$ and $\sigma(\delta)$. Let $g\in\G'(\Ok_L)$ such that $g^{-1}b\sigma(g)=\sigma(\delta)$, then $g\in X(\sigma(\{\mu\}),b)_K$ and its image in $X(\sigma(\{\mu\}),b)_{K'}$ is equal to $1$. Thus $i_{\overline{x}}(g)\in\mathscr{S}_{K_p}(k)$ satisfies $\Upsilon(i_{\overline{x}}(g))=\delta$, and $\pi_{\rmK_p,\rmK_p'}(x)=x'$.

\textbf{Axiom 5}: (Basic non-emptiness) Recall $\mathrm{I}_p$ was an Iwahori subgroup of $G(\Q_p)$. Let $\tau_{\{\mu\}}$ denote the unique minimal element of $\Adm(\{\mu\})$. We need to show $\lambda_{\mathrm{I}_p}^{-1}(\tau_{\{\mu\}})\neq\emptyset$.

By \cite{KMS}, there exists $\overline{x}\in\mathscr{S}_{I_p}(k)$ such that $\delta_{\mathrm{I}_p}=[b]_{basic}$. Let $g\in X_{\sigma(\tau_{\{\mu\}})}(b)\subset X(\sigma(\{\mu\}),b)$. By Prop \ref{prop1}, the map $i_x:X(\sigma(\{\mu\}),b)\rightarrow \mathscr{S}_{\mathrm{I}_p}$ is well defined.  Then by Proposition \ref{KR}, $i_x(g)\in\lambda_{\mathrm{I}_p}^{-1}(\tau_{\{\mu\}})$.

The main application of the above results is the non-emptiness of Newton strata. 

\begin{thm} $\lambda_{K_p}$ and $\delta_{K_p}$ are surjective.
\end{thm}
\begin{proof}Since the $\pi_{\mathrm{K}_p,\mathrm{K}_p'}$ are compatbile with the maps $\delta$ and $\lambda$ it suffices to prove the result when $\mathrm{K}_p$ is Iwahori. Since $q$ is a smooth map, hence open, the image of $q$ is an open mapping. Thus $q\otimes k$ is surjective, in particular $\lambda_{I_p}$ is surjective.

It follows from \cite{HR} that $\delta$ is also surjective. Indeed it is proved in loc. cit. that for each $[b]\in B(G,\mu)$, there exists $w\in \Adm(\{\mu\})$ $\sigma$-straight such that $\lambda^{-1}(w)\subset \delta^{-1}([b])$.
\end{proof}
\section{Lifting to special points}

\subsection{}In this section we show that every isogeny class in $\mathscr{S}_{\mathrm{K}_p}(G,X)$ admits a lift to a special point of $Sh_{\mathrm{K}_p}(G,X)$. We assume the group $G_{\Q_p}$ is residually split, in particular it is quasi-split. The proof follows ideas from \cite[\S2]{Ki3}, the main new input for being a generalization of the so called Langlands-Rapoport lemma, see \cite[Lemma 2.2.2]{Ki3}. This allows us to associate a Kottwitz triple to each isogeny class, a key ingredient needed to enumerate the set of isogeny classes. A proof of this result has also been annouced in \cite{Ki4} using a different method. 

We first recall some notation from \cite{Ki3}.  As before $(G,X)$ is of Hodge type and $\rmK_p=\G(\Z_p)$ is a connected parahoric subgroup corresponding to $K\subset\mathbb{S}$, but now $k\subset \Fpbar$ will denote a finite extension of the residue field $k_E$ of $\mathcal{O}_{E_{(v)}}$. $\mathrm{K}^p\subset G(\A_f^p)$ is a sufficiently small compact open subgroup and $\mathrm{K}=\rmK_p\rmK^p$. We write $r$ for the degree of $k$ over $\mathbb{F}_p$, and write $W:=W(k)$, and $K_0=W(k)[\frac{1}{p}]$. 

\subsection{}For $x\in \mathscr{S}_{\mathrm{K}}(G,X)(k)$ we write $\overline{x}$ for the $\Fpbar$ point associated to $x$. Recall we have an associated abelian variety $\mathcal{A}_x$ together with Frobenius invariant  tensors $s_{\alpha,l,x}\in H^1_{\acute{e}t}(\mathcal{A}_{\overline{x}},\Q_l)^\otimes$ whose stablizer in $GL(H^1_{\acute{e}t}(\mathcal{A}_{\overline{x}},\Q_l))$ can be identified with $G_{\Q_l}$ via the level structure $\epsilon^p$. Since the $s_{\alpha,l,x}$ are invariant under the action of the geometric Frobenius $\gamma_l$ on $H^1_{\acute{e}t}(\mathcal{A}_{\overline{x}},\Q_l)$, we may consider  $\gamma_l$ as an element of $G(\Q_l)$. We let $I_{l/k}$ denote the centralizer of $\gamma_l$ in $G(\Q_l)$ and $I_l$ the centralizer of $\gamma_l^n$ for $n$ sufficiently large, cf. \cite[\S 2.1.2]{Ki3}.

We also fix an identification $$\D(\pdiv_x)\otimes_W K_0\cong V_{\Z_p}^*\otimes_{\Z_p}K_0$$ taking $s_{\alpha,0,x}$ to $s_{\alpha}$. The Frobenius on $\D(\pdiv_x)$ is of the form $\varphi=\delta\sigma$ for some $\delta\in G(K_0)$ and we $\gamma_p$ for the element $\delta\sigma(\delta)\dots\sigma^{r-1}(\delta)\in G(K_0)$.

Let $I_{p/k}$ denote the group over $\Q_p$ whose $R$ points are given by $$I_{p/k}(R)=\{g\in G(K_0\otimes_{\Q_p}R)|g^{-1}\delta\sigma(g)=\delta\}$$
Clearly  $I_{p/k}\subset J_\delta$. We have $\gamma_p\in I_{p/k}(\Q_p)$ and we have $I_{p/k}\otimes_{\Q_p}K_0$ is identified with the centralizer of $\gamma_p$ in $G_{K_0}$.

For $n\in \mathbb{N}$, we write $k_n$ for the degree $n$ extension of $k$, and $I_{p/k_n}$ the group over $\Q_p$ defined as above with  $K_0$ replaced with $W(k_n)[\frac{1}{p}]$. $I_{p}$ will then denote the $I_{p/k_n}$ for sufficiently large $n$.

Finally we let $\text{Aut}_{\Q}(\mathcal{A}_x)$ denote the group over $\Q$ defined by $$\text{Aut}_{\Q}(\mathcal{A}_x)(R)=(\text{End}_{\Q}(\mathcal{A}_x)\otimes_{\Z}R)^\times$$ where $\mbox{End}_{\Q}(\mathcal{A}_x)$ denotes the set of endomorphisms of $\mathcal{A}_x$  viewed as an abelian variety up to isogeny defined over $k$.
We write $I_{/k}\subset \text{Aut}_{\Q}(\mathcal{A}_x)$ for the subgroup of elements which preserve the tensors $s_{\alpha,l,x}$ for $l\neq p$ and $s_{\alpha,0,x}$. We obtain maps $I_{/k}\rightarrow I_{l/k}$ for all $l$ (including $l=p$).

Similarly we write $I\subset\mbox{Aut}_\Q(\mathcal{A}_x\otimes\Fpbar)$ for the subgroup which fixes $s_{\alpha,l}$ for all $l\neq p$ and  $s_{\alpha,0,x}$. Again we have maps $I\rightarrow I_l$ for all $l$.

\subsection{}By the argument in Proposition \ref{maps} the projection maps $\pi_{\rmK_p,\rmK_p'}$ are compatible with the construction of $\delta$ and $\gamma_l$. Thus the above definitions are independent of level structure, i.e. for $x\in \mathscr{S}_{\rmK_p}(G,X)(k)$ and $y=\pi_{\rmK_p,\rmK_p'}(x)$, the construction above give rise to the same groups $I_{l/k}, I_{p/k}$ and $I$.

The same proof as in \cite[2.1.3 and 2.1.5]{Ki3} gives us the following proposition:

\begin{prop}\label{Kisin}
i) The map $i_{\overline{x}}$ of Proposition \ref{prop6.1} induces an injective map 

$$i_{\overline{x}}: I(\Q)\backslash X(\{\mu\},\delta)_K\times G(\A_f^p)\rightarrow\mathscr{S}_{\mathrm{K}_p}(G,X)(\Fpbar)$$

ii) Let $H^p=\prod_{l\neq p}I_{l/k}(\Q_l)\cap K^p$ and $H_p=I_{p/k}\cap \G(W(k))$. The map in i) induces an injective map:

$$I_{/k}(\Q)\backslash\prod_lI_{l/k}(\Q_l)/H_p\times H^p\rightarrow\mathscr{S}_{\rmK}(G,X)(k)$$

iii) For some prime $l\neq p$, the connected component of $I_{\Q_l}=I\otimes_{\Q}\Q_l$ contains the connected component of the identity in $I_l$. In particular the ranks of $I$ and $G$ are equal.

\end{prop}

\subsection{}The next lemma is the   key technical ingredient needed for the existence of CM lifts, for this we need to recall some group theoretic preliminaries. Recall $W$ is the Iwahori Weyl group of $G$ and $\mathbb{S}$ is set of simple reflections in $W$ corresponding to a choice of base alcove. We write $K\subset \mathbb{S}$ for the subset of reflections fixing the special vertex corresponding to a choice of special parahoric $\G$. Let $W_K$ be the group generated by the reflections in $K$, it is identified with the relative Weyl group $N(L)/T(L)$. By \cite{HR}, we have an identification 

$$\G(\Ok_L)\backslash G(L)/\G(\Ok_L)\cong W_K\backslash W/W_K$$ and this latter set can be identified with $X_*(T)_I/W_K$. The choice of alcove determines a chamber $V_+$ in $V:=X_*(T)_I\otimes\mathbb{R}$ and a Borel subgroup $B$ of $G$ defined over $L$. We now describe the relationship between $V_+$ and $B$ more explicitly.

Let $\langle\ ,\ \rangle: X_*(T)\times X^*(T)\rightarrow \Z$ be the natural pairing and we use the same symbol to denote the scalar extension to $\R$. We let $\Psi\subset X^*(T)$ denote the set of roots, then $B$ determines a system of positive roots $\Psi_+\subset\Psi$. Now for $K/L$ a finite Galois extension  over which $T$ splits, we have a norm map $$\mbox{Nm}:X_*(T)_I\rightarrow X_*(T)^I$$ given by $$\underline{\mu}\mapsto\sum_{\sigma\in\text{Gal}(K/L)}\sigma(\mu)$$ 
where $\underline{\mu}\in X_*(T)_I$ and $\mu\in X_*(T)$ is a lift. This extends linearly to a map $V\rightarrow X_*(T)^I\otimes\mathbb{R}$. Then $V_+$ can be identified with the subset of $V$ consisting of $x$ such that $\langle \mbox{Nm}(x),\alpha\rangle\geq 0$ for all $\alpha\in \Psi_+$. 

We write $X_*(T)_{I,+}$ for the subset of $X_*(T)_I$ which mapsto $V_+$, then $W_K\backslash W/W_K$ can be identified with $X_*(T)_{I,+}$.

Now recall we have the affine Weyl group $W_a\subset W$. By \cite{BT1}, there is a reduced root system  $\Sigma$ such that $$W_a\cong Q^\vee(\Sigma)\ltimes W(\Sigma)$$ where $Q^\vee(\Sigma)$ is the coroot lattice of $\Sigma$ and $W(\Sigma)$ is its Weyl group. The roots in $\Sigma$ are proportional to the roots in the relative root system of $G$ over $L$, however the root systems themselves may not be proportional. The choice of Borel $B$, then determines an ordering of the roots in $\Sigma$.

 We thus have identifications $X_*(T_{sc})_I\cong Q^\vee(\Sigma)$ and $W_K\cong W(\Sigma)$. The length function and Bruhat order on $W$ is determined by $W_a$ and hence by $\Sigma$.  For $\underline{\lambda},\underline{\mu}\in X_*(T) _{I,+}$, we write  $\underline{\lambda}\preccurlyeq\underline{\mu}$ if $\underline{\mu}-\underline{\lambda}$ is a positive linear combination of positive coroots in $Q^\vee(\Sigma)$ with integral coefficients. By \cite[\S 2]{Lu} applied to the root system $\Sigma$, we have $$t_{\underline{\lambda}}\leq t_{\underline{\mu}}\Leftrightarrow \underline{\lambda}'\preccurlyeq\underline{\mu}'$$ where $\underline{\lambda}'$ and $\underline{\mu}'$ are the dominant representatives of $\underline{\lambda}$ and $\underline{\mu}$ respectively.

\subsection{}The following is a generalization of \cite[2.2.2]{Ki3} to residually split (but possibly ramified) groups. Recall we have the Newton cocharacter $\nu_{\delta}:\D\rightarrow G$, which is central in $J_\delta$ and hence central in $I_p$.

\begin{lemma}\label{LR}Let $T_p\subset I_p$ be a maximal torus defined over $\Q_p$. Then there exists a cocharacter $\mu_{T_p}\in X_*(T_p)$ such that:

i) Considered as a $G$-valued cocharacter, $\mu_{T_p}$ is conjugate to $\mu$.

ii) $\overline{\mu}_{T_p}^{T_p}=\nu_\delta$

\end{lemma}
\begin{proof}Let $T'\subset T_p$ denote the maximal $\Q_p$ split subtorus. The same proof as \cite[Lemma 2.2.2]{Ki3} shows upon changing the isomorphism $\D(\pdiv_x)\otimes K_0\cong V_{\Z_p}^*\otimes K_0$, we may assume $T'$ commutes with $\delta$. Since $\Q_p$ structure on the image of $T'_{K_0}$ in $G_{K_0}$ differs by the one on $T'$ by conjugation by $\delta$, we may consider $T'$ as a subtorus of $G$. Let $M$ denote the centralizer of $T'$ in $G$, we have $\delta\in M(K_0)$.

 Let $T'\subset T'_2$ be a maximal $\Q_p$ split torus in $G$, and $T_2$ it's centralizer; it is a maximal torus since $G$ is quasi split. Then $T_2\subset M$. Let $P$ be a parabolic subgroup containing $M$ with unipotent radical $N$. Let $g\in X(\{\mu\},\delta)_K$ which exists since $\delta\in B(G,\{\mu\})$. Then there exists $\underline{\mu}_1\in X_*(T)_I$ with $t_{\underline{\mu}_1}\leq t_{\underline{\mu}}$ such that $g^{-1}b\sigma(g)\in \G(\Ok_L)\dot{t}_{\underline{\mu}_1}\G(\Ok_L)$. Note that when $G$ splits over an unramified extension, $\underline{\mu}$ is minuscule and hence $\underline{\mu}_1=\underline{\mu}$, this is not true in general. By the Iwasawa decomposition, we may assume $g=mn$ for $n\in N(L), m\in M(L)$. Let $\mathcal{M}(\Ok_L)=M(L)\cap\G(\Ok_L)$ a special parahoric subgroup of $M$ defined over $\Z_p$.
 
 Then we have $$m^{-1}\delta\sigma(m)\in \mathcal{M}(\Ok_L) \dot{t}_{\underline{\lambda}}\mathcal{M}(\Ok_L)$$
 for some $\underline{\lambda}\in X_*(T)_I$. Let $m^{-1}\delta\sigma(m)=m_1\dot{t}_{\underline{\lambda}}m_2$ in the decomposition above. Now $$g^{-1}\delta\sigma(g)=n^{-1}m^{-1}\delta\sigma(m)\sigma(n)=\tilde{n}m^{-1}\delta\sigma(m)$$for some $\tilde{n}\in N(L)$. Thus $$(m_1^{-1}\tilde{n}m_1)\dot{t}_{\underline{\lambda}}m_2\in N(L)\dot{t}_{\underline{\lambda}}\mathcal{M}(\Ok_L)\cap \G(\Ok_L)\dot{t}_{\underline{\mu}_1}\G(\Ok_L)$$ hence by \cite[Lemma 10.2.1]{HaR} we have $t_{\underline{\lambda}}\leq t_{\underline{\mu}_1}\leq t_{\underline{\mu}}$.
 
 By Lemma \ref{cochar}, below we have there exists a lift of $\underline{\lambda}$ to $v_2\in X_*(T_2)$ which is conjugate to $\mu$ in $G$. Then $\delta\in B(M,\{v_2\}_M)$ by \cite{He3}, hence $\overline{v}_2$ has the same image as  $ \nu_\delta$ in $\pi_1(M)_{\Q}$. 
 
 Now let $\mu_{T_p}\in X_*(T_p)$ be cocharacter conjugate to $v_2$ in $M$. Then since conjugate cocharacters have the same image in $\pi_1$, we have $\overline{\mu}_{T_p}^M=\nu_\delta\in \pi_1(M)_\Q$, where $\overline{\mu}^M_{T_p}$, denotes the Galois of average $\mu_{T_p}$ computed as a cocharacter of $M$. Then as before, since the  two $\Q_p$ structures on image on the image of $T_{K_0}$ differ by conjugation by $\delta$, we have $\overline{\mu}_{T_p}^{T_p}=\overline{\mu}_{T_p}^{M}\in\pi_1(M)_{\Q}$. 
 
 Now as $\overline{\mu}_{T_p}^{T_p}$ is defined over $\Q_p$, hence we may consider it as an element of $X_*(T')_{\Q}$. We have $\sigma(\nu_\delta)=\delta^{-1}\nu_\delta\delta$, hence $\nu_\delta$ is defined over $\Q_p$ and we have $\nu_\delta\in X_*(T')_{\Q}$. Since $T'\subset M$ is central, we have $\nu_\delta=\overline{\mu}_{T_p}^{T_p}$.

\end{proof}

\begin{lemma}\label{cochar}Let $\mu$ be a minuscule cocharacter of $G$ and $t_{\lambda}\in X_*(T)_I$ whose image in $W_K\backslash W/W_K$ lies in $\Adm(\{\mu\})_K$. Then there exists a cocharacter $v_2\in X_*(T)$ lifting $\lambda$ which is conjugate to $\mu$ in $G$.
\end{lemma}

\begin{proof} Let $\underline{\lambda}\in X_*(T)_{I,+}$ denote the dominant representative of $\lambda$ for our choice of Borel $B$. By \cite[\S 2]{Lu}, $t_\lambda\in \Adm(\{\mu\})_K$ implies $t_{\underline{\mu}}-t_{\underline{\lambda}}$ is a positive linear combination of coroots in $\Sigma$ (Recall $\underline{\mu}$ is the image of a dominant representative of $\{\mu\}$ in $X_*(T)$). Note that in general $\mu$ being minuscule in $G$ does not imply $\underline{\mu}$ is minuscule for the root system $\Sigma$, so that it is possible that $t_{\underline{\lambda}}\neq t_{\underline{\mu}}.$

Since $W_0=N(L)/T(L)$ is a subgroup of the absolute Weyl group it suffices to prove the result for $\underline{\lambda}$. By Stembridge's Lemma \cite[Lemma 2.3]{Ra1} there exists a sequence of positive coroots $\underline{\alpha}^\vee_1,\dots,\underline{\alpha}^\vee_n\in \Sigma^\vee$ such that $$\underline{\mu}-\underline{\alpha}^\vee_1-\dots\underline{\alpha}^\vee_i\in X_*(T)_I$$ is dominant for all $i$ and $\underline{\mu}-\underline{\alpha}^\vee_1-\dots.-\underline{\alpha}^\vee_n=\underline{\lambda}$. We prove by induction on $i$ that $$\underline{\lambda_i}:=\underline{\mu}-\underline{\alpha}^\vee_1-\dots-\underline{\alpha}^\vee_i\in X_*(T)_I$$ admits a lifting $\lambda_i\in X_*(T)$ which is conjugate to $\mu$. The case is $i=0$ in which case $\underline{\lambda}=\underline{\mu}$ and the result is obvious.

Suppose we have shown the existence of $\lambda_i$. Let $\alpha^\vee_{i+1}\in X_*(T_{sc})$ be a positive coroot lifting $\underline{\alpha}^\vee_i$. Then since $\underline{\lambda}_{i+1}$ is dominant, we have $\langle \underline{\lambda}_{i+1},\underline{\alpha}^\vee_{i+1}\rangle\geq 0$ and hence $\langle \underline{\lambda}_i,\underline{\alpha}_i\rangle=\langle \underline{\lambda}_{i+1}+\underline{\alpha}^\vee_{i+1},\underline{\alpha}_i\rangle>0$.

 Letting $K/L$ be a finite Galois extension over which $T$ splits, we have by  definition $$\langle \underline{\lambda_i},\underline{\alpha_{i+1}}\rangle=\frac{1}{n}\sum_{\sigma\in\text{Gal}(K/L)}\langle\sigma(\lambda_i),\alpha_{i+1}\rangle$$
 Since $\langle\sigma(\lambda_i),\alpha_{i+1}\rangle=\langle\lambda_i,\sigma(\alpha_{i+1})\rangle$, upon replacing $\alpha_{i+1}$ by $\sigma(\alpha_{i+1})$, we may assume $\langle\lambda_i,\alpha_{i+1}\rangle>0$. By assumption $\lambda_i$ is minuscule hence $\langle\lambda_i,\alpha_{i+1}\rangle=1$. We set $\lambda_{i+1}=s_{\alpha_{i+1}}(\lambda_i)=\lambda_i-\alpha_{i+1}$ where $s_{\alpha_{i+1}}$ is the simple reflection corresponding to $\alpha_{i+1}$. Then $\lambda_{i+1}$ is minuscule since it is the Weyl conjugate of a minuscule cocharacter, and $\lambda_{i+1}$ is a lift of $\underline{\lambda}_{i+1}.$
 \end{proof}
 
 \begin{thm}\label{CM-lift} Let $x\in\mathscr{S}_{\rmK}(G,X)(k)$. The isogeny class of $x$ contains a point which lifts to a special point on $Sh_{\rmK}(G,X)$.
 \end{thm}
 \begin{proof}
 
 Let $T_p\subset I_p$ a maximal torus and $\mu_{T_p}$ the cocharacter constructed above. Recall we have fixed the ismorphism $\D(\pdiv_{x})\otimes K_0\cong V_{\Z_p}^*\otimes K_0$ such that $\delta$ commutes with the maximal $\Q_p$ split subtorus $T'$ of $T_p$ and hence $\nu_{\delta}$ is defined over $\Q_p$. Let $M_{\nu_{\delta}}$ denote the centralizer of $\nu_{\delta}$, then there is an inner twisting $M_{,\nu_{\delta}\overline{\Q}_p}\cong J_{b,\overline{\Q}_p}$. By \cite[Lemma 2.1]{Langlands}, there is an embedding $j:T_p\hookrightarrow M_{\nu_{\delta}}$ over $\Q_p$ which is $M_{\nu_{\delta}}$-conjugate to 
 \begin{equation}
 \label{eq3}
 T_{p,\overline{\Q}_p}\hookrightarrow I_{p,\overline{\Q}_p}\hookrightarrow J_{b,\overline{\Q}_p}\xrightarrow{\sim}M_{\nu_{\delta},\overline{\Q}_p}\end{equation}
By Steinberg's theorem, there is an element $m\in M_{\nu_{\delta}}(L)$ which conjugates $j$ to \ref{eq3}. Thus upon modifying the isomorphism $\D(\pdiv_{x})\otimes L\cong V_{\Z_p}^*\otimes L$ by $m$, we have $\delta$ commutes with $j(T')\subset G$. Let $M$ denote the centralizer of $j(T')$ in $G$, hence $T_p$ is an elliptic maximal torus in $M$ and $\delta\in M(L)$. By \cite{Ko1}, we may modify the isomorphism $\D(\pdiv_{x})\otimes L\cong V_{\Z_p}^*\otimes L$ by an element of $M(L)$ so that $\delta\in T_p(L)$. 

Let $K/L$ be a finite extension such that $\mu$ is defined over $K$, then by \cite{RZ}, the filtration induced by $j\circ\mu_{T_p}$ is admissible. As $j\circ\mu_{T_p}$ is conjugate to $\mu$, the filtration has weight $0,1$ hence by \cite[2.2.6]{Ki1}, there exists a $p$-divisible group $\tilde{\pdiv}'$ over $\Ok_K$ with special fiber $\pdiv'$, such that we have an identification $\D(\pdiv')\otimes L\cong \D(\pdiv_{x})\otimes L$. This induces a quasi-isogeny $\theta: \pdiv_{x}\rightarrow \pdiv'$.
 
  Let $\tilde{x}\in\mathscr{S}_{\mathrm{K}}(G,X)(\Ok_{K})$ be a point lifting ${x}$, $s_{\alpha,\acute{e}t,\tilde{x}}\in T_p\pdiv_{\tilde{x}}^{\vee,\otimes}$ and $s_{\alpha,0,x}\in \D(\pdiv_{x})^\otimes$ the corresponding crystalline tensors. Let $s_{\alpha,\et}'\in T_p\tilde{\pdiv}'^{\vee,\otimes}$ the tensors corresponding to the $s_{\alpha,0,{x}}$ under the $p$-adic comparison isomorphism. As in \cite[1.1.19]{Ki3}, there exists a $\Q_p$-linear  isomorphism $$T_p\pdiv_{\tilde{x}}^{\vee}\otimes \Q_p\cong T_p\tilde{\pdiv}'^\vee\otimes\Q_p$$
  taking $s_{\alpha,\acute{e}t,\tilde{x}}$ to $s_{\alpha,\acute{e}t}'$. Upon making a finite extension of $K$, we may assume the image of $T_p\pdiv_{\tilde{x}}^\vee$  in $T_p\tilde{\pdiv}'^\vee\otimes\Q_p$ is stable under the Galois action. Upon replacing $\tilde{\pdiv}'$ by an isogenous $p$-divisible group, we may assume there is an isomorphism $$T_p\pdiv_{\tilde{x}}\cong T_p\tilde{\pdiv}'$$ taking $s_{\alpha,\acute{e}t,\tilde{x}}$ to $s_{\alpha,\acute{e}t}'$. 
  
  By Proposition \ref{prop3}, we have $s_{\alpha,0,{x}}\in\D(\pdiv')^\otimes$ and we have a sequence of isomorphisms
  
  $$\D(\pdiv{x})\cong T_p\pdiv_{\tilde{x}}\otimes_{\Z_p}\Ok_L\cong T_p\tilde{\pdiv}'\otimes_{\Z_p}\Ok_L\cong \D(\pdiv')$$ which preserve $s_{\alpha,0,{x}}$. We may thus identify $\D(\pdiv')$ with $g\D(\pdiv_{x})$ for some $g\in G(L)$. As in Proposition 5.4, the filtration induced by $g^{-1}b\sigma(g)$ is the specialization of a filtration induced by a $G$-valued cocharacter conjugate to $\mu_y$. Hence the filtration corresponds to point of the local model $M^{loc}_{\G}(k)$ and we have $g^{-1}b\sigma(g)\in \Adm(\{\mu\})$, i.e. $g\in X(\{\mu\},b)$.
  
  Thus upon replacing ${x}$ by $i_{x}(g)\in \mathscr{S}_{\rmK}(G,X)(k)$, we may assume  there is a deformation $\tilde{\pdiv}$ of $\pdiv_{x}$ to $\Ok_K$, such that the coorresponding filtration on $\D(\pdiv)\otimes K$ is induced by $\mu_{T_p}$. Since $\mu_{T_p}$ is conjugate to $\mu_h^{-1}$, $\tilde{\pdiv}$ corresponds to a point $\tilde{x}\in\mathscr{S}_{\mathrm{K}}(G,X)(\Ok_K)$ by Proposition \ref{prop12}.
  
That $\tilde{x}$ is a special point of $Sh_{\mathrm{K}}(G,X)$ now follows from the same proof as \cite[2.2.3]{Ki3}. Indeed since $I$ and $I_p$ have the same rank we may assume $T_p$ comes from a maximal torus $T$ in $I$. Then $T\subset I \subset \text{Aut}_{\Q}(\mathcal{A}_x)$ is compatible with filtrations and hence lifts to the isogeny category. As $T$ fixes $s_{\alpha,0,{x}}$, it fixes $s_{\alpha,p,\tilde{x}}$ and hence also $s_{\alpha,B,\tilde{x}}$. Thus $T$ is naturally a subgroup of $G$ and is a maximal torus by Proposition \ref{Kisin}. The Mumford-Tate group is a subgroup of $G$ which commutes with $T$, hence is contained in $T$. Hence $\tilde{x}$ is a special point.
\end{proof}

As in \cite[\S 2.3]{Ki3}, we may use the above to associate an element $\gamma_0\in G(\Q)$ to each isogeny class such that:

i) For all $l\neq p$, $\gamma_0$ is $G$-conjugate to $\gamma_l$ in $G(\Q_l)$.\

ii) $\gamma_0$ is stably conjugate to $\gamma_p$ in $G(\overline{\Q}_p)$ 

iii) $\gamma_0$ is elliptic in $G(\R)$.
\\ i.e. $(\gamma_0,(\gamma_l)_{l\neq p},\delta)$ form a Kottwitz triple. Indeed using Theorem \ref{CM-lift}, we may assume that $x$ lifts to a special point $\tilde{x}\in\mathscr{S}_{\mathrm{K}}(G,X)(\Ok_K)$ such that the action of $T\subset \mbox{Aut}_{\Q}\mathcal{A}_x$ lifts to $\mbox{Aut}_{\Q}\mathcal{A}_{\tilde{x}}$.

Now $\gamma$ lifts to an element $\tilde{\gamma}\in T(\Q)\subset \mbox{Aut}_{\Q}\mathcal{A}_{\tilde{x}}$. If we let $\tilde{\gamma}$ act on the Betti cohomology of $\mathcal{A}_{\tilde{x}}$, then $\tilde{\gamma}$ fixes $s_{\alpha,B,\tilde{x}}$ since it fixes $s_{\alpha,\acute{e}t,\tilde{x}}$. We thus obtain an element $\gamma_0$ in $G(\Q)$ which is conjugate to $(\gamma_l)_{l\neq p}$ by the \'etale Betti comparison. Similarly $\gamma_0$ and $\gamma_p$ are conjugate over $G(\C)$ by the comparision isomorphisms between crystalline de Rham and Betti cohomology. 

By the positivity of the Rosatti involution, $T(\R)/w_h(\R^\times)$ is compact, and hence $\gamma_0$ is elliptic.

The following version of Tate's theorem, as well as the structural result on the group $I$, can be deduced in the same way as \cite[Cor. 2.3.2, 2.3.5]{Ki3}.

\begin{cor}\label{cor5}i) For every prime $l$ the natural maps $$I_{/k,\Q_l}\cong I_{/k}\otimes_{\Q}\Q_l\rightarrow I_{l/k}$$
	$$I_{\Q_l}=I\otimes_\Q\Q_l\rightarrow I_l$$ are isomorphisms.
	
	ii) Let $I_0$ denote the centralizer of $\gamma_0$. Then $I$ is the inner form of $I$ such that for each place $l$, $I_{\Q_l}\cong I_l$.
	\end{cor}

\bibliographystyle{amsplain}
\bibliography{bibfile}

\end{document}